\documentclass[a4paper,12pt]{article}
\usepackage[top=3cm, bottom=3 cm]{geometry}

\usepackage{amsmath}
\usepackage{amsthm}
\usepackage{amssymb}

\usepackage{mathtools}
\usepackage{tikz}
\usetikzlibrary{decorations.pathreplacing,calligraphy,positioning,math,decorations.markings,arrows,shapes,calc}
\usepackage{setspace}
\newtheorem{theorem}{Theorem}

\allowdisplaybreaks 
\usepackage{setspace}

\makeatletter
 
  \@addtoreset{equation}{section}
\makeatother

\makeatletter
\def\tbcaption{\def\@captype{table}\caption}
\makeatother

\newtheorem{corollary}{Corollary}

\newtheorem{remark}{Remark}

\DeclarePairedDelimiter\ceil{\lceil}{\rceil}
\DeclarePairedDelimiter\floor{\lfloor}{\rfloor}

\DeclarePairedDelimiterX{\cif}[1]{(}{)}{\delimsize(#1\delimsize)}

\begin{document}

\title{Metallic cubes}

\author{Tomislav Do\v{s}li\'c\\
    University of Zagreb Faculty of Civil Engineering,\\
    Zagreb, Croatia \\
         and \\
    Faculty of Information Studies, Novo Mesto, Slovenia\\
  \texttt{tomislav.doslic@grad.unizg.hr} \and
Luka Podrug\\
    University of Zagreb Faculty of Civil Engineering, \\
    Zagreb, Croatia\\
  \texttt{luka.podrug@grad.unizg.hr}}

\date{}
\maketitle

\begin{abstract}
We study a recursively defined two-parameter family of graphs which generalize
Fibonacci cubes and Pell graphs and determine their basic structural
and enumerative properties. In particular, we show that all of them are
induced subgraphs of hypercubes and present their canonical decompositions.
Further, we compute their metric invariants and establish some Hamiltonicity
properties.  We show that the new family inherits many useful properties of
Fibonacci cubes and hence could be interesting for potential applications. We
also compute the degree distribution, opening thus the way for computing many
degree-based topological invariants. Several possible directions of further
research are discussed in the concluding section.
\end{abstract}

\section{Introduction}\label{intr}

The hypercubes $Q_n$ are one of the best known and most researched graph
families. In spite of their simple definition - $Q_n$ it the $n$-the
Cartesian power of $K_2$ - they have rich internal structure and pose many
non-trivial research problems. In recent years and decades they have been
studied, among other reasons, for their potential uses as processor
interconnection networks of low diameter. This advantage is, however, often
offset by rigid restrictions on the number of nodes (always a power of two),
prompting thus interest in more flexible architectures along the same
general lines. Two variations, the Fibonacci and the Lucas cubes, turned
out to be promising, giving rise to a large number of papers studying their
various properties. As it could be guessed from their names, the number of
vertices in such structures grows slower, with the $n$-th power of the
Golden Ratio, hence as $\left ( \frac{1+\sqrt {5}}{2} \right )^n \approx
1.618^n$, offering more flexibility while preserving desirable properties of
hypercubes \cite{hsu1}. Hence the number of vertices satisfies the simplest
non-trivial three-term linear recurrence $v_n = v_{n-1} + v_{n-2}$ for
$n \geq 3$ with the initial conditions $v_0 = 2, v_2 = 3$.

Both hypercubes and Fibonacci cubes have natural representations as graphs
whose vertices consist of binary strings, subject to some additional 
restrictions in the case of Fibonacci cubes, two vertices being adjacent 
if and only if the Hamming distance between the corresponding strings is
equal to one. Fibonacci cubes have attracted much attention and spawned a
number of generalizations, most of them leading to graphs with the number 
of vertices counted by various types of higher-order Fibonacci numbers
\cite{hsu2}.
In a recent paper by Munarini \cite{Munarini}, the author considered a
generalization of different type, extending the alphabet over which the 
strings are constructed and imposing some restrictions on the new element
in the alphabet. (Previous generalizations were mostly concerned with
binary strings with various restrictions on adjacencies of ones.) It turned
out that the new graphs exhibit many similarities with Fibonacci cubes. In
particular, they share the property of being partial cubes, hence isometric
subgraphs of hypercubes. Also, the number of vertices in his graphs
satisfies a three-term linear recurrence with constant coefficients. More
precisely, it is given by the Pell numbers, the sequence of nonnegative
integers satisfying the three-term linear recurrence 
$P_n = 2 P_{n-1} + P_{n-2}$, for $n \geq 2$, with the initial conditions
$P_0 = 1, P_1 = 2$. Hence he named the new graphs the {\em Pell graphs}.

In this paper we take further the line of research started by Munarini in
\cite{Munarini} and ask whether graphs with desirable properties of
Fibonacci cubes and Pell graphs could be obtained by further extending the
alphabet underlying the strings serving as their vertices. We found that the
answer is positive. In the rest of this paper we describe one possible
generalization, the most natural in our opinion, a recursively defined
family of graphs which shares many structural properties with previously
considered partial cubes. Our graphs are parameterized by two nonnegative
integer parameters, one, $a$, denoting the alphabet size, the other, $n$,
denoting the step in the construction. As the number of vertices of the
graph obtained at the $n$-th step of a construction for a fixed value of
$a$ satisfies the three-term linear recurrence $v_n = a v_{n-1} + v_{n-2}$,
we call the resulting graphs the {\em metallic cubes}. We find the name
informative, reflecting both their partial cube nature and the asymptotic
behavior of the number of vertices, expressed via roots of the
characteristic equation $x^2 -a x -1 =0$, known as {\em metallic means}.

In the next section we give necessary definitions and references. Section 3
is concerned with basic structural properties of metallic cubes. The results
are then used in Section 4 to compute the number of edges and the degree 
distribution. Section 5 deals with metric properties, i.e., with diameter,
radius, center and periphery of metallic cubes. In section 6 we address and
partially answer their Hamiltonicity properties. The paper is concluded by
a short section summarizing our findings and indicating some possibilities
for further research.

\section{Definitions and illustrations}

The hypercube $Q_n$ is defined as follows: The set of vertices $V(Q_n)$
consists of all binary strings of length $n$ and two vertices are adjacent if
they differ in a single bit, i.e., if one vertex can be obtained from the
other by replacing $0$ with $1$ or vice versa, only once. The Fibonacci cubes
$\Gamma_n$ are special subgraphs of hypercubes where the vertices are the
Fibonacci strings $\mathcal{F}_n$, defined as binary strings of length $n$
without consecutive ones. The number of vertices in Fibonacci cubes is 
counted by the Fibonacci numbers, $V(\Gamma_n)=F_{n+2}$. The Pell graphs
$\Pi_n^{a}$, introduced by Munarini \cite{Munarini} in 2019, are graphs whose
vertices are strings of length $n$ over the alphabet
$\left\lbrace 0,1,2\right\rbrace$ with property that $2$ comes only in
blocks of even length. Two vertices are adjacent if one can be obtained from
the  other by replacing $0$ with $1$ or by replacing block $11$ with block
$22$. The vertices in this family are counted by the Pell numbers. The
Munarini's paper inspired us to try to generalize this idea by obtaining
family of graphs whose numbers of vertices satisfy recursive relation
$s_n=a\cdot s_{n-1}+s_{n-2}$ for arbitrary non-negative integer $a$ while
preserving main properties of Pell graphs. 

Let $a$ be a non-negative integer and let $\mathcal{S}_{a}$ denote the free
monoid containing $a+1$ generators
$\left\lbrace 0,1,2,\dots,a-1,0a\right\rbrace$. By a \textit{string} we mean
an element of monoid $\mathcal{S}_{a}$, i.e., a word from alphabet
$\left\lbrace 0,1,2,\dots,a-1,a\right\rbrace$ with property that letter
$a$ can only appear in block $0a$. Other letters can appear arbitrary. For
strings $\alpha=x_1\cdots x_n$ and $\beta=y_1\cdots y_m$ we define their
concatenation in the usual way, $\alpha\beta=x_1\cdots x_n y_1\cdots y_m$. 

If $\mathcal{S}^{a}_n$ denotes the set of all elements from the monoid
$\mathcal{T}_{a}$ of length $n$, one can easily obtain recursive relation
for the cardinal number $s^{a}_n=\left|\mathcal{S}^{a}_n\right|$. A string
of length $n$ can end with any of letters $0,1,\dots,a-1$ and the rest of
the string can be formed in $s_{n-1}^a$ ways. If a string ends with the 
letter $a$, that necessarily means that there is at least one zero in front
of that $a$, and the rest of the string can be formed in $s_{n-2}^a$ ways.
Hence, $s_n^a$ satisfies the recursive relation \begin{equation}
s^{a}_{n}=as^{a}_{n-1}+s^{a}_{n-2} \label{Metallic_recursion}
\end{equation}
with initial values $s^{a}_0=1$ and $s^{a}_1=a$. 

Now let $\Pi_n^{a}$ denote a graph whose vertices are elements of the set
$\mathcal{S}^a_n$, i.e., $ V\left(\Pi_n^{a}\right)=\mathcal{S}^{a}_n$ and
for any $v_1,v_2\in V\left(\Pi_n^{a,b}\right)$ we have
$\left(v_1, v_2\right)\in E\left(\Pi_n^{a,b}\right)$  if and only if one
vertex can be obtained from the other by replacing a single letter $k$ with
$k+1$ for $0\leq k\leq a-1$. An alternative definition of adjacency can be
given via modified Hamming distance. For $\alpha=\alpha_1\cdots\alpha_n$
and $\beta=\beta_1\cdots\beta_n$ we define
$$\overline{h}(\alpha,\beta)=\sum\limits_{k=1}^n|\alpha_k-\beta_k|.$$
Then $\alpha$ and $\beta$ are adjacent if and only if
$\overline{h}(\alpha,\beta)=1$. As an example, Figure \ref{fig:Pi_3^{3}}
shows graphs $\Pi^a_n$ for $a=3$ and $n = 1,2,3$.

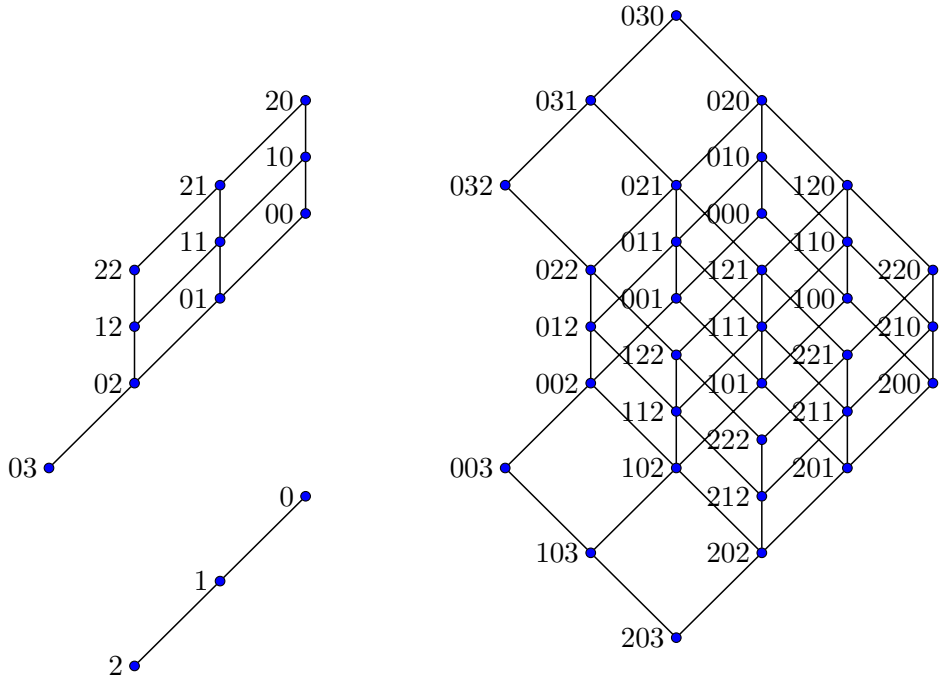
\begin{figure}[h!] \centering
\begin{tikzpicture}[scale=0.75]
\tikzmath{\x1 = 0.35; \y1 =-0.05; \z1=180; \w1=0.2; \xs=-8; \ys=0; \yss=-5;
\x2 = \x1 + 1; \y2 =\y1 +3; } 
\small
\node [label={[label distance=\y1 cm]\z1: $000$},circle,fill=blue,draw=black,scale=\x1](A1) at (0,4) {};
\node [label={[label distance=\y1 cm]\z1: $010$},circle,fill=blue,draw=black,scale=\x1](A2) at (0,5) {};
\node [label={[label distance=\y1 cm]\z1: $020$},circle,fill=blue,draw=black,scale=\x1](A3) at (0,6) {};
\node [label={[label distance=\y1 cm]\z1: $001$},circle,fill=blue,draw=black,scale=\x1](A4) at (-1.5,2.5) {};
\node [label={[label distance=\y1 cm]\z1: $011$},circle,fill=blue,draw=black,scale=\x1](A5) at (-1.5,3.5) {};
\node [label={[label distance=\y1 cm]\z1: $021$},circle,fill=blue,draw=black,scale=\x1](A6) at (-1.5,4.5) {};
\node [label={[label distance=\y1 cm]\z1: $002$},circle,fill=blue,draw=black,scale=\x1](A7) at (-3,1) {};
\node [label={[label distance=\y1 cm]\z1: $012$},circle,fill=blue,draw=black,scale=\x1](A8) at (-3,2) {};
\node [label={[label distance=\y1 cm]\z1: $022$},circle,fill=blue,draw=black,scale=\x1](A9) at (-3,3) {};
\node [label={[label distance=\y1 cm]\z1: $100$},circle,fill=blue,draw=black,scale=\x1](A10) at (1.5,2.5) {};
\node [label={[label distance=\y1 cm]\z1: $110$},circle,fill=blue,draw=black,scale=\x1](A11) at (1.5,3.5) {};
\node [label={[label distance=\y1 cm]\z1: $120$},circle,fill=blue,draw=black,scale=\x1](A12) at (1.5,4.5) {};
\node [label={[label distance=\y1 cm]\z1: $101$},circle,fill=blue,draw=black,scale=\x1](A13) at (0,1) {};
\node [label={[label distance=\y1 cm]\z1: $111$},circle,fill=blue,draw=black,scale=\x1](A14) at (0,2) {};
\node [label={[label distance=\y1 cm]\z1: $121$},circle,fill=blue,draw=black,scale=\x1](A15) at (0,3) {};
\node [label={[label distance=\y1 cm]\z1: $102$},circle,fill=blue,draw=black,scale=\x1](A16) at (-1.5,-0.5) {};
\node [label={[label distance=\y1 cm]\z1: $112$},circle,fill=blue,draw=black,scale=\x1](A17) at (-1.5,0.5) {};
\node [label={[label distance=\y1 cm]\z1: $122$},circle,fill=blue,draw=black,scale=\x1](A18) at (-1.5,1.5) {};
\node [label={[label distance=\y1 cm]\z1: $200$},circle,fill=blue,draw=black,scale=\x1](A19) at (3,1) {};
\node [label={[label distance=\y1 cm]\z1: $210$},circle,fill=blue,draw=black,scale=\x1](A20) at (3,2) {};
\node [label={[label distance=\y1 cm]\z1: $220$},circle,fill=blue,draw=black,scale=\x1](A21) at (3,3) {};
\node [label={[label distance=\y1 cm]\z1: $201$},circle,fill=blue,draw=black,scale=\x1](A22) at (1.5,-0.5) {};
\node [label={[label distance=\y1 cm]\z1: $211$},circle,fill=blue,draw=black,scale=\x1](A23) at (1.5,0.5) {};
\node [label={[label distance=\y1 cm]\z1: $221$},circle,fill=blue,draw=black,scale=\x1](A24) at (1.5,1.5) {};
\node [label={[label distance=\y1 cm]\z1: $202$},circle,fill=blue,draw=black,scale=\x1](A25) at (0,-2) {};
\node [label={[label distance=\y1 cm]\z1: $212$},circle,fill=blue,draw=black,scale=\x1](A26) at (0,-1) {};
\node [label={[label distance=\y1 cm]\z1: $222$},circle,fill=blue,draw=black,scale=\x1](A27) at (0,0) {};
\node [label={[label distance=\y1 cm]\z1: $003$},circle,fill=blue,draw=black,scale=\x1](A28) at (-4.5,-0.5) {};
\node [label={[label distance=\y1 cm]\z1: $103$},circle,fill=blue,draw=black,scale=\x1](A29) at (-3,-2) {};
\node [label={[label distance=\y1 cm]\z1: $203$},circle,fill=blue,draw=black,scale=\x1](A30) at (-1.5,-3.5) {};
\node [label={[label distance=\y1 cm]\z1: $030$},circle,fill=blue,draw=black,scale=\x1](A31) at (-1.5,7.5) {};
\node [label={[label distance=\y1 cm]\z1: $031$},circle,fill=blue,draw=black,scale=\x1](A32) at (-3,6) {};
\node [label={[label distance=\y1 cm]\z1: $032$},circle,fill=blue,draw=black,scale=\x1](A33) at (-4.5,4.5) {};

\draw [line width=\w1 mm] (A1)--(A2)--(A3) (A4)--(A5)--(A6) (A7)--(A8)--(A9) (A10)--(A11)--(A12) (A13)--(A14)--(A15) (A16)--(A17)--(A18) (A19)--(A20)--(A21) (A22)--(A23)--(A24) (A25)--(A26)--(A27) (A1)--(A4)--(A7)--(A28) (A2)--(A5)--(A8) (A3)--(A6)--(A9) (A10)--(A13)--(A16)--(A29) (A11)--(A14)--(A17) (A12)--(A15)--(A18) (A19)--(A22)--(A25)--(A30) (A20)--(A23)--(A26) (A21)--(A24)--(A27) (A1)--(A10)--(A19) (A2)--(A11)--(A20) (A31)--(A3)--(A12)--(A21) (A4)--(A13)--(A22) (A5)--(A14)--(A23) (A32)--(A6)--(A15)--(A24) (A7)--(A16)--(A25) (A8)--(A17)--(A26) (A33)--(A9)--(A18)--(A27)  (A28)--(A29)--(A30) (A31)--(A32)--(A33);

\node [label={[label distance=\y1 cm]\z1: $00$},circle,fill=blue,draw=black,scale=\x1](B1) at (0+\xs,4+\ys) {};
\node [label={[label distance=\y1 cm]\z1: $10$},circle,fill=blue,draw=black,scale=\x1](B2) at (0+\xs,5+\ys) {};
\node [label={[label distance=\y1 cm]\z1: $20$},circle,fill=blue,draw=black,scale=\x1](B3) at (0+\xs,6+\ys) {};
\node [label={[label distance=\y1 cm]\z1: $01$},circle,fill=blue,draw=black,scale=\x1](B4) at (-1.5+\xs,2.5+\ys) {};
\node [label={[label distance=\y1 cm]\z1: $11$},circle,fill=blue,draw=black,scale=\x1](B5) at (-1.5+\xs,3.5+\ys) {};
\node [label={[label distance=\y1 cm]\z1: $21$},circle,fill=blue,draw=black,scale=\x1](B6) at (-1.5+\xs,4.5+\ys) {};
\node [label={[label distance=\y1 cm]\z1: $02$},circle,fill=blue,draw=black,scale=\x1](B7) at (-3+\xs,1+\ys) {};
\node [label={[label distance=\y1 cm]\z1: $12$},circle,fill=blue,draw=black,scale=\x1](B8) at (-3+\xs,2+\ys) {};
\node [label={[label distance=\y1 cm]\z1: $22$},circle,fill=blue,draw=black,scale=\x1](B9) at (-3+\xs,3+\ys) {};
\node [label={[label distance=\y1 cm]\z1: $03$},circle,fill=blue,draw=black,scale=\x1](B28) at (-4.5+\xs,-0.5+\ys) {};

\draw [line width=\w1 mm] (B1)--(B2)--(B3) (B4)--(B5)--(B6) (B7)--(B8)--(B9)  (B1)--(B4)--(B7)--(B28) (B2)--(B5)--(B8) (B3)--(B6)--(B9); 

\node [label={[label distance=\y1 cm]\z1: $0$},circle,fill=blue,draw=black,scale=\x1](C1) at (0+\xs,4+\yss) {};
\node [label={[label distance=\y1 cm]\z1: $1$},circle,fill=blue,draw=black,scale=\x1](C4) at (-1.5+\xs,2.5+\yss) {};
\node [label={[label distance=\y1 cm]\z1: $2$},circle,fill=blue,draw=black,scale=\x1](C7) at (-3+\xs,1+\yss) {};

\draw [line width=\w1 mm] (C1)--(C4)--(C7); 

\end{tikzpicture} 
\caption{Graph $\Pi_n^{3}$ for $n=1$ (lower left), $2$ (upper left), $3$ (right).} \label{fig:Pi_3^{3}}
\end{figure}

From the definition of family $\Pi^a_n$ it is immediately clear that
$\Pi^1_n=\Gamma_{n-1}$ but $\Pi^2_n$ are not isomorphic to Pell graphs.
One possible generalization could be obtained using alphabet
$\left\lbrace 0,1,\dots,a\right\rbrace$ where two vertices are adjacent if
one can be obtained from the other by replacing single letter $i$ with $i+1$
for $\leq i\leq a-2$ or by replacing a block $(a-1)(a-1)$ with $aa$. This
family of graphs would be a straightforward generalization of Pell graphs as 
defined by Munarini, and would inherit many of their properties. Instead of
pursuing this approach we opted for a generalization better suited to higher
values of $a$ yielding along the way an alternative definition of Pell graphs
and offering numerous possibilities for comparisons. The difference is 
illustrated in Figure \ref{fig:comparison}, where it can be seen that,
although similar, the graphs obtained under two generalizations are not
isomorphic, for vertex $111$ of Pell graph shown in the left panel of
Figure \ref{fig:comparison} has degree $5$, while the maximum degree in
$\Pi^2_3$ is $4$. 

In Table \ref{table:number of vertices} we show the number of vertices for
few initial values of $a$ and $n$. The Pell number appear as the second row,
while the (shifted) Fibonacci numbers appear in the first one.
It can be observed that the rows of
Table \ref{table:number of vertices} grow exponentially, as the $n$-th
powers of the corresponding metallic means, while the elements in columns
grow polynomially. In fact, since the number of vertices of $\Pi^a_n$ is
equal to $$|V(\Pi^a_n)|=\sum\limits_{k\geq 0} \binom{n-k}{k}a^{n-2k},$$
for a fixed $n$, we obtain that growth of the number of vertices is polynomial
with degree $n$.

\begin{table}\centering$\begin{array}{r|cccccccc}
_a \backslash ^n & 1 & 2 & 3 & 4 & 5 & 6 & 7 & 8\\
\hline
1 & 1 & 2 & 3 & 5 & 8 & 13 & 21 & 34\\
2 & 2 & 5 & 12& 29& 70& 169& 408&985\\
3 & 3 & 10 & 33& 109& 360& 1189& 3927&12970\\
4 & 4 & 17 & 72& 305& 1292& 5473& 23184&98209\\
5 & 5 & 26 & 135& 701& 3640& 18901& 98145&509626\\
6 & 6 & 37 & 228& 1405& 8658& 53353& 328776&2026009
\end{array}$\label{table:number of vertices}\caption{Number of vertices in $\Pi^a_n$.}
\end{table}

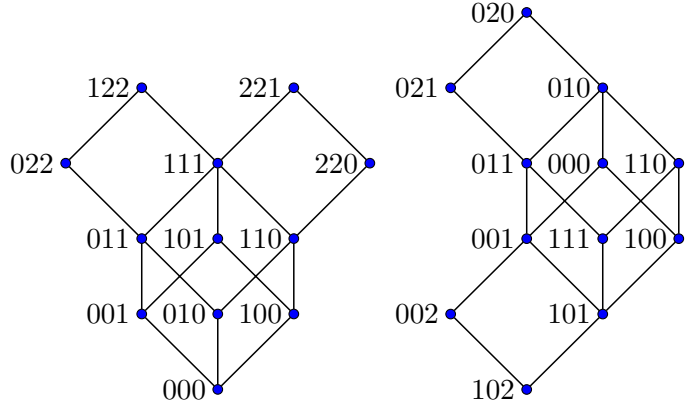
\begin{figure}\centering
\begin{tikzpicture}[scale=1]
\tikzmath{\x1 = 0.35; \y1 =-0.05; \z1=180; \w1=0.2; \xs=-8; \ys=0; \yss=-5;
\x2 = \x1 + 1; \y2 =\y1 +3; } 
\small
\node [label={[label distance=\y1 cm]\z1: $000$},circle,fill=blue,draw=black,scale=\x1](A1) at (1,0) {};
\node [label={[label distance=\y1 cm]\z1: $001$},circle,fill=blue,draw=black,scale=\x1](A2) at (0,1) {};
\node [label={[label distance=\y1 cm]\z1: $101$},circle,fill=blue,draw=black,scale=\x1](A3) at (1,2) {};
\node [label={[label distance=\y1 cm]\z1: $100$},circle,fill=blue,draw=black,scale=\x1](A4) at (2,1) {};
\node [label={[label distance=\y1 cm]\z1: $010$},circle,fill=blue,draw=black,scale=\x1](A5) at (1,1) {};
\node [label={[label distance=\y1 cm]\z1: $011$},circle,fill=blue,draw=black,scale=\x1](A6) at (0,2) {};
\node [label={[label distance=\y1 cm]\z1: $111$},circle,fill=blue,draw=black,scale=\x1](A7) at (1,3) {};
\node [label={[label distance=\y1 cm]\z1: $110$},circle,fill=blue,draw=black,scale=\x1](A8) at (2,2) {};
\node [label={[label distance=\y1 cm]\z1: $022$},circle,fill=blue,draw=black,scale=\x1](A9) at (-1,3) {};
\node [label={[label distance=\y1 cm]\z1: $122$},circle,fill=blue,draw=black,scale=\x1](A10) at (0,4) {};
\node [label={[label distance=\y1 cm]\z1: $221$},circle,fill=blue,draw=black,scale=\x1](A11) at (2,4) {};
\node [label={[label distance=\y1 cm]\z1: $220$},circle,fill=blue,draw=black,scale=\x1](A12) at (3,3) {};

\draw [line width=\w1 mm] (A1)--(A2)--(A3)--(A4)--(A1)--(A5)--(A6)--(A7)--(A8)--(A5) (A2)--(A6) (A3)--(A7) (A4)--(A8)  (A6)--(A9)--(A10)--(A7)--(A11)--(A12)--(A8);

\end{tikzpicture} \begin{tikzpicture}[scale=1]
\tikzmath{\x1 = 0.35; \y1 =-0.05; \z1=180; \w1=0.2; \xs=-8; \ys=0; \yss=-5;
\x2 = \x1 + 1; \y2 =\y1 +3; } 
\small
\node [label={[label distance=\y1 cm]\z1: $101$},circle,fill=blue,draw=black,scale=\x1](A1) at (1,0) {};
\node [label={[label distance=\y1 cm]\z1: $111$},circle,fill=blue,draw=black,scale=\x1](A2) at (1,1) {};
\node [label={[label distance=\y1 cm]\z1: $001$},circle,fill=blue,draw=black,scale=\x1](A3) at (0,1) {};
\node [label={[label distance=\y1 cm]\z1: $011$},circle,fill=blue,draw=black,scale=\x1](A4) at (0,2) {};
\node [label={[label distance=\y1 cm]\z1: $002$},circle,fill=blue,draw=black,scale=\x1](A5) at (-1,0) {};
\node [label={[label distance=\y1 cm]\z1: $100$},circle,fill=blue,draw=black,scale=\x1](A6) at (2,1) {};
\node [label={[label distance=\y1 cm]\z1: $110$},circle,fill=blue,draw=black,scale=\x1](A7) at (2,2) {};
\node [label={[label distance=\y1 cm]\z1: $000$},circle,fill=blue,draw=black,scale=\x1](A8) at (1,2) {};
\node [label={[label distance=\y1 cm]\z1: $010$},circle,fill=blue,draw=black,scale=\x1](A9) at (1,3) {};
\node [label={[label distance=\y1 cm]\z1: $102$},circle,fill=blue,draw=black,scale=\x1](A10) at (0,-1) {};
\node [label={[label distance=\y1 cm]\z1: $020$},circle,fill=blue,draw=black,scale=\x1](A11) at (0,4) {};
\node [label={[label distance=\y1 cm]\z1: $021$},circle,fill=blue,draw=black,scale=\x1](A12) at (-1,3) {};

\draw [line width=\w1 mm] (A4)--(A12)--(A11)--(A9)--(A8)--(A6)--(A7)--(A2)--(A1)--(A10)--(A5)--(A3)--(A4) (A4)--(A9)--(A7) (A2)--(A4) (A6)--(A1)--(A3)--(A8); 
\end{tikzpicture}  
\caption{The Pell graph $\Pi_3$ (left) and the corresponding metallic
cube $\Pi^2_3$ (right).} \label{fig:comparison}
\end{figure}   

\section{Basic structural properties}
\subsection{Canonical decompositions and bipartivity}
It is well known that Fibonacci cubes and Pell graphs admit recursive
decomposition \cite{Munarini}. In this section we show that such decompositions
naturally extend to metallic cubes.

We start with the observation that the set of vertices $\mathcal{S}_n^a$
can be divided into disjoint sets based of the starting letter. One set 
contains the vertices starting with block $0a$, and the remaining $a$ sets
contain the vertices that start with $0, 1,\dots,a-2$ and $a-1$, respectively.
Assuming $\alpha\in \mathcal{S}^a_{n-1}$ and $\beta\in \mathcal{S}^a_{n-2}$,
the vertices $0\alpha$, $1\alpha$, $\dots$, $(a-1)\alpha$ generate $a$ copies
of a graph $\Pi^a_{n-1}$ and vertices $0a\beta$ generate one copy of
$\Pi^a_{n-2}$. That brings us to our first result.

\begin{theorem}\label{candec} Metallic cube $\Pi^a_n$ admits the decomposition
\begin{equation*}
 \Pi^a_n=\Pi^a_{n-1}\oplus\cdots\oplus\Pi^a_{n-1}\oplus\Pi^a_{n-2}
\end{equation*} 
in which the first factor $\Pi^a_{n-1}$ is repeated $a$ times.
\end{theorem}
The decomposition of Theorem \ref{candec} is called the {\em canonical
decomposition} of $\Pi^a_n$.

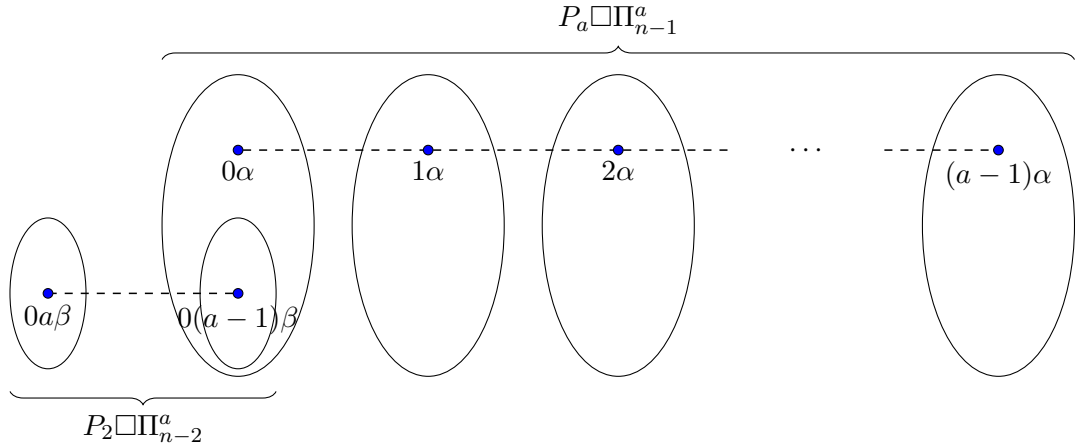
\begin{figure}\centering
\newcommand{\boundellipse}[3]
{(#1) ellipse (#2 and #3)
}

\begin{tikzpicture}[scale=1]\tikzmath{\x1 = 0.35; \y1 =-0.05; \z1=-90; \w1=0.2; \xs=-8; \ys=0; \yss=-5;
\x2 = \x1 + 1; \y2 =\y1 +3; } \small
\draw \boundellipse{-2.5,-0.9}{0.5}{1};
\draw \boundellipse{0,-0.9}{0.5}{1};
\draw \boundellipse{0,0}{1}{2};
\draw \boundellipse{2.5,0}{1}{2};
\draw \boundellipse{5,0}{1}{2};
\draw \boundellipse{10,0}{1}{2};
\node [label={[label distance=\y1 cm]\z1: $0\alpha$},circle,fill=blue,draw=black,scale=\x1](A1) at (0,1) {};
\node [label={[label distance=\y1 cm]\z1: $1\alpha$},circle,fill=blue,draw=black,scale=\x1](A2) at (2.5,1) {};
\node [label={[label distance=\y1 cm]\z1: $2\alpha$},circle,fill=blue,draw=black,scale=\x1](A3) at (5,1) {};
\node [label={[label distance=\y1 cm]\z1: $(a-1)\alpha$},circle,fill=blue,draw=black,scale=\x1](A4) at (10,1) {};
\node [label={[label distance=-0.35 cm]\z1: $\cdots$}](A5) at (7.5,1) {};

\draw [line width=\w1 mm,dashed] (A1)--(A2)--(A3)--(6.5,1)  (8.5,1)--(A4); 

\node [label={[label distance=\y1 cm]\z1: $0(a-1)\beta$},circle,fill=blue,draw=black,scale=\x1](A6) at (0,-0.9) {};
\node [label={[label distance=\y1 cm]\z1: $0a\beta$},circle,fill=blue,draw=black,scale=\x1](A7) at (-2.5,-0.9) {};
\draw [line width=\w1 mm,dashed] (A6)--(A7); 

\draw [black, left=3pt,
    decorate, 
    decoration = {brace,amplitude=5pt}](-1,2.2)--(11,2.2) node[pos=0.5,above=5pt,black]{$P_a\square\Pi^a_{n-1}$};
    
\draw [black, left=3pt,
    decorate, 
    decoration = {brace,amplitude=5pt}](0.5,-2.2)--(-3,-2.2) node[pos=0.5,below=5pt,black]{$P_2\square\Pi^a_{n-2}$};
\end{tikzpicture}\caption{Canonical decomposition $\Pi^a_n=\Pi^a_{n-1}\oplus\cdots\oplus\Pi^a_{n-1}\oplus\Pi^a_{n-2}$.}\label{fig_can_dec_diagram}
\end{figure}

In Figure \ref{fig_can_dec_diagram} we show a schematic representation of
the described canonical decomposition. As an example,
Figure \ref{fig:canon dec of Pi_3^{3}} shows the canonical decomposition of
graph $\Pi_3^{3}$ into three copies of graph $\Pi_2^{3}$ and one copy of
graph $\Pi_1^{3}$. 

It is worth noting that
$\Pi^a_{n-1}\oplus\cdots\oplus\Pi^a_{n-1}=P_{a}\square \Pi^a_{n-1}$,
where $P_{a}$ is the path graph on $a$ vertices.

A map $\chi:V(G)\to\left\lbrace 0,1 \right\rbrace$ is a {\em proper
$2$-coloring} of a graph $G$ if $\chi(v_1)\neq \chi(v_2)$ for every two
adjacent vertices $v_1,v_2\in V(G)$. A graph $G$ is {\em bipartite} if its
set of vertices $V(G)$ can be decomposed into two disjoint subsets such that
no two vertices of the same subset share an edge. Equivalently, a graph $G$
is bipartite if it admits a proper $2$-coloring.

\begin{theorem}
All metallic cubes are bipartite.
\end{theorem} \begin{proof}
Proof is by induction on $n$. We start by observing that $\Pi^a_1$ is
isomorphic to the path graph on $a$ vertices and, thus, bipartite. Since
$\Pi^a_2$ is an $a\times a$ grid with one additional vertex $0a$, it is an
easy exercise to see that it also admits a proper $2$-coloring. Now we
suppose that $\Pi^a_k$ is bipartite for every $k<n$. By the inductive
hypothesis, it admits a proper $2$-coloring
$\chi:\Pi^a_{n-1}\to \left\lbrace 0,1 \right\rbrace$. Consider the map
$\chi':\Pi^a_{n-1}\to\left\lbrace 0,1 \right\rbrace$, where
$\chi'(v)=1-\chi(v)$. Then the map $\chi'$ is a complementary proper
$2$-coloring for graph $\Pi^a_{n-1}$. Since
$\Pi^a_n=\Pi^a_{n-1}\oplus\cdots\oplus\Pi^a_{n-1}\oplus\Pi^a_{n-2}$, we
can choose a coloring $\chi$ for subgraph $\Pi^a_{n-1}$ in $\Pi^a_n$ if the
subgraph is induced by vertices starting with even $k$, and $\chi'$ if 
$k$ is odd. Finally, for the one copy of $\Pi^a_{n-2}$ in the canonical
decomposition, we can choose $\chi'$ restricted to $\Pi^a_{n-2}$. Thus we
obtained a proper $2$-coloring of graph $\Pi^a_{n}$.
\end{proof}
Figure \ref{fig:canon dec of Pi_3^{3}} shows a proper $2$-coloring of
graph $\Pi^3_3$.

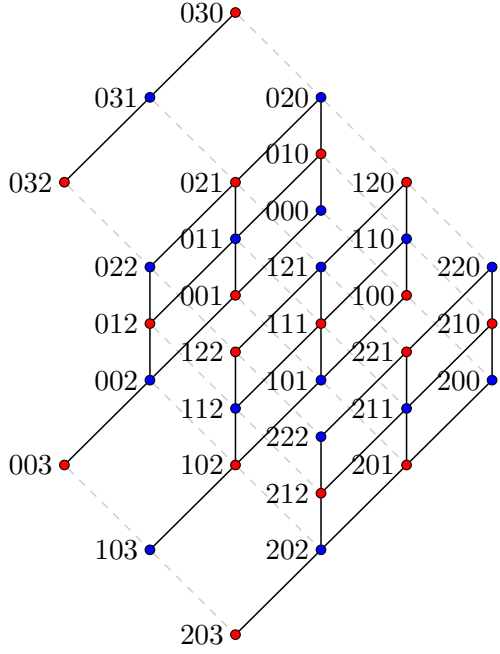
\begin{figure}[h!] \centering
\begin{tikzpicture}[scale=0.75]
\tikzmath{\x1 = 0.35; \y1 =-0.05; \z1=180; \w1=0.2; \xs=-8; \ys=0; \yss=-5;
\x2 = \x1 + 1; \y2 =\y1 +3; } 
\small
\node [label={[label distance=\y1 cm]\z1: $000$},circle,fill=blue,draw=black,scale=\x1](A1) at (0,4) {};
\node [label={[label distance=\y1 cm]\z1: $010$},circle,fill=red,draw=black,scale=\x1](A2) at (0,5) {};
\node [label={[label distance=\y1 cm]\z1: $020$},circle,fill=blue,draw=black,scale=\x1](A3) at (0,6) {};
\node [label={[label distance=\y1 cm]\z1: $001$},circle,fill=red,draw=black,scale=\x1](A4) at (-1.5,2.5) {};
\node [label={[label distance=\y1 cm]\z1: $011$},circle,fill=blue,draw=black,scale=\x1](A5) at (-1.5,3.5) {};
\node [label={[label distance=\y1 cm]\z1: $021$},circle,fill=red,draw=black,scale=\x1](A6) at (-1.5,4.5) {};
\node [label={[label distance=\y1 cm]\z1: $002$},circle,fill=blue,draw=black,scale=\x1](A7) at (-3,1) {};
\node [label={[label distance=\y1 cm]\z1: $012$},circle,fill=red,draw=black,scale=\x1](A8) at (-3,2) {};
\node [label={[label distance=\y1 cm]\z1: $022$},circle,fill=blue,draw=black,scale=\x1](A9) at (-3,3) {};
\node [label={[label distance=\y1 cm]\z1: $100$},circle,fill=red,draw=black,scale=\x1](A10) at (1.5,2.5) {};
\node [label={[label distance=\y1 cm]\z1: $110$},circle,fill=blue,draw=black,scale=\x1](A11) at (1.5,3.5) {};
\node [label={[label distance=\y1 cm]\z1: $120$},circle,fill=red,draw=black,scale=\x1](A12) at (1.5,4.5) {};
\node [label={[label distance=\y1 cm]\z1: $101$},circle,fill=blue,draw=black,scale=\x1](A13) at (0,1) {};
\node [label={[label distance=\y1 cm]\z1: $111$},circle,fill=red,draw=black,scale=\x1](A14) at (0,2) {};
\node [label={[label distance=\y1 cm]\z1: $121$},circle,fill=blue,draw=black,scale=\x1](A15) at (0,3) {};
\node [label={[label distance=\y1 cm]\z1: $102$},circle,fill=red,draw=black,scale=\x1](A16) at (-1.5,-0.5) {};
\node [label={[label distance=\y1 cm]\z1: $112$},circle,fill=blue,draw=black,scale=\x1](A17) at (-1.5,0.5) {};
\node [label={[label distance=\y1 cm]\z1: $122$},circle,fill=red,draw=black,scale=\x1](A18) at (-1.5,1.5) {};
\node [label={[label distance=\y1 cm]\z1: $200$},circle,fill=blue,draw=black,scale=\x1](A19) at (3,1) {};
\node [label={[label distance=\y1 cm]\z1: $210$},circle,fill=red,draw=black,scale=\x1](A20) at (3,2) {};
\node [label={[label distance=\y1 cm]\z1: $220$},circle,fill=blue,draw=black,scale=\x1](A21) at (3,3) {};
\node [label={[label distance=\y1 cm]\z1: $201$},circle,fill=red,draw=black,scale=\x1](A22) at (1.5,-0.5) {};
\node [label={[label distance=\y1 cm]\z1: $211$},circle,fill=blue,draw=black,scale=\x1](A23) at (1.5,0.5) {};
\node [label={[label distance=\y1 cm]\z1: $221$},circle,fill=red,draw=black,scale=\x1](A24) at (1.5,1.5) {};
\node [label={[label distance=\y1 cm]\z1: $202$},circle,fill=blue,draw=black,scale=\x1](A25) at (0,-2) {};
\node [label={[label distance=\y1 cm]\z1: $212$},circle,fill=red,draw=black,scale=\x1](A26) at (0,-1) {};
\node [label={[label distance=\y1 cm]\z1: $222$},circle,fill=blue,draw=black,scale=\x1](A27) at (0,0) {};
\node [label={[label distance=\y1 cm]\z1: $003$},circle,fill=red,draw=black,scale=\x1](A28) at (-4.5,-0.5) {};
\node [label={[label distance=\y1 cm]\z1: $103$},circle,fill=blue,draw=black,scale=\x1](A29) at (-3,-2) {};
\node [label={[label distance=\y1 cm]\z1: $203$},circle,fill=red,draw=black,scale=\x1](A30) at (-1.5,-3.5) {};
\node [label={[label distance=\y1 cm]\z1: $030$},circle,fill=red,draw=black,scale=\x1](A31) at (-1.5,7.5) {};
\node [label={[label distance=\y1 cm]\z1: $031$},circle,fill=blue,draw=black,scale=\x1](A32) at (-3,6) {};
\node [label={[label distance=\y1 cm]\z1: $032$},circle,fill=red,draw=black,scale=\x1](A33) at (-4.5,4.5) {};

\draw [line width=\w1 mm] (A1)--(A2)--(A3) (A4)--(A5)--(A6) (A7)--(A8)--(A9) (A10)--(A11)--(A12) (A13)--(A14)--(A15) (A16)--(A17)--(A18) (A19)--(A20)--(A21) (A22)--(A23)--(A24) (A25)--(A26)--(A27) (A1)--(A4)--(A7)--(A28) (A2)--(A5)--(A8) (A3)--(A6)--(A9) (A10)--(A13)--(A16)--(A29) (A11)--(A14)--(A17) (A12)--(A15)--(A18) (A19)--(A22)--(A25)--(A30) (A20)--(A23)--(A26) (A21)--(A24)--(A27) (A31)--(A32)--(A33); 

\draw [line width=\w1 mm,dashed, opacity=0.2] (A1)--(A10)--(A19) (A2)--(A11)--(A20) (A31)--(A3)--(A12)--(A21) (A4)--(A13)--(A22) (A5)--(A14)--(A23) (A32)--(A6)--(A15)--(A24) (A7)--(A16)--(A25) (A8)--(A17)--(A26) (A33)--(A9)--(A18)--(A27)  (A28)--(A29)--(A30); 

\end{tikzpicture}
\caption{The canonical decomposition and a proper $2$-coloring of $\Pi_3^{3}=\Pi_3^{2}\oplus\Pi_3^{2}\oplus\Pi_3^{2}\oplus\Pi_3^{1}$.} \label{fig:canon dec of Pi_3^{3}}
\end{figure}

As we have mentioned above, the sequence $s^a_n$ defined by recurrence
(\ref{Metallic_recursion}) satisfies the ell-known identity
\begin{align}\label{eq:metallic identity}
s^a_n=\sum\limits_{k\geq 0} \binom{n-k}{k}a^{n-2k}.
\end{align}
We now present another decomposition of metallic cubes, providing a 
combinatorial representation of identity (\ref{eq:metallic identity}).
To this end, recall that $P_a$ denotes a path graph with $a$ vertices,
and $P_a^k$ denotes Cartesian product of path $P_a$ with itself $k$ times,
that is $P_{a}\square P_{a}\cdots \square P_{a}$. Also note that
$|V(P_a^k)|=a^k$. Graphs $P_a^k$ are called grids or lattices.
The following theorem illustrates the combinatorial meaning of the Fibonacci 
polynomials of identity (\ref{eq:metallic identity}).

\begin{theorem}
Graph $\Pi^a_n$ can be decomposed into $F_{n+1}$ grids, where $F_n$ denotes $n$-th Fibonacci number. \label{tm:decomposition II} \end{theorem}
\begin{proof}
Consider the subset of strings $S\subset\mathcal{S}^a_n$ where each string
$\alpha\in S$ has $k$ blocks $0a$ in the same position. Then $\alpha$ has
$n-2k$ letters that are not part of a $0a$ block. There are $a^{n-2k}$ such
strings and the subset $S\subset\mathcal{S}^a_n$ induces a subgraph of
$\Pi^a_n$ isomorphic to grid $P_{a}^{n-2k}$ with $a^{n-2k}$ vertices. Also
note that different locations of blocks $0a$ produce different strings. Hence,
for different alignments, the induced grids are vertex-disjoint. To finish
our proof, we just need to determine the number of strings with exactly $k$
blocks $0a$. We can identify block as a single letter, and reduce our problem
to a subset of $n-k$ positions, where we need to choose $k$ positions for
blocks. Thus, we have $\binom{n-k}{k}$ different alignments for $k$ blocks
$0a$, and every alignment induces disjoint subgraphs $P_a^{n-2k}$. We
obtained a decomposition
$$\Pi^a_n=\bigoplus_{k\geq 0}\binom{n-k}{k}P^{n-2k}_a.$$
Then the number of grids is $\sum\limits_{k\geq 0}\binom{n-k}{k}=F_{n+1}$,
and this completes our proof.
\end{proof}

To obtain our last result in this subsection, we consider the map
$\rho:V(\Pi^a_n)\to\mathcal{F}_n$ defined on alphabet
$\left\lbrace0,1,\dots,a\right\rbrace$ and extended to $V(\Pi^a_n)$ by
concatenation, as follows
$$\rho(\alpha)=\begin{cases} 0, & 0\leq \alpha \leq a-1,\\
		1, & \alpha=1.\end{cases} $$ 
For a Fibonacci string $w$, let $\rho^{-1}(w)$ denote the subgraph of
$\Pi^a_n$ induced by vertices
$\left\lbrace v\in V(\Pi^a_n)| \rho(v)=w\right\rbrace$. 

\begin{theorem} Let $\Pi^a_n/\rho$ be the quotient graph of $\Pi^a_n$
obtained by identifying all vertices which are identified by $\rho$, and
two vertices $w_1$ and $w_2$ in $\Pi^a_n/\rho$ are adjacent if there is at
least one edge in $\Pi^a_n$ connecting blocks $\rho^{-1}(w_1)$ and
$\rho^{-1}(w_2)$ Then $\Pi^a_n/\rho$ is isomorphic to the Fibonacci
cube $\Gamma_{n-1}$.
\end{theorem}
\begin{proof} The map $\rho$ identifies two vertices $v_1,v_2\in\Pi^a_n$
if they have the same number and the same positions of blocks $0a$. But
all vertices having equal number of $k$ blocks $0a$ in the same positions
induce a grid subgraph $P_a^{n-2k}$. So, each such grid is mapped into the
single vertex in $\Pi^a_n/\rho$. By Theorem \ref{tm:decomposition II},
$|V(\Pi^a_n/\rho)|=F_{n+1}$. Two grids have at least one edge connecting
them if the number of $0a$ blocks between them differ by exactly one, with
the grid having one block less, having all $0a$ blocks on same positions as
the other grid. For example, for $a=5$ and $n=5$, one grid $P^1_5$ is induced
by vertices $\alpha0505$ and another grid $P^3_5$ is induced by vertices
$\beta_105\beta_2\beta_3$, where
$\alpha,\beta_1,\beta_2,\beta_3\in\left\lbrace 0,1,2,3,4\right\rbrace$.
Then those grids  must have an edge connecting them. For example, the edge
connecting vertices $00504$ and $00505$. From the definition of map $\rho$
it is clear that $\rho$ maps two grids in the neighboring vertices in
$\Pi^a_n/\rho$ only if their image by $\rho$ differs in a single bit. Hence,
$\Pi^a_n/\rho$ is isomorphic to $\Gamma_{n-1}$. \end{proof}

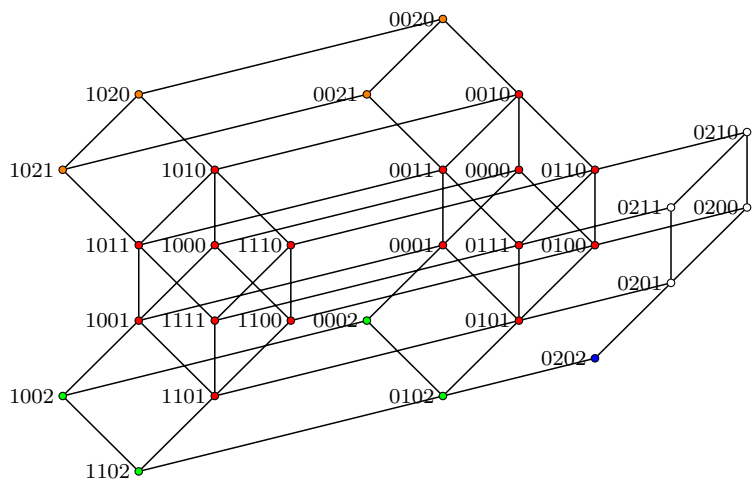
\begin{figure}[h!] \centering
\begin{tikzpicture}[scale=1]
\tikzmath{\x1 = 0.35; \y1 =-0.05; \z1=180; \w1=0.2; \xs=-8; \ys=0; \yss=-5;
\x2 = \x1 + 1; \y2 =\y1 +3; } 
\scriptsize
\node [label={[label distance=\y1 cm]\z1: $0001$},circle,fill=red,draw=black,scale=\x1](A1) at (0,1) {};
\node [label={[label distance=\y1 cm]\z1: $0101$},circle,fill=red,draw=black,scale=\x1](A2) at (1,0) {};
\node [label={[label distance=\y1 cm]\z1: $0100$},circle,fill=red,draw=black,scale=\x1](A3) at (2,1) {};
\node [label={[label distance=\y1 cm]\z1: $0000$},circle,fill=red,draw=black,scale=\x1](A4) at (1,2) {};
\node [label={[label distance=\y1 cm]\z1: $0011$},circle,fill=red,draw=black,scale=\x1](A5) at (0,2) {};
\node [label={[label distance=\y1 cm]\z1: $0111$},circle,fill=red,draw=black,scale=\x1](A6) at (1,1) {};
\node [label={[label distance=\y1 cm]\z1: $0110$},circle,fill=red,draw=black,scale=\x1](A7) at (2,2) {};
\node [label={[label distance=\y1 cm]\z1: $0010$},circle,fill=red,draw=black,scale=\x1](A8) at (1,3) {};
\node [label={[label distance=\y1 cm]\z1: $0002$},circle,fill=green,draw=black,scale=\x1](A9) at (-1,0) {};
\node [label={[label distance=\y1 cm]\z1: $0102$},circle,fill=green,draw=black,scale=\x1](A10) at (0,-1) {};
\node [label={[label distance=\y1 cm]\z1: $0021$},circle,fill=orange,draw=black,scale=\x1](A11) at (-1,3) {};
\node [label={[label distance=\y1 cm]\z1: $0020$},circle,fill=orange,draw=black,scale=\x1](A12) at (0,4) {};
\node [label={[label distance=\y1 cm]\z1: $1001$},circle,fill=red,draw=black,scale=\x1](A13) at (-4,0) {};
\node [label={[label distance=\y1 cm]\z1: $1101$},circle,fill=red,draw=black,scale=\x1](A14) at (-3,-1) {};
\node [label={[label distance=\y1 cm]\z1: $1100$},circle,fill=red,draw=black,scale=\x1](A15) at (-2,0) {};
\node [label={[label distance=\y1 cm]\z1: $1000$},circle,fill=red,draw=black,scale=\x1](A16) at (-3,1) {};
\node [label={[label distance=\y1 cm]\z1: $1011$},circle,fill=red,draw=black,scale=\x1](A17) at (-4,1) {};
\node [label={[label distance=\y1 cm]\z1: $1111$},circle,fill=red,draw=black,scale=\x1](A18) at (-3,0) {};
\node [label={[label distance=\y1 cm]\z1: $1110$},circle,fill=red,draw=black,scale=\x1](A19) at (-2,1) {};
\node [label={[label distance=\y1 cm]\z1: $1010$},circle,fill=red,draw=black,scale=\x1](A20) at (-3,2) {};
\node [label={[label distance=\y1 cm]\z1: $1002$},circle,fill=green,draw=black,scale=\x1](A21) at (-5,-1) {};
\node [label={[label distance=\y1 cm]\z1: $1102$},circle,fill=green,draw=black,scale=\x1](A22) at (-4,-2) {};
\node [label={[label distance=\y1 cm]\z1: $1021$},circle,fill=orange,draw=black,scale=\x1](A23) at (-5,2) {};
\node [label={[label distance=\y1 cm]\z1: $1020$},circle,fill=orange,draw=black,scale=\x1](A24) at (-4,3) {};

\node [label={[label distance=\y1 cm]\z1: $0201$},circle,fill=white,draw=black,scale=\x1](A25) at (3,0.5) {};
\node [label={[label distance=\y1 cm]\z1: $0200$},circle,fill=white,draw=black,scale=\x1](A26) at (4,1.5) {};
\node [label={[label distance=\y1 cm]\z1: $0210$},circle,fill=white,draw=black,scale=\x1](A27) at (4,2.5) {};
\node [label={[label distance=\y1 cm]\z1: $0211$},circle,fill=white,draw=black,scale=\x1](A28) at (3,1.5) {};
\node [label={[label distance=\y1 cm]\z1: $0202$},circle,fill=blue,draw=black,scale=\x1](A29) at (2,-0.5) {};

\draw [line width=\w1 mm] (A1)--(A2)--(A3)--(A4)--(A1)--(A5)--(A6)--(A7)--(A8)--(A5)--(A11)--(A12)--(A8) (A6)--(A2)--(A10)--(A9)--(A1) (A3)--(A7) (A4)--(A8) 
(A13)--(A14)--(A15)--(A16)--(A13)--(A17)--(A18)--(A19)--(A20)--(A17)--(A23)--(A24)--(A20) (A18)--(A14)--(A22)--(A21)--(A13) (A15)--(A19) (A16)--(A20)  (A25)--(A26)--(A27)--(A28)--(A25)--(A29) (A22)--(A10)--(A29) (A14)--(A2)--(A25) (A15)--(A3)--(A26) (A18)--(A6)--(A28) (A19)--(A7)--(A27)  (A21)--(A9) (A13)--(A1) (A16)--(A4) (A17)--(A5) (A20)--(A8) (A23)--(A11) (A24)--(A12); 
\end{tikzpicture}  
\caption{Decomposition of $\Pi_4^{2}=P^{4}_2\oplus P^{2}_2\oplus P^{2}_2\oplus P^{2}_2\oplus P_2^{0}$.} \label{fig:canon dec of Pi_4^2}
\end{figure}


\subsection{Embedding into hypercubes}

Since hypercubes $Q_n$ have binary strings as vertices, and all binary
strings of length $n$ are vertices in $\Pi_n^{a}$ for $a\geq 2$, we have a
natural inclusion $Q_n\subset \Pi^{a}_n$. Furthermore, for $a\geq 1$, since
$\mathcal{S}^1_n\subset \mathcal{S}^a_n$, for every $a$, we have
$\Gamma_{n-1}=\Pi^{1}_n\subseteq\Pi^{a}_n$. The following theorem shows
that every metallic cube is an induced subgraph of some Fibonacci cube. 

\begin{theorem}\label{tm:Pi_embedding}  For any $a\geq 1$ and $n\geq 1$,
metallic cubes are induced subgraphs of Fibonacci cubes and, hence,
hypercubes.
\end{theorem}
\begin{proof} 
For $a=2$, we define a map $\sigma$ on primitive blocks, i.e., on the blocks
that every string from the set $\mathcal{S}^{2}_{n}$ can be uniquely
decomposed into, $0$, $1$, and $02$ as follows \begin{align*}
\sigma(0)&=001\\
\sigma(1)&=000\\
\sigma(02)&=001010. 
\end{align*}
For $a>2$, primitive blocks, are $0,1,\dots,a-1$ and $0a$. So, we set  
\begin{align*}
\sigma(0)&=010101\cdots 010101\\
\sigma(1)&=000101\cdots 010101\\
\sigma(2)&=000001\cdots 010101\\
 \vdots\\
\sigma(a-2)&=000000 \cdots 000001\\
\sigma(a-1)&=000000 \cdots 000000\\
\sigma(0a)&=010101\cdots 01010100100000\cdots 000000, 
\end{align*}
where strings $\sigma(k)$ have length $2a-2$, for $0\leq k\leq a-1$ and
string $\sigma(0a)$ has length $4n-4$.  

In case of $a=2$, the map $\sigma$, defined on primitive blocks, expands to
$\sigma:\Pi^2_n\to\Pi^1_{3n}$ by concatenation. In case of $a\geq 3$,
it expands to $\sigma:\Pi^a_n\to\Pi^1_{(2a-2)n}$. Since the map is injective
and preserves adjacency, we obtained an induced subgraph of Fibonacci cube
isomorphic to the graph $\Pi^{a}_{n}$. \end{proof}

Next we show that metallic cubes are median graphs.
A {\em median} of three vertices is a vertex that lies on a shortest path
between every two of three vertices. We say that a graph $G$ is a {\em median
graph} if every three vertices of $G$ have unique median. 

\begin{figure}[h!] \centering
\begin{tikzpicture}[scale=0.75]
\tikzmath{\x1 = 0.35; \y1 =-0.05; \z1=180; \w1=0.2; \xs=-8; \ys=0; \yss=-5;
\x2 = \x1 + 1; \y2 =\y1 +3; } 
\small

\node [label={[label distance=\y1 cm]\z1: $00$},circle,fill=blue,draw=black,scale=\x1](B1) at (0+\xs,4+\ys) {};
\node [label={[label distance=\y1 cm]\z1: $10$},circle,fill=red,draw=black,scale=\x1](B2) at (0+\xs,5+\ys) {};
\node [label={[label distance=\y1 cm]\z1: $20$},circle,fill=blue,draw=black,scale=\x1](B3) at (0+\xs,6+\ys) {};
\node [label={[label distance=\y1 cm]\z1: $01$},circle,fill=blue,draw=black,scale=\x1](B4) at (-1.5+\xs,2.5+\ys) {};
\node [label={[label distance=\y1 cm]\z1: $11$},circle,fill=blue,draw=black,scale=\x1](B5) at (-1.5+\xs,3.5+\ys) {};
\node [label={[label distance=\y1 cm]\z1: $21$},circle,fill=blue,draw=black,scale=\x1](B6) at (-1.5+\xs,4.5+\ys) {};
\node [label={[label distance=\y1 cm]\z1: $02$},circle,fill=blue,draw=black,scale=\x1](B7) at (-3+\xs,1+\ys) {};
\node [label={[label distance=\y1 cm]\z1: $12$},circle,fill=blue,draw=black,scale=\x1](B8) at (-3+\xs,2+\ys) {};
\node [label={[label distance=\y1 cm]\z1: $22$},circle,fill=red,draw=black,scale=\x1](B9) at (-3+\xs,3+\ys) {};
\node [label={[label distance=\y1 cm]\z1: $03$},circle,fill=red,draw=black,scale=\x1](B28) at (-4.5+\xs,-0.5+\ys) {};

\draw [line width=\w1 mm] (B1)--(B2)--(B3) (B4)--(B5)--(B6)   (B1)--(B4)--(B7)  (B3)--(B6)--(B9); 
\draw [line width=\w1 mm,red]   (B7)--(B8)--(B9)  (B7)--(B28) (B2)--(B5)--(B8) ;

\end{tikzpicture} 
\caption{Median of vertices $10$, $22$ and $03$ is a vertex $12$ .} \label{fig:Pi_3^2 median}
\end{figure}
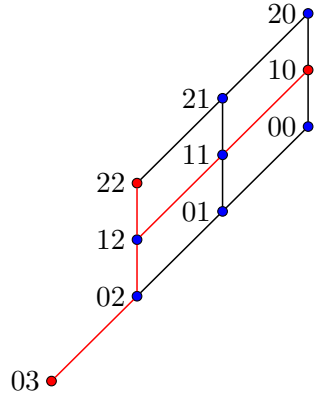

Figure \ref{fig:Pi_3^2 median} shows that the unique median of vertices $10$,
$22$ and $03$ is vertex $12$, because that is only vertex that lies in
shortest paths between every two red vertices.

Klav\v{z}ar \cite{Klavzar} proved that Fibonacci and Lucas cubes are median
graphs, and Munarini \cite{Munarini} proved that Pell graphs are median
graphs. They both used the following theorem by Mulder \cite{Mulder}:
\begin{theorem}[Mulder]
A graph $G$ is median graph if and only if $G$ is a connected induced subgraph
of an $n$-cube such that with any three vertices of $G$ their median in
$n$-cube is also a vertex of $G$. 
\end{theorem}

\begin{theorem} For all $a\geq 1$ and $n\geq 0$, $\Pi^{a}_n$ is a median graph. 
\end{theorem}
\begin{proof}
It is well known  that the median of the triple in $Q_n$ is obtained by the
majority rule. So let $\alpha=\epsilon_1\cdots\epsilon_n$,
$\beta=\delta_1\cdots\delta_n$ and $\gamma=\rho_1\cdots\rho_n$ be binary
strings, i.e., $\epsilon_i,\gamma_i,\rho_i\in\left\lbrace 0,1\right\rbrace$.
Then their median is $m=\zeta_i\cdots\zeta_n$, where $\zeta_i$ is equal to
the number that appears at least twice among the numbers $\epsilon_i$,
$\gamma_i$ and $\rho_i$.

Let $a\geq 3$ and let $\sigma$ be the map defined in the proof of
Theorem \ref{tm:Pi_embedding}. By the same theorem, $\Pi^a_n$ is a connected
induced subgraph of a $(2a-2)n$-cube. To finish our proof, we just need
to verify that the subgraph induced by the set
$\sigma\left(\mathcal{S}^a_n\right)$  is median closed, i.e., that for
every three vertices in $\sigma\left(\mathcal{S}^a_n\right)$, their median
is also a vertex in $\sigma\left(\mathcal{S}^a_n\right)$. Every string from 
the set $\sigma\left(\mathcal{S}^a_n\right)$ can be uniquely decomposed into
blocks $\sigma\left(j\right)$ with length $2a-2$, for $0\leq j\leq a-1$,
and block $\sigma\left(0a\right)$ with length $4a-4$. Note that the median
of three blocks, where two blocks are the same, is that block that appears
at least twice. Hence, we only need to consider cases where all three blocks
are different. First consider three blocks $\sigma\left(i\right)$,
$\sigma\left(j\right)$ and $\sigma\left(k\right)$. Without loss of
generality we can assume that $0\leq i<j<k\leq a-1$. Then their median is
$\sigma\left(j\right)$. In the second case we have $\sigma\left(i\right)$
and $\sigma\left(j\right)$ for $0\leq i\leq j\leq a-1$, and the second half
of the block $\sigma\left(0a\right)$, i.e. of the string $0010\cdots 0$ with
length $(2a-2)n$. But, since one comes only in even positions in strings
$\sigma\left(i\right)$ and $\sigma\left(j\right)$, their median is 
$\sigma\left(j\right)$. Also note that we do not have to consider a case
with the first half of the block $\sigma\left(0a\right)$, because it is
equal to the block $\sigma\left(a-1\right)$. So, we conclude that the
subgraph induced by set $\sigma\left(\mathcal{S}^a_n\right)$ is median closed.
The proof for $a=2$ is similar, so we omit the details.
\end{proof}

\section{The number of edges and distribution of degrees}

Now that we have elucidated the recursive structure of metallic cubes via
the canonical decomposition, the recurrences for the number of edges and
the details of degree distributions can be simply read off from our
structural results. We start with counting the edges of $\Pi^a_n$.

Let $e^a_n$ denote the number of edges in $\Pi^a_n$, i.e.,
$e^a_n=|E(\Pi^a_n)|$. Since $\Pi^a_0$ is empty graph, $e^a_0=0$.
The graph $\Pi^a_1$ is the path graph with $a$ vertices, so, we have
$e^a_1=a-1$. The graph $\Pi^a_2$ is an $a\times a$ grid with addition of
vertex $0a$ being adjacent to vertex $0(a-1)$. Hence, $e^a_2=a^2+1$.
For larger values of  $n$, the graph $\Pi^a_n$ consists of $a$ copies of
$\Pi^a_{n-1}$ and a single copy of $\Pi^a_{n-2}$, and those subgraphs
contribute with $a\cdot e^a_{n-1}+e^a_{n-2}$ edges. Furthermore, there are
$(a-1)\cdot s^a_{n-1}+s^a_{n-2}$ edges connecting $a$ subgraphs
$\Pi^a_{n-1}$ and one subgraph $\Pi^a_{n-2}$. Since
$s^a_n-s^a_{n-1}=(a-1)s^a_{n-1}+s^a_{n-2}$, the overall number of edges is
\begin{equation*}
e^a_n=a\cdot e^a_{n-1}+e^a_{n-2}+s^a_n-s^a_{n-1}.
\end{equation*}     
Hence, the number of edges satisfies a non-homogeneous linear recurrence of
length 2 with the same coefficients in the homogeneous part as the recurrence
for the number of vertices.
Table \ref{table:number of edges} shows some first few values of $e^a_n$.
With some care, patterns of the coefficients of the polynomials appearing
there can be analyzed, suggesting the explicit formula established in our
following theorem.
\begin{theorem}
The number of edges in graph $\Pi^a_n$ is
\begin{equation*}
e^a_n=\sum_{k=0}^n(-1)^{n+k}\ceil*{\frac{n+k}{2}}\binom{\floor*{\frac{n+k}{2}}}{k}a^k
\end{equation*} 
\end{theorem}
\begin{proof}
We proceed by induction on $n$. For $n=1$ and $n=2$ the statement holds.
Furthermore, the number of vertices $s_n$ satisfy the well-known identity
$s^a_n=\sum\limits_{k\geq 0}\binom{n-k}{k}a^{n-2k}$. Then we have
\begin{align*}
s^a_{n}-s^a_{n-1}&=\sum\limits_{k\geq 0}\binom{n-k}{k}a^{n-2k}-\sum\limits_{k\geq 0}\binom{n-k-1}{k}a^{n-2k-1}\\
&=\sum\limits_{k=0}^n(-1)^k\binom{n-\ceil*{\frac{k}{2}}}{\floor*{\frac{k}{2}}}a^{n-k}\\
&=\sum\limits_{k= 0}^n(-1)^k\binom{n-\ceil*{\frac{k}{2}}}{n-k}a^{n-k}\\
&=\sum\limits_{k=0}^n(-1)^{n+k}\binom{n-\ceil*{\frac{n-k}{2}}}{k}a^{k}\\
&=\sum\limits_{k=0}^n(-1)^{n+k}\binom{\floor*{\frac{n+k}{2}}}{k}a^{k}
\end{align*}
By using the inductive hypothesis, after adjusting indices and expanding the
 range of summation, we obtain
\begin{align*}
a\cdot e^a_{n-1}+e^a_{n-2}=&\sum_{k=0}^{n-1}(-1)^{n+k-1}\ceil*{\frac{n+k-1}{2}}\binom{\floor*{\frac{n+k-1}{2}}}{k}a^{k+1}+\\&+\sum_{k=0}^{n-2}(-1)^{n+k}\ceil*{\frac{n+k-2}{2}}\binom{\floor*{\frac{n+k-2}{2}}}{k}a^{k}\\
=&\sum_{k=1}^{n}(-1)^{n+k}\ceil*{\frac{n+k-2}{2}}\binom{\floor*{\frac{n+k-2}{2}}}{k-1}a^{k}+\\&+\sum_{k=0}^{n}(-1)^{n+k}\ceil*{\frac{n+k-2}{2}}\binom{\floor*{\frac{n+k-2}{2}}}{k}a^{k}\\
=&\sum_{k=0}^{n}(-1)^{n+k}\ceil*{\frac{n+k-2}{2}}\binom{\floor*{\frac{n+k}{2}}}{k}a^{k}
\end{align*}
Now, by using expressions for $s^a_{n}-s^a_{n-1}$ and
$a\cdot e^a_{n-1}+e^a_{n-2}$ our claim follows at once. \end{proof}

\begin{table}\centering$\begin{array}{r|l}
n & e^a_n\\
\hline
1 & a-1\\
2 & 2a^2-2a+1\\
3 & 3a^3-3a^2+4a-1\\
4 & 4a^4-4a^3+9a^2-6a+1\\
5 & 5a^5-5a^4+16a^3-12a^2+9a-1\\
\end{array}$\label{table:number of edges}\caption{Number of edges in $\Pi^a_n$.}
\end{table}

For Fibonacci cubes, Klav\v{z}ar \cite{Klavzar} proved that
$|E(\Gamma_{n})|=F_{n+1}+\sum_{k=1}^{n-2}F_kF_{n+1-k}$. By plugging in $a=1$
into our result and recalling that 
$\Pi^1_ {n}=\Gamma_{n-1}$, we have obtained a combinatorial proof of the
following identity.
\begin{corollary}
\begin{equation*}
\sum_{k=0}^n(-1)^{n+k}\ceil*{\frac{n+k}{2}}\binom{\floor*{\frac{n+k}{2}}}{k}=\sum_{k=0}^nF_kF_{n-k}.
\end{equation*}
\end{corollary}

Now we turn our attention to degrees of vertices of metallic cubes and to
their distribution. We consider first the case $a\geq 3$.
Hence, let $a\geq 3$ and $v=\alpha_1\cdots\alpha_n\in\Pi^a_n$ be any vertex.
If $0<\alpha_i<a-1$ for all $i$, then $v$ has maximum possible degree, which
is $2n$. Each $0a$ block reduces the degree by $3$, each $0(a-1)$-block
reduces the degree by $1$, as well as each letter $0$ or $a-1$ which is not
part of some $0(a-1)$-block. We are interested in the number of vertices with
degree $2n-m$. So let $v$ be a vertex containing $l$ $0a$-blocks, $h$
$0(a-1)$-blocks and $k$ letters $0$ or $a-1$, where none of $k$ letters $0$
and $a-1$ are part of $0(a-1)$-block. Then $v$ has degree $2n-3l-h-k$.
Now we want to count the number of such vertices. There are
$\binom{n-h-l}{h}\binom{n-2h-l}{l}\binom{n-2h-2l}{k}$ ways to choose
positions of $0a$-blocks, $0(a-1)$-blocks and letters $0$ or $a-1$. The
rest of $n-2h-2l-k$ positions we can fill in $(a-2)^{n-2h-2l-k}$ ways.
If we set that there are $2^k$ ways to fill $k$ positions with $0$ or $a-1$,
we possibly obtained more that $l$ $0(a-1)$-blocks and counted some vertices
more than once. Let
$q(n,l,h,k)=q(l,k)=\binom{n-h-l}{h}\binom{n-2h-l}{l}\binom{n-2h-2l}{k}2^k(a-2)^{n-2h-2l-k}$
denote the number of vertices where all vertices with exactly $l$ blocks
are counted once, and vertices with more then $l$ blocks are counted more
then once. More precisely, beside vertices of length $l$, number $q(l)$
counts all vertices with $l+1, l+2, \dots l+\floor*{\frac{k}{2}}$ blocks
$0(a-1)$, where each vertex with exactly $l+j$ blocks is counted
$\binom{l+j}{l}$ times. If $p(n,l,h,k)=p(l,k)$ denotes the number of
vertices with precisely $l$ blocks $0(a-1)$, we have
\begin{align*}
q(l,k)=\sum\limits_{j=0}^{\floor*{\frac{k}{2}}}\binom{l+j}{l}p(l+j,k-2j).
\end{align*} 

Now we retrieve the numbers $p(l,k)$ by direct counting. To obtain the number
of vertices with exactly $l$ blocks, we have to subtract the number of
vertices with more than $l$ blocks. Vertices having $l+1$ block were
counted $\binom{l+1}{l}$ times in $q(l,k)$ and subtracting
$q(l,k)-q(l+1,k-2)\binom{l+1}{l}$ corrects that error. But then the vertices
containing $l+2$ blocks were counted
$\binom{l+2}{l}-\binom{l+2}{l+1}\binom{l+1}{l}=\binom{l+2}{l}$ times so far,
so we need to add $\binom{l+2}{l}q(l+2,k-4)$ to cancel that error.
Inductively, vertex with $l+j$ blocks is counted
\begin{align*}
\binom{l+j}{j}-\sum\limits_{s=0}^{j-1}\binom{l+j}{l+s}\binom{l+s}{l}&=\binom{l+j}{j}-\binom{l+j}{l}\sum\limits_{s=0}^{j-1}\binom{s}{j}\\
&=\begin{cases} 0 & \textup{ for } j \textup{ odd} \\ -2\binom{l+j}{j} & \textup{ for } j \textup{ even} \end{cases}
\end{align*} 
times, and there are $q(l+j,k-2j)$ such vertices. So we conclude that the
number of vertices that have exactly $l$ blocks $0(a-1)$ is
\begin{align}\label{eq:Delta(n,2n-m)}
p(l,k)=\sum_{j=0}^{\floor*{\frac{k}{2}}}(-1)^j \binom{l+j}{j}q(l+j,k-2j).
\end{align}

Let $\Delta_{n,m}$ denotes the overall number of vertices of degree $m$.
Then we have
\begin{align*}
\Delta_{n,2n-m}&=\sum_{\substack{3h+l+k=m \\ 2h+2l+k\leq n \\l,h,k\geq 0}} p(n,h,l,k).
\end{align*}
The maximum degree is achieved for $m=0$. Then formula
(\ref{eq:Delta(n,2n-m)}) reduces itself to $\Delta_{n,2n}=(a-2)^n$. The minimum
degree is achieved for the largest possible number of $0a$-blocks, which
reduce degree the most, hence $l=\floor*{\frac{n}{2}}$ and
$k=n-2\floor*{\frac{n}{2}}$. Thus $m=n+\floor*{\frac{n}{2}}$. The formula
for $\Delta_{n,\ceil*{\frac{n}{2}}}$ reduces to
\begin{align*}
\Delta_{n,\ceil*{\frac{n}{2}}}&=q\left(n,\floor*{\frac{n}{2}},0,n-2\floor*{\frac{n}{2}}\right)\\&=\left(\binom{n-\floor*{\frac{n}{2}}}{\floor*{\frac{n}{2}}}2^{n-2\floor*{\frac{n}{2}}}\right)(a-2)^{0}\\
&=\begin{cases}1&\textup{for } n \textup{ even;}\\ n+1 &\textup{for } n \textup{ odd.} \end{cases}
\end{align*}
Figure \ref{fig:Distribution of degree} shows degree of each vertex in
graphs $\Pi^3_1$, $\Pi^3_2$ and $\Pi^3_3$.

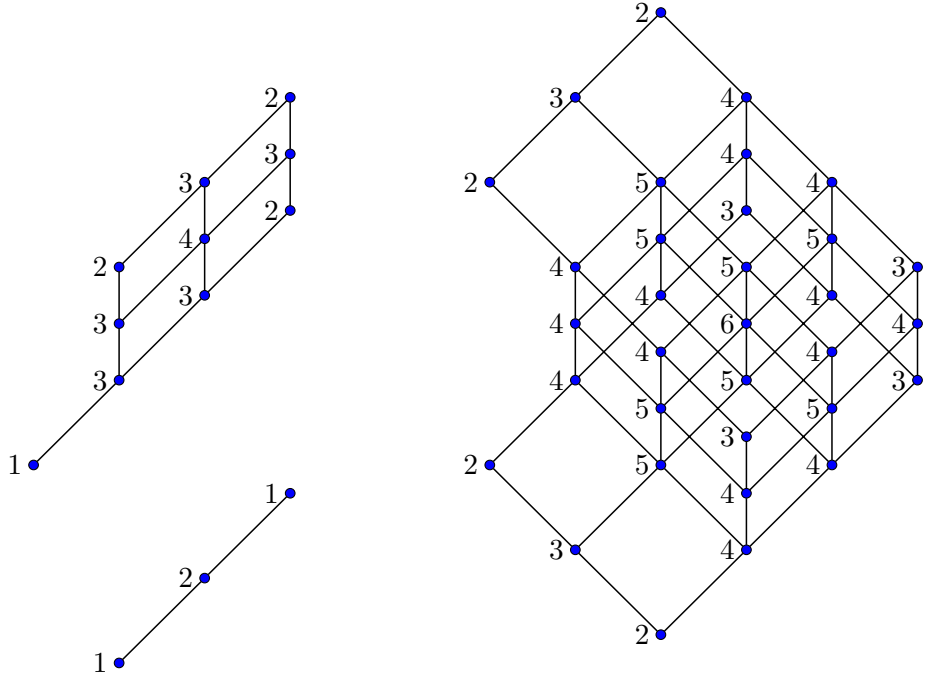
\begin{figure}[h!] \centering
\begin{tikzpicture}[scale=0.75]
\tikzmath{\x1 = 0.35; \y1 =-0.05; \z1=180; \w1=0.2; \xs=-8; \ys=0; \yss=-5;
\x2 = \x1 + 1; \y2 =\y1 +3; } 
\small
\node [label={[label distance=\y1 cm]\z1: $3$},circle,fill=blue,draw=black,scale=\x1](A1) at (0,4) {};
\node [label={[label distance=\y1 cm]\z1: $4$},circle,fill=blue,draw=black,scale=\x1](A2) at (0,5) {};
\node [label={[label distance=\y1 cm]\z1: $4$},circle,fill=blue,draw=black,scale=\x1](A3) at (0,6) {};
\node [label={[label distance=\y1 cm]\z1: $4$},circle,fill=blue,draw=black,scale=\x1](A4) at (-1.5,2.5) {};
\node [label={[label distance=\y1 cm]\z1: $5$},circle,fill=blue,draw=black,scale=\x1](A5) at (-1.5,3.5) {};
\node [label={[label distance=\y1 cm]\z1: $5$},circle,fill=blue,draw=black,scale=\x1](A6) at (-1.5,4.5) {};
\node [label={[label distance=\y1 cm]\z1: $4$},circle,fill=blue,draw=black,scale=\x1](A7) at (-3,1) {};
\node [label={[label distance=\y1 cm]\z1: $4$},circle,fill=blue,draw=black,scale=\x1](A8) at (-3,2) {};
\node [label={[label distance=\y1 cm]\z1: $4$},circle,fill=blue,draw=black,scale=\x1](A9) at (-3,3) {};
\node [label={[label distance=\y1 cm]\z1: $4$},circle,fill=blue,draw=black,scale=\x1](A10) at (1.5,2.5) {};
\node [label={[label distance=\y1 cm]\z1: $5$},circle,fill=blue,draw=black,scale=\x1](A11) at (1.5,3.5) {};
\node [label={[label distance=\y1 cm]\z1: $4$},circle,fill=blue,draw=black,scale=\x1](A12) at (1.5,4.5) {};
\node [label={[label distance=\y1 cm]\z1: $5$},circle,fill=blue,draw=black,scale=\x1](A13) at (0,1) {};
\node [label={[label distance=\y1 cm]\z1: $6$},circle,fill=blue,draw=black,scale=\x1](A14) at (0,2) {};
\node [label={[label distance=\y1 cm]\z1: $5$},circle,fill=blue,draw=black,scale=\x1](A15) at (0,3) {};
\node [label={[label distance=\y1 cm]\z1: $5$},circle,fill=blue,draw=black,scale=\x1](A16) at (-1.5,-0.5) {};
\node [label={[label distance=\y1 cm]\z1: $5$},circle,fill=blue,draw=black,scale=\x1](A17) at (-1.5,0.5) {};
\node [label={[label distance=\y1 cm]\z1: $4$},circle,fill=blue,draw=black,scale=\x1](A18) at (-1.5,1.5) {};
\node [label={[label distance=\y1 cm]\z1: $3$},circle,fill=blue,draw=black,scale=\x1](A19) at (3,1) {};
\node [label={[label distance=\y1 cm]\z1: $4$},circle,fill=blue,draw=black,scale=\x1](A20) at (3,2) {};
\node [label={[label distance=\y1 cm]\z1: $3$},circle,fill=blue,draw=black,scale=\x1](A21) at (3,3) {};
\node [label={[label distance=\y1 cm]\z1: $4$},circle,fill=blue,draw=black,scale=\x1](A22) at (1.5,-0.5) {};
\node [label={[label distance=\y1 cm]\z1: $5$},circle,fill=blue,draw=black,scale=\x1](A23) at (1.5,0.5) {};
\node [label={[label distance=\y1 cm]\z1: $4$},circle,fill=blue,draw=black,scale=\x1](A24) at (1.5,1.5) {};
\node [label={[label distance=\y1 cm]\z1: $4$},circle,fill=blue,draw=black,scale=\x1](A25) at (0,-2) {};
\node [label={[label distance=\y1 cm]\z1: $4$},circle,fill=blue,draw=black,scale=\x1](A26) at (0,-1) {};
\node [label={[label distance=\y1 cm]\z1: $3$},circle,fill=blue,draw=black,scale=\x1](A27) at (0,0) {};
\node [label={[label distance=\y1 cm]\z1: $2$},circle,fill=blue,draw=black,scale=\x1](A28) at (-4.5,-0.5) {};
\node [label={[label distance=\y1 cm]\z1: $3$},circle,fill=blue,draw=black,scale=\x1](A29) at (-3,-2) {};
\node [label={[label distance=\y1 cm]\z1: $2$},circle,fill=blue,draw=black,scale=\x1](A30) at (-1.5,-3.5) {};
\node [label={[label distance=\y1 cm]\z1: $2$},circle,fill=blue,draw=black,scale=\x1](A31) at (-1.5,7.5) {};
\node [label={[label distance=\y1 cm]\z1: $3$},circle,fill=blue,draw=black,scale=\x1](A32) at (-3,6) {};
\node [label={[label distance=\y1 cm]\z1: $2$},circle,fill=blue,draw=black,scale=\x1](A33) at (-4.5,4.5) {};

\draw [line width=\w1 mm] (A1)--(A2)--(A3) (A4)--(A5)--(A6) (A7)--(A8)--(A9) (A10)--(A11)--(A12) (A13)--(A14)--(A15) (A16)--(A17)--(A18) (A19)--(A20)--(A21) (A22)--(A23)--(A24) (A25)--(A26)--(A27) (A1)--(A4)--(A7)--(A28) (A2)--(A5)--(A8) (A3)--(A6)--(A9) (A10)--(A13)--(A16)--(A29) (A11)--(A14)--(A17) (A12)--(A15)--(A18) (A19)--(A22)--(A25)--(A30) (A20)--(A23)--(A26) (A21)--(A24)--(A27) (A1)--(A10)--(A19) (A2)--(A11)--(A20) (A31)--(A3)--(A12)--(A21) (A4)--(A13)--(A22) (A5)--(A14)--(A23) (A32)--(A6)--(A15)--(A24) (A7)--(A16)--(A25) (A8)--(A17)--(A26) (A33)--(A9)--(A18)--(A27)  (A28)--(A29)--(A30) (A31)--(A32)--(A33);

\node [label={[label distance=\y1 cm]\z1: $2$},circle,fill=blue,draw=black,scale=\x1](B1) at (0+\xs,4+\ys) {};
\node [label={[label distance=\y1 cm]\z1: $3$},circle,fill=blue,draw=black,scale=\x1](B2) at (0+\xs,5+\ys) {};
\node [label={[label distance=\y1 cm]\z1: $2$},circle,fill=blue,draw=black,scale=\x1](B3) at (0+\xs,6+\ys) {};
\node [label={[label distance=\y1 cm]\z1: $3$},circle,fill=blue,draw=black,scale=\x1](B4) at (-1.5+\xs,2.5+\ys) {};
\node [label={[label distance=\y1 cm]\z1: $4$},circle,fill=blue,draw=black,scale=\x1](B5) at (-1.5+\xs,3.5+\ys) {};
\node [label={[label distance=\y1 cm]\z1: $3$},circle,fill=blue,draw=black,scale=\x1](B6) at (-1.5+\xs,4.5+\ys) {};
\node [label={[label distance=\y1 cm]\z1: $3$},circle,fill=blue,draw=black,scale=\x1](B7) at (-3+\xs,1+\ys) {};
\node [label={[label distance=\y1 cm]\z1: $3$},circle,fill=blue,draw=black,scale=\x1](B8) at (-3+\xs,2+\ys) {};
\node [label={[label distance=\y1 cm]\z1: $2$},circle,fill=blue,draw=black,scale=\x1](B9) at (-3+\xs,3+\ys) {};
\node [label={[label distance=\y1 cm]\z1: $1$},circle,fill=blue,draw=black,scale=\x1](B28) at (-4.5+\xs,-0.5+\ys) {};

\draw [line width=\w1 mm] (B1)--(B2)--(B3) (B4)--(B5)--(B6) (B7)--(B8)--(B9)  (B1)--(B4)--(B7)--(B28) (B2)--(B5)--(B8) (B3)--(B6)--(B9); 

\node [label={[label distance=\y1 cm]\z1: $1$},circle,fill=blue,draw=black,scale=\x1](C1) at (0+\xs,4+\yss) {};
\node [label={[label distance=\y1 cm]\z1: $2$},circle,fill=blue,draw=black,scale=\x1](C4) at (-1.5+\xs,2.5+\yss) {};
\node [label={[label distance=\y1 cm]\z1: $1$},circle,fill=blue,draw=black,scale=\x1](C7) at (-3+\xs,1+\yss) {};

\draw [line width=\w1 mm] (C1)--(C4)--(C7); 

\end{tikzpicture} 
\caption{Graph $\Pi_n^{3}$ for $n=1,2,3$ with degree of each vertex.} \label{fig:Distribution of degree}
\end{figure}

It is now time to consider the cases $a=1$ and $a=2$. Their separate 
treatment is necessary since, for $a=2$, there are no letters beside $0,a-1$
and $a$ and for $a=1$ we only have $0$ and $a$. For $a=2$ we have
$k=n-2h-2l$ and  $q(n,h,l,k)$ becomes
$$q(n,h,l)=\binom{n-h-l}{h}\binom{n-2h-l}{l}2^{n-2h-2l},$$ and the rest of
the argument is the same. The maximum degree is now obtained by setting
$l=\floor*{\frac{n}{2}}$ and $h=0$. The minimum degree is obtained by
setting $h=\floor*{\frac{n}{2}}$ and $l=0$. Since both cases produce the
same number of vertices, we have  
\begin{align*}
\Delta_{n,\ceil*{\frac{n}{2}}}=\Delta_{n,n+\floor*{\frac{n}{2}}}&=q\left(n,0,\floor*{\frac{n}{2}}\right)\\&=\left(\binom{n-\floor*{\frac{n}{2}}}{\floor*{\frac{n}{2}}}2^{n-2\floor*{\frac{n}{2}}}\right)\\
&=\begin{cases}1&\textup{for } n \textup{ even;}\\ n+1 &\textup{for } n \textup{ odd.} \end{cases}
\end{align*}

For $a=1$, we would need to modify our approach, but we omit the details
since distribution of degrees for Fibonacci cubes is already available in the
literature \cite{degree}.

We conclude this section by computing the bivariate generating functions for
two-indexed sequences $\Delta_{n,k}$ counting the number of vertices in
$\Pi_n^a$ with exactly $k$ neighbors. We suppress $a$ in order to simplify
the notation. Take $a\geq 2$. Recall that $\mathcal{T}_a$ denotes the set of
all words from alphabet $\left\lbrace 0,1,2,\dots,a-1,a\right\rbrace$ with
property that letter $a$ can only appear immediately after $0$. Every letter
starts with $0,\dots a-1$ or with the block $0a$. Hence, set $\mathcal{T}_a$
can be decomposed into disjoint subsets based on starting letters. We write
$$\mathcal{T}_a=\epsilon+0\mathcal{T}_a+1\mathcal{T}_a+\cdots+(a-1)\mathcal{T}_a+0a\mathcal{T}_a,$$
where $\epsilon$ denotes the empty word. Let
$\Delta(x,y)=\sum\limits_{n,k}\Delta_{n,k}x^ny^k$ be a formal power series
where $\Delta_{n,k}$ is the number of vertices of length $n$ having $k$
neighbors. Let $\mathcal{T}^0_a$ contain all vertices that start with $0$,
but not with $0a$-block and let $\Delta^0(x,y)$ denote corresponding formal
series. Then we have
\begin{align*}
\Delta(x,y)&=1+\Delta^0(x,y)+(a-2)xy^2\Delta(x,y)+xy\Delta(x,y) +x^2y\Delta(x,y)\\
\Delta^0(x,y)&=xy\left(1+\Delta^0(x,y)+(a-1)xy^2\Delta(x,y)+x^2y\Delta(x,y)\right)
\end{align*}  
Note that adding $0$ at the beginning of a vertex increases the number of
neighbors by one, except when the vertex starts with $a-1$, in which case
adding $0$ increases the number of neighbors by two. By solving the above
system of equations we readily obtain the generating function:
$$\Delta(x,y)=\dfrac{1}{1-(2y+(a-2)y^2)x-(y-y^2+y^3)x^2}.$$
For the reader's convenience, we list some first few values of
$\Delta_{n,k}$ in Table \ref{table:distribution of degrees}.
It is interesting to observe that their rows are not unimodal. Also, they
do not (yet) appear in the {\em On-Line Encyclopedia of Integer Sequences}
\cite{oeis}.
\begin{table}[h]\centering
$\begin{array}{cl}
 a=2&
\begin{array}{c|ccccccc}
n\backslash k & 1 & 2  & 3 & 4 & 5 & 6 & 7\\ 
\hline
1 &  2 & 0 & 0 & 0 & 0 & 0 & 0 \\
2 & 1 & 3 & 1 & 0 & 0 & 0 & 0 \\
3 &  0 & 4 & 4 & 4 & 0 & 0 & 0 \\
4 & 0 & 1 & 10 & 7 & 10 & 1 & 0 \\
5 & 0 & 0 & 6 & 20 & 18 & 20 & 6 \\
\end{array}\\ [10ex]
 a=3 &
\begin{array}{c|cccccccccc}
n\backslash k & 1 & 2  & 3 & 4 & 5 & 6 & 7 & 8 & 9 & 10\\ 
\hline
1 &  2 & 1 & 0 & 0 & 0 & 0 & 0 & 0 & 0 & 0\\
2 & 1 & 3 & 5 & 1 & 0 & 0 & 0 & 0 & 0 & 0\\
3 &  0 & 4 & 6 & 14 & 8 & 1 & 0 & 0 & 0 & 0\\
4 & 0 & 1 & 10 & 19 & 33 & 34 & 11 & 1  & 0 & 0 \\
5 & 0 & 0 & 6 & 23 & 60 & 85 & 108 & 63 & 14 &  1 \\
\end{array}
\end{array}$
\caption{Some first values of $\Delta_{n,k}$ for $a=2$ and $a=3$.} \label{table:distribution of degrees}
\end{table}

The obtained bivariate generating functions could be further used to obtain
the average values and higher order moments of the degree distributions, but 
we leave the details to the interested reader.

\section{Metric properties}

In this section we look at some metric-related properties of metallic cubes.
In particular, for a given metallic cube $\Pi_n^a$ we compute its radius and
diameter and determine its center and periphery as functions of $a$ and $n$.

A path of length $n$ in graph is a sequence of $n$ edges $e_1,e_2,\cdots, e_n$
such that $e_i$ and $e_{i+1}$ have common vertex for every $1\leq i \leq n-1$.
The distance between any two vertices $v_1,v_2\in V(G)$, denoted by
$d(v_1,v_2)$, is the length of any shortest path between them. By
{\em eccentricity} $e(v)$ of a vertex $v\in V(G)$ we mean the distance from
$v$ to any vertex farthest from it. More precisely,
$e(v)=\max_{w\in V(G)}d(v,w)$. The {\em radius} of $G$, denoted by $r(G)$,
is the minimum eccentricity of the vertices, and {\em diameter} of $G$,
denoted by $d(G)$, is the maximum eccentricity over all vertices of $G$. A 
vertex $v$ is {\em central} if $e(v)=r(G)$. The subset
$Z(G)=\left\lbrace v\in V(G): e(v)=r(G)\right\rbrace\subset V(G)$ is called 
the {\em center} of $G$, while the subset
$P(G)=\left\lbrace v\in V(G): e(v)=r(G)\right\rbrace\subset V(G)$ is called 
the {\em periphery} of $G$. In the next few theorems, we determine radius,
center, diameter and periphery of metallic cubes. 

\begin{remark} Let $v=\alpha_1\cdots\alpha_n\in V(\Pi^a_n)$ be an arbitrary
vertex. To reach a vertex most distanced from $v$, every letter in $v$ greater
than $\floor*{\frac{a}{2}}$ must be changed into $0$. Any letter smaller
than $\floor*{\frac{a}{2}}$ must be changed to $a-1$ or, if the letter before
it was changed to $0$, into $a$. A letter equal to $\floor*{\frac{a}{2}}$
must be changed to $0$ unless the letter before it was changed into $0$,
in which case it must be changed into $a$.\end{remark}

\begin{theorem}\label{tm:radius} For $a\geq 1$ and $n\geq 1$ we have
\begin{align}\label{eq:radius}
r(\Pi^a_n)=\floor*{\frac{a}{2}}\ceil*{\frac{n}{2}}+\ceil*{\frac{a}{2}}\floor*{\frac{n}{2}}.
\end{align}
\end{theorem}
\begin{proof} Let $\epsilon=\floor*{\frac{a}{2}}$ and
$\hat{\epsilon}=\epsilon\epsilon\cdots\epsilon\in V(\Pi^a_n)$. We want to
compute the eccentricity $e(\hat{\epsilon})$. It is not hard to see that any
longest path from $\epsilon$ leads to the vertex $0a0a\cdots 0$ if $n$ is odd,
or to $0a0a\cdots 0a$ if $n$ is even. Note that the most distant vertex of
$\hat{\epsilon}$ does not have to be unique. Single letter $\epsilon$ can be
changed to $0$ in $\floor*{\frac{a}{2}}$ steps, and to $a$ in
$a-\floor*{\frac{a}{2}}=\ceil*{\frac{a}{2}}$ steps. Since $\hat{\epsilon}$
has length $n$, there are $\ceil*{\frac{n}{2}}$ odd positions and
$\floor*{\frac{n}{2}}$ even positions. That brings us to the total of
$\floor*{\frac{a}{2}}\ceil*{\frac{n}{2}}+\ceil*{\frac{a}{2}}\floor*{\frac{n}{2}}$
steps. Thus,
$e(\epsilon\epsilon\cdots\epsilon)=\floor*{\frac{a}{2}}\ceil*{\frac{n}{2}}+\ceil*{\frac{a}{2}}\floor*{\frac{n}{2}}$. 

As the next step, we want to show that every other vertex has eccentricity
at least $e(\epsilon)$. Let $v\in V(\Pi^a_n)$ be arbitrary. If $v$ has $k$
blocks $0a$, then it has $n-2k$ letters in alphabet $0,1,\dots,a-1$. The
most distant vertex is obtained if we change every single letter
$0,1,\dots,a-1$ in string $v$ at least $\floor*{\frac{a}{2}}$ times, and if
we change each block $0a$ into $(a-1)0$. The latter transition requires
$k(2a-1)$ steps. 

To prove that
$e(v)\geq \floor*{\frac{a}{2}}\ceil*{\frac{n}{2}}+\ceil*{\frac{a}{2}}\floor*{\frac{n}{2}}$,
we will consider cases of even and odd $a$ separately. If $a$ is even,
then formula (\ref{eq:radius}) becomes  $r(\hat{\epsilon})=n\cdot\frac{a}{2}$.
Hence,
\begin{align*}
e(v)\geq (n-2k)\cdot\frac{a}{2}+k(2a-1)= n\cdot\frac{a}{2}+k(a-1)\geq n\cdot\frac{a}{2}= e(\hat{\epsilon}).
\end{align*} 
For $a$ odd, formula \ref{eq:radius} reduces to
$r(\hat{\epsilon})=n\cdot\frac{a}{2}-\frac{1}{2}\left(\ceil*{\frac{n}{2}}-\floor*{\frac{n}{2}}\right)$.
We have to prove that changing arbitrary number of letters in $\hat{\epsilon}$
does not decrease the eccentricity. So, let $k$ denote the number of letters
that differ from $\epsilon$. That leaves us with $s$ blocks of $\epsilon$ of
length $h_1,h_2\dots,h_s$. Note that $n=k+\sum h_i$ and $s\leq k+1$. According
to the first part of the proof, each $\epsilon$-block contributes to
eccentricity with
$h_i\cdot\frac{a}{2}-\frac{1}{2}\left(\ceil*{\frac{h_i}{2}}-\floor*{\frac{h_i}{2}}\right)$,
and each $\alpha\neq\epsilon$ contributes at least with $\frac{a+1}{2}$.
Note that block $0a$ contributes more than any two letters, so we can assume
$v$ does not contain $0a$ blocks. We have
\begin{align*}
e(v)&\geq \sum\left(h_i\cdot\frac{a}{2}-\frac{1}{2}\left(\ceil*{\frac{h_i}{2}}-\floor*{\frac{h_i}{2}}\right)\right)+k\cdot\frac{a+1}{2}\\
&= n\cdot\frac{a}{2}-\frac{1}{2}\sum\left(\ceil*{\frac{h_i}{2}}-\floor*{\frac{h_i}{2}}\right)+\frac{k}{2}\\
&\geq n\cdot\frac{a}{2}+\frac{k-s}{2}\\
&\geq n\cdot\frac{a}{2}-\frac{1}{2}.
\end{align*}
Hence, for odd $n$, we have $e(v)\geq e(\hat{\epsilon})$. For even $n$, we
have to show that
$\sum\left(\ceil*{\frac{h_i}{2}}-\floor*{\frac{h_i}{2}}\right)\leq k$.
If $\sum\left(\ceil*{\frac{h_i}{2}}-\floor*{\frac{h_i}{2}}\right)=k+1$,
each $h_i$ is odd and $s$ and $k$ have different parity. If $s$ is odd, then
$\sum h_i$ is odd, and since $n=\sum h_i+k$, $k$ is odd too. It follows
that $s$ is even. Then $\sum h_i$ is even and $k$ is even too. Thus, we
reached contradiction with
$\sum\left(\ceil*{\frac{h_i}{2}}-\floor*{\frac{h_i}{2}}\right)=k+1$.
It follows that
$\sum\left(\ceil*{\frac{h_i}{2}}-\floor*{\frac{h_i}{2}}\right)\leq k$, and   
\begin{align*}
e(v)&\geq n\cdot\frac{a}{2}-\frac{1}{2}\sum\left(\ceil*{\frac{h_i}{2}}-\floor*{\frac{h_i}{2}}\right)+\frac{k}{2}\\
&\geq n\cdot\frac{a}{2}\\
&= e(\hat{\epsilon}).
\end{align*}
\end{proof}

\begin{corollary} For $a$ and $n$ odd, the center of $\Pi_n^a$ consists
of a single vertex, $Z(\Pi^a_n)=\left\lbrace \hat{\epsilon}\right\rbrace$.
\end{corollary}
\begin{proof}
Let $a$ and $n$ be odd and let $\epsilon=\floor*{\frac{a}{2}}=\frac{a-1}{2}$.
By Theorem \ref{tm:radius}, we have $R(\Pi^a_n)=n\cdot\frac{a}{2}-\frac{1}{2}$
and $\hat{\epsilon}\in Z(\Pi^a_n)$. Suppose $v\in Z(\Pi^a_n)$. We want to prove
that $v=\hat{\epsilon}$. As in the proof of Theorem \ref{tm:radius}, let
$h_1,\cdots,h_s$ denote the length of $\epsilon$ blocks in $v$, and let $k$
denotes the number of letters $\alpha\neq\epsilon$. Since $v\in Z(\Pi^a_n)$,
contributions of $\epsilon$-blocks and letters between them must be minimal.
Minimal contribution of each letter $\alpha\neq\epsilon$ is $\frac{a+1}{2}$,
which is achieved only by changing $\epsilon+1$ into $0$ or $\epsilon-1$ into
$a-1$. Hence, $\alpha=\epsilon\pm 1$. Minimal contribution of each
$\epsilon$-blocks is
$h_i\cdot\frac{a}{2}-\frac{1}{2}\left(\ceil*{\frac{h_i}{2}}-\floor*{\frac{h_i}{2}}\right)$.  
We have:
\begin{align*}
e(v)&=\sum\left(h_i\cdot\frac{a}{2}-\frac{1}{2}\left(\ceil*{\frac{h_i}{2}}-\floor*{\frac{h_i}{2}}\right)\right)+k\cdot\frac{a+1}{2}\\
&= n\cdot\frac{a}{2}-\frac{1}{2}\sum\left(\ceil*{\frac{h_i}{2}}-\floor*{\frac{h_i}{2}}\right)+\frac{k}{2}\\
&= n\cdot\frac{a}{2}-\frac{1}{2}. 
\end{align*}
It follows that
$\frac{1}{2}\sum\left(\ceil*{\frac{h_i}{2}}-\floor*{\frac{h_i}{2}}\right)-\frac{k}{2}=\frac{1}{2}$,
so we conclude that
$\sum\left(\ceil*{\frac{h_i}{2}}-\floor*{\frac{h_i}{2}}\right)=s=k+1$. That
means that the length of $\epsilon$-blocks must be odd and between every two
blocks there is a single letter $\alpha_i$ different from $\epsilon$. So $v$
has the form
$\hat{\epsilon}_1\alpha_1\cdots \hat{\epsilon}_{k}\alpha_{k}\hat{\epsilon}_{k+1}$.
The first $\epsilon$-block must be changed into $0a0\cdots 0a0$, hence
$\alpha_1$ must be $\epsilon+1$, which, by changing into $0$ contributes
$\frac{a+1}{2}$. Otherwise, $\alpha_1=\epsilon-1$ could be changed into $a$,
and contribution would be increased by one. But then, in order to achieve
maximal distance, block $\hat{\epsilon_2}$ must be changed into
$a0a0a\cdots0a$, which is contradiction with minimality of the contribution.
That means $k=0$, and $v=\epsilon$.
\end{proof}

Let $\epsilon=\floor*{\frac{a}{2}}$ and
$\hat{\epsilon}=\epsilon\cdots\epsilon\in V(\Pi^a_n)$, as before. We define
$U^a_n\subseteq V(\Pi^a_n)$ to be the set containing all vertices such that
\begin{enumerate}
\item they differ from $\hat{\epsilon}$ in at most one position, i.e.,
they have form $\epsilon\cdots\epsilon\alpha\epsilon\cdots\epsilon$,
where $\alpha=\epsilon$ or $\alpha=\epsilon\pm 1$,
\item letter $\epsilon+1$ can only appear on even positions,
\item letter $\epsilon-1$ can only appear on odd positions.
\end{enumerate} 
Note that for an even $n$, every vertex $v\in U^a_n$ not equal to
$\hat{\epsilon}$ has exactly one $\epsilon$-block of odd length. 

\begin{corollary} For $a$ odd and $n$ even, $Z(\Pi^a_n)=U_n^a$.
\end{corollary}
\begin{proof}
By Theorem \ref{tm:radius}, $\hat{\epsilon}\in Z(\Pi^a_n)$, so let
$v\in U^a_n$ be arbitrary and different from $\hat{\epsilon}$. Let $h_1$
denote the length of odd $\epsilon$-block and, since $n$ is even, the other
$\epsilon$-block has even length $h_2$. If $v$ has one letter $\epsilon+1$,
then $\epsilon+1$ is immediately after $\epsilon$-block of length $h_1$.
So, by the proof of Theorem \ref{tm:radius}, $\epsilon$-block of length $h_1$,
by changing into block $0a0\cdots 0a0$, contributes with
$h_1\cdot\frac{a}{2}-\frac{1}{2}$, letter $\epsilon+1$ contributes with
$\frac{a+1}{2}$ by changing into $0$, and $\epsilon$-block of length $h_2$
contributes with $h_2\cdot\frac{a}{2}$ by changing into block $a0\cdots a0$.
So, \begin{align*}
e(v)&= h_1\cdot\frac{a}{2}-\frac{1}{2}+\frac{a+1}{2}+ h_2\cdot\frac{a}{2}= n\cdot\frac{a}{2},
\end{align*} 
hence $v\in Z(\Pi^a_n)$. In the other case, where $v$ contains letter
$\epsilon-1$ just before $\epsilon$-block of length $h_1$, we have similar
situation. Letter $\epsilon-1$ is either first letter or comes immediately
after $\epsilon$-block of length $h_2$. In both cases it must be changed
into $a-1$, thus contributing with $\frac{a+1}{2}$. The $\epsilon$-block
of length $h_2$ contributes with $h_2\cdot\frac{a}{2}$ by changing into
$0a\cdots 0a$, and $\epsilon$-block of length $h_1$ contributes with
$h_1\cdot\frac{a}{2}-\frac{1}{2}$. Thus, we proved
$U^a_n\subseteq Z(\Pi^a_n)$. 

Now we want to prove that $Z(\Pi^a_n)\subseteq U^a_n$. Let $v\in Z(\Pi^a_n)$.
Then $e(v)=n\cdot \frac{a}{2}$. On the other hand, let $k$ denote number of
letters different from $\epsilon$. Those letters must be $\epsilon\pm 1$,
since every other letter increases the eccentricity. If $h_1,\dots, h_s$
denote the length of $\epsilon$-blocks, then
\begin{align*}
e(v)&=\sum\left(h_i\cdot\frac{a}{2}-\frac{1}{2}\left(\ceil*{\frac{h_i}{2}}-\floor*{\frac{h_i}{2}}\right)\right)+k\cdot\frac{a+1}{2}\\
&= n\cdot\frac{a}{2}-\frac{1}{2}\sum\left(\ceil*{\frac{h_i}{2}}-\floor*{\frac{h_i}{2}}\right)+\frac{k}{2}.
\end{align*}
It follows that
$\sum\left(\ceil*{\frac{h_i}{2}}-\floor*{\frac{h_i}{2}}\right)=k$, which
means that the number of $\epsilon$-blocks of odd length is $k$. If $k=0$,
then $v=\hat{\epsilon}$. If $k=1$, then there are at most two
$\epsilon$-blocks with lengths $h_1$ and $h_2$. Since $\epsilon+1$ coming
before $\epsilon$-block of odd length increases its eccentricity, we
conclude that it must come after it. In the other case, for the same reason,
the letter $\epsilon-1$ must come before $\epsilon$-block of odd length.
Since $n$ is even, the other $\epsilon$-block has even length. So,
$v\in U^a_n$. If $k>1$, then
$v=\hat{\epsilon_1}\alpha_1\cdots \hat{\epsilon_k}\alpha_k\epsilon_{k+1}$,
where $k$ blocks have odd length and one block has even length.
Since $k>1$, we have at least one substring
$\hat{\epsilon_i}\alpha_i\hat{\epsilon_j}$ with both $\epsilon$-blocks of
odd length. Since $\alpha_i$ must be $\epsilon+1$, it increases
the eccentricity of block $\hat{\epsilon_j}$. Hence, $k\leq 1$ and
$v\in U^a_n$. 
\end{proof}

Let $\epsilon=\frac{a}{2}$ and let $V^a_n\subseteq Z(\Pi^a_n)$ be the set
containing all vertices with letters $\epsilon$ or $\epsilon-1$ such that
$\epsilon-1$ can not appear after an $\epsilon$-block of odd length. 

\begin{corollary}For $a$ even, $Z(\Pi^a_n)=V^a_n$.
\end{corollary}
\begin{proof}
Let $a$ be even and let $\epsilon=\frac{a}{2}$. By Theorem \ref{tm:radius},
we have $R(\Pi^a_n)=n\cdot\frac{a}{2}$ and $\hat{\epsilon}\in Z(\Pi^a_n)$.
Let $v\in Z(\Pi^a_n)$. Minimal contribution of each letter is $\frac{a}{2}$.
For minimal contribution to be achieved, one can only change letter
$\frac{a}{2}$ into $0$ or $a$, and $\frac{a}{2}-1$ into $a-1$. For letter
$\frac{a}{2}-1$ to be changed into $a-1$, the letter just before it must be
changed into $a-1$ or $a$. So, letter $\frac{a}{2}-1$ can come only after
another letter $\frac{a}{2}-1$ or after an $\epsilon$-block of even length.
Hence, $v\in V_n$. 

Conversely, from definition of $V^a_n$, it is easy to see that
$e(v)=n\cdot\frac{a}{2}$ whenever $v\in V^a_n$, hence
$V^a_n\subset Z(\Pi^a_n)$.\end{proof} 

For example, $Z(\Pi^5_6)=\left\lbrace 122222,221222,222212,222222,222223, 222322, 232222 \right\rbrace$, $Z(\Pi^5_7)=\left\lbrace 2222222 \right\rbrace$ and $Z(\Pi^4_4)=\left\lbrace 1111,1112,1122,1221,1222,2211,2212,2222 \right\rbrace$.

\begin{remark} $|Z(\Pi^a_n)|=\begin{cases}
    1 & \textup{for } a \textup{ and } n \textup{ odd, }\\
    n+1 & \textup{for } a \textup{ odd and } n \textup{ even,} \\
    F_{n+2} & \textup{for } a \textup{ even.}
  \end{cases}$ \end{remark} 

Our final result in this section is about diameters of metallic cubes.
 
\begin{theorem}\label{tm:diameter} For $a\geq 1$ and $n\geq 0$ we have
\begin{align}\label{eq:diameter}
d(\Pi^a_n)=an-1.
\end{align}
The periphery consists of two vertices,
$0a0a\cdots 0a(0)$ and $(a-1)0a\cdots a0(a)$.
\end{theorem}
\begin{proof}
Consider the vertex $v=0a0a\cdots 0$ if $n$ is odd or $v=0a0a\cdots a$ if
$n$ is even. The first letter contributes with $a-1$ by changing into $a-1$,
and every other letter with $a$ by changing into $0$ or $a$. We have
$e(v)=an-1$. It is clear that no other vertex has eccentricity greater then
$an-1$, because a string can not begin with a letter $a$, thus maximal
contribution of the first letter is also achieved. So for $n$ even we
have $P(\Pi^a_n)=\left\lbrace 0a\cdots 0a, (a-1)0a\cdots a0\right\rbrace$,
and for $n$ odd $P(\Pi^a_n)=\left\lbrace 0a\cdot a0, (a-1)0a\cdots 0a\right\rbrace$. 
\end{proof}   
 
\begin{figure}[h!] \centering
\begin{tikzpicture}[scale=0.75]
\tikzmath{\x1 = 0.35; \y1 =-0.05; \z1=180; \w1=0.2; \xs=-8; \ys=0; \yss=-5;
\x2 = \x1 + 1; \y2 =\y1 +3; } 
\small
\node [label={[label distance=\y1 cm]\z1: $6$},circle,fill=blue,draw=black,scale=\x1](A1) at (0,4) {};
\node [label={[label distance=\y1 cm]\z1: $6$},circle,fill=blue,draw=black,scale=\x1](A2) at (0,5) {};
\node [label={[label distance=\y1 cm]\z1: $7$},circle,fill=blue,draw=black,scale=\x1](A3) at (0,6) {};
\node [label={[label distance=\y1 cm]\z1: $5$},circle,fill=blue,draw=black,scale=\x1](A4) at (-1.5,2.5) {};
\node [label={[label distance=\y1 cm]\z1: $5$},circle,fill=blue,draw=black,scale=\x1](A5) at (-1.5,3.5) {};
\node [label={[label distance=\y1 cm]\z1: $6$},circle,fill=blue,draw=black,scale=\x1](A6) at (-1.5,4.5) {};
\node [label={[label distance=\y1 cm]\z1: $6$},circle,fill=blue,draw=black,scale=\x1](A7) at (-3,1) {};
\node [label={[label distance=\y1 cm]\z1: $5$},circle,fill=blue,draw=black,scale=\x1](A8) at (-3,2) {};
\node [label={[label distance=\y1 cm]\z1: $6$},circle,fill=blue,draw=black,scale=\x1](A9) at (-3,3) {};
\node [label={[label distance=\y1 cm]\z1: $6$},circle,fill=blue,draw=black,scale=\x1](A10) at (1.5,2.5) {};
\node [label={[label distance=\y1 cm]\z1: $5$},circle,fill=blue,draw=black,scale=\x1](A11) at (1.5,3.5) {};
\node [label={[label distance=\y1 cm]\z1: $6$},circle,fill=blue,draw=black,scale=\x1](A12) at (1.5,4.5) {};
\node [label={[label distance=\y1 cm]\z1: $5$},circle,fill=blue,draw=black,scale=\x1](A13) at (0,1) {};
\node [label={[label distance=\y1 cm]\z1: $4$},circle,fill=blue,draw=black,scale=\x1](A14) at (0,2) {};
\node [label={[label distance=\y1 cm]\z1: $5$},circle,fill=blue,draw=black,scale=\x1](A15) at (0,3) {};
\node [label={[label distance=\y1 cm]\z1: $6$},circle,fill=blue,draw=black,scale=\x1](A16) at (-1.5,-0.5) {};
\node [label={[label distance=\y1 cm]\z1: $5$},circle,fill=blue,draw=black,scale=\x1](A17) at (-1.5,0.5) {};
\node [label={[label distance=\y1 cm]\z1: $5$},circle,fill=blue,draw=black,scale=\x1](A18) at (-1.5,1.5) {};
\node [label={[label distance=\y1 cm]\z1: $7$},circle,fill=blue,draw=black,scale=\x1](A19) at (3,1) {};
\node [label={[label distance=\y1 cm]\z1: $6$},circle,fill=blue,draw=black,scale=\x1](A20) at (3,2) {};
\node [label={[label distance=\y1 cm]\z1: $7$},circle,fill=blue,draw=black,scale=\x1](A21) at (3,3) {};
\node [label={[label distance=\y1 cm]\z1: $6$},circle,fill=blue,draw=black,scale=\x1](A22) at (1.5,-0.5) {};
\node [label={[label distance=\y1 cm]\z1: $5$},circle,fill=blue,draw=black,scale=\x1](A23) at (1.5,0.5) {};
\node [label={[label distance=\y1 cm]\z1: $6$},circle,fill=blue,draw=black,scale=\x1](A24) at (1.5,1.5) {};
\node [label={[label distance=\y1 cm]\z1: $7$},circle,fill=blue,draw=black,scale=\x1](A25) at (0,-2) {};
\node [label={[label distance=\y1 cm]\z1: $6$},circle,fill=blue,draw=black,scale=\x1](A26) at (0,-1) {};
\node [label={[label distance=\y1 cm]\z1: $6$},circle,fill=blue,draw=black,scale=\x1](A27) at (0,0) {};
\node [label={[label distance=\y1 cm]\z1: $7$},circle,fill=blue,draw=black,scale=\x1](A28) at (-4.5,-0.5) {};
\node [label={[label distance=\y1 cm]\z1: $7$},circle,fill=blue,draw=black,scale=\x1](A29) at (-3,-2) {};
\node [label={[label distance=\y1 cm]\z1: $8$},circle,fill=blue,draw=black,scale=\x1](A30) at (-1.5,-3.5) {};
\node [label={[label distance=\y1 cm]\z1: $8$},circle,fill=blue,draw=black,scale=\x1](A31) at (-1.5,7.5) {};
\node [label={[label distance=\y1 cm]\z1: $7$},circle,fill=blue,draw=black,scale=\x1](A32) at (-3,6) {};
\node [label={[label distance=\y1 cm]\z1: $7$},circle,fill=blue,draw=black,scale=\x1](A33) at (-4.5,4.5) {};

\draw [line width=\w1 mm] (A1)--(A2)--(A3) (A4)--(A5)--(A6) (A7)--(A8)--(A9) (A10)--(A11)--(A12) (A13)--(A14)--(A15) (A16)--(A17)--(A18) (A19)--(A20)--(A21) (A22)--(A23)--(A24) (A25)--(A26)--(A27) (A1)--(A4)--(A7)--(A28) (A2)--(A5)--(A8) (A3)--(A6)--(A9) (A10)--(A13)--(A16)--(A29) (A11)--(A14)--(A17) (A12)--(A15)--(A18) (A19)--(A22)--(A25)--(A30) (A20)--(A23)--(A26) (A21)--(A24)--(A27) (A1)--(A10)--(A19) (A2)--(A11)--(A20) (A31)--(A3)--(A12)--(A21) (A4)--(A13)--(A22) (A5)--(A14)--(A23) (A32)--(A6)--(A15)--(A24) (A7)--(A16)--(A25) (A8)--(A17)--(A26) (A33)--(A9)--(A18)--(A27)  (A28)--(A29)--(A30) (A31)--(A32)--(A33);

\node [label={[label distance=\y1 cm]\z1: $4$},circle,fill=blue,draw=black,scale=\x1](B1) at (0+\xs,4+\ys) {};
\node [label={[label distance=\y1 cm]\z1: $4$},circle,fill=blue,draw=black,scale=\x1](B2) at (0+\xs,5+\ys) {};
\node [label={[label distance=\y1 cm]\z1: $5$},circle,fill=blue,draw=black,scale=\x1](B3) at (0+\xs,6+\ys) {};
\node [label={[label distance=\y1 cm]\z1: $3$},circle,fill=blue,draw=black,scale=\x1](B4) at (-1.5+\xs,2.5+\ys) {};
\node [label={[label distance=\y1 cm]\z1: $3$},circle,fill=blue,draw=black,scale=\x1](B5) at (-1.5+\xs,3.5+\ys) {};
\node [label={[label distance=\y1 cm]\z1: $4$},circle,fill=blue,draw=black,scale=\x1](B6) at (-1.5+\xs,4.5+\ys) {};
\node [label={[label distance=\y1 cm]\z1: $4$},circle,fill=blue,draw=black,scale=\x1](B7) at (-3+\xs,1+\ys) {};
\node [label={[label distance=\y1 cm]\z1: $3$},circle,fill=blue,draw=black,scale=\x1](B8) at (-3+\xs,2+\ys) {};
\node [label={[label distance=\y1 cm]\z1: $4$},circle,fill=blue,draw=black,scale=\x1](B9) at (-3+\xs,3+\ys) {};
\node [label={[label distance=\y1 cm]\z1: $5$},circle,fill=blue,draw=black,scale=\x1](B28) at (-4.5+\xs,-0.5+\ys) {};

\draw [line width=\w1 mm] (B1)--(B2)--(B3) (B4)--(B5)--(B6) (B7)--(B8)--(B9)  (B1)--(B4)--(B7)--(B28) (B2)--(B5)--(B8) (B3)--(B6)--(B9); 

\node [label={[label distance=\y1 cm]\z1: $2$},circle,fill=blue,draw=black,scale=\x1](C1) at (0+\xs,4+\yss) {};
\node [label={[label distance=\y1 cm]\z1: $1$},circle,fill=blue,draw=black,scale=\x1](C4) at (-1.5+\xs,2.5+\yss) {};
\node [label={[label distance=\y1 cm]\z1: $2$},circle,fill=blue,draw=black,scale=\x1](C7) at (-3+\xs,1+\yss) {};

\draw [line width=\w1 mm] (C1)--(C4)--(C7); 

\end{tikzpicture} 
\caption{Graphs $\Pi_n^{3}$ for $n=1,2,3$ with eccentricity of every vertex.} \label{fig:Pi_3_ecentricity}
\end{figure}
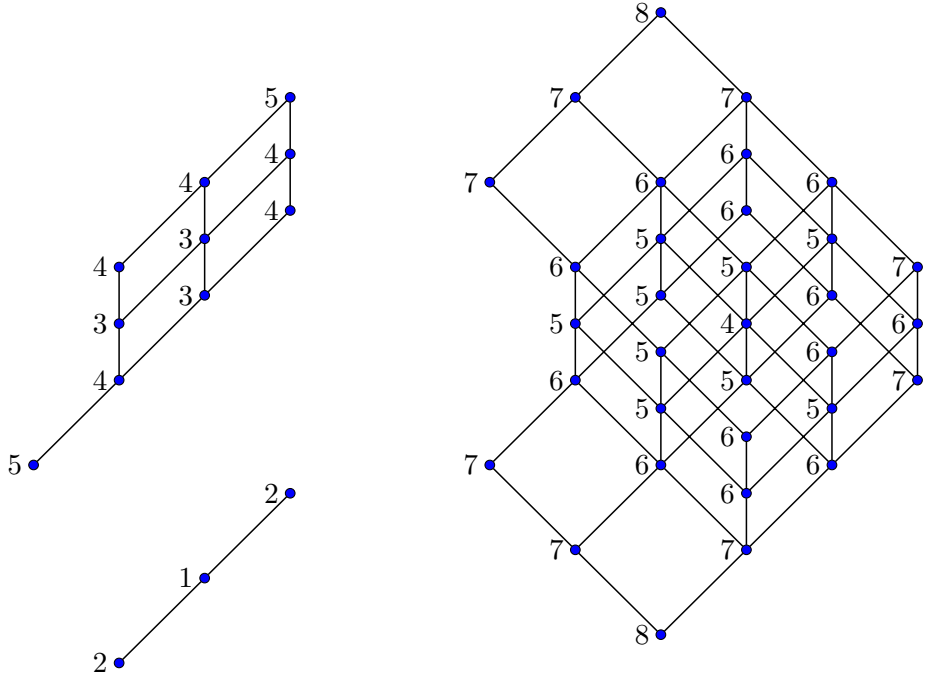

\section{Hamiltonicity}

A \textit{Hamiltonian path} in a graph $G$ is a path that visits each vertex
of $G$ exactly once. A {\em Hamiltonian cycle} in $G$ is a cycle that visits
each vertex exactly once. A graph containing a Hamiltonian cycle is called
{\em Hamiltonian graph}. In this section we examine some Hamiltonicity-related
properties of metallic cubes.

In their paper \cite{Hamilton}, Liu et al. showed that all (generalized)
Fibonacci cubes contain a Hamiltonian path. In the next theorem we show that
every metallic cube contains a Hamiltonian path.   

\begin{theorem}\label{tm:Hamil} For all $a,n\geq 1$, $\Pi^a_n$ contains a
Hamiltonian path.
\end{theorem}
\begin{proof}
Proof is by induction on $n$. We first consider the case of $a$ even. We
wish to prove that, not only $\Pi^a_n$ contains Hamiltonian path, but it
contains a path starting at $0a\cdots0a(0)$ and ending at $(a-1)\cdots(a-1)$.
For $n=1$, graph $\Pi^a_1$ is path on $a$ vertices and the claim is
obviously valid. For $n=2$, graph $\Pi^a_2$ is an $a\times a$ grid with
addition of one vertex $0a$ being adjacent to the vertex $0(a-1)$. So, for
even $a$, a Hamiltonian path is
$0a\to\cdots\to 00\to 10\to \cdots \to 1(a-1) \to\cdots\to (a-1)(a-1)$.
Figure \ref{fig:Hamiltonian path n=2} shows a Hamiltonian path in $\Pi^2_2$
and in $\Pi^3_2$. Now suppose that $\Pi^a_n$ contains a Hamiltonian path for
all $k<n$. By assumption, let $H_{n-1}$ and $H_{n-2}$ denote Hamiltonian
paths in $\Pi^a_{n-1}$ and $\Pi^a_{n-2}$, respectively. Paths $H_{n-1}$ and
$H_{n-2}$ have starting point $0a0a\cdots0a(0)$ and end-point at
$(a-1)\cdots(a-1)$. Let $0H_{n-1},\dots,(a-1)H_{n-1}$  and  $0aH_{n-2}$
denote paths is $\Pi^a_n$ obtained from  $H_{n-1}$ and $H_{n-2}$ by appending
$0,1,\dots,a-1$ or $0a$ on each vertex and let
$0\overline{H}_{n-1},\dots,(a-1)\overline{H}_{n-1}$  and 
$0a\overline{H}_{n-2}$ denote their reverse paths, i.e., paths starting at
$(a-1)\cdots(a-1)$ and ending at $0a\cdots0a(0)$. Path $0aH_{n-2}$ ends at
$0a(a-1)\cdots(a-1)$ and path $0\overline{H}_{n-1}$ begins at
$0(a-1)\cdots(a-1)$. Since vertices $0a(a-1)\cdots(a-1)$ and
$0(a-1)\cdots(a-1)$ are adjacent in $\Pi^a_n$, those two paths can be
concatenated. Hence, by concatenating
$0aH_{n-2}, 0\overline{H}_{n-1}, 1H_{n-1},\cdots, (a-2)\overline{H}_{n-1}$
and $(a-1)H_{n-1}$, we obtain a Hamiltonian path in $\Pi^a_n$ that begins at
$0a\cdots0a(0)$ and ends at $(a-1)\cdots(a-1)$.

Let $a$ now be odd. Now we wish to prove that $\Pi^a_n$ contains a Hamiltonian
path with endpoints at $0a(a-1)\cdots0a(a-1)$ and $(a-1)0a\cdots(a-1)0a$. For
example, if $a=3$ and $n=1,2,3,4,5$, the end points are $0$, $2$, $03$, $20$,
$032$, $203$, $0320$, $2032$, $03203$ and $20320$, respectively. The rest
of the proof is carried out for $n$ where $n\mod 3=0$. Cases for $n\mod 3=1$
or $2$ are similar, so we omit the details. The base of induction is easy to
check, similar as for even $a$. By assumption, let $H_{n-1}$ and $H_{n-2}$
denote Hamiltonian paths in $\Pi^a_{n-1}$ and $\Pi^a_{n-2}$, respectively.
Path $H_{n-1}$ has the starting point $0a(a-1)0a(a-1)\cdots0a$ and the
end-point $(a-1)0a(a-1)0a\cdots(a-1)0$, and path $H_{n-2}$ starts at
$0a(a-1)0a(a-1)\cdots 0$ and ends at $(a-1)0a(a-1)0a\cdots(a-1)$. Let
$0H_{n-1},\dots,(a-1)H_{n-1}$ and $0aH_{n-2}$ denote paths in $\Pi^a_n$
obtained from $H_{n-1}$ and $H_{n-2}$ by appending $0,1,\dots,a-1$ and $0a$,
respectively, on each vertex and let
$0\overline{H}_{n-1},\dots,(a-1)\overline{H}_{n-1}$ 
and $0a\overline{H}_{n-2}$ denote their reverse paths, with start- and
end-points switched. Path $0aH_{n-2}$ ends at $0a(a-1)\cdots 0a(a-1)$ and
path $0\overline{H}_{n-1}$ begins at $0(a-1)0a\cdots(a-1)0$. Since vertices
$0a(a-1)\cdots 0a(a-1)$ and $0(a-1)0a\cdots(a-1)0$ are adjacent in $\Pi^a_n$,
those two paths can be concatenated. Hence, by concatenating
$0aH_{n-2}, 0\overline{H}_{n-1}, 1H_{n-1},2\overline{H}_{n-1} \cdots, (a-2)H_{n-1}$
and $(a-1)\overline{H}_{n-1}$, we obtain a Hamiltonian path in $\Pi^a_n$ that
begins at $0a(a-1)\cdots0a(a-1)$ and ends at $(a-1)0a\cdots(a-1)0a$.  
\end{proof}
As an example, Figure \ref{fig:Hamiltonian path n=4} shows Hamiltonian paths
in $\Pi^2_4$ and in $\Pi^3_4$.

\begin{figure} \centering  \begin{tikzpicture}[scale=1]
\tikzmath{\x1 = 0.35; \y1 =-0.05; \z1=180; \w1=0.2; \xs=-8; \ys=0; \yss=-5;
\x2 = \x1 + 1; \y2 =\y1 +3; } 
\small
\node [label={[label distance=\y1 cm]\z1: $01$},circle,fill=blue,draw=black,scale=\x1](A14) at (-3,-1) {};
\node [label={[label distance=\y1 cm]\z1: $00$},circle,fill=blue,draw=black,scale=\x1](A15) at (-2,0) {};
\node [label={[label distance=\y1 cm]\z1: $11$},circle,fill=blue,draw=black,scale=\x1](A18) at (-3,0) {};
\node [label={[label distance=\y1 cm]\z1: $10$},circle,fill=blue,draw=black,scale=\x1](A19) at (-2,1) {};
\node [label={[label distance=\y1 cm]\z1: $02$},circle,fill=blue,draw=black,scale=\x1](A22) at (-4,-2) {};

\draw [line width=\w1 mm] 
(A18)--(A19) (A18)--(A14)--(A22) (A15)--(A19); 

\draw [line width={\w1+0.5 mm} ,red] (A22)--(A14)--(A15)--(A19)--(A18); 
\end{tikzpicture} \begin{tikzpicture}[scale=0.75]
\tikzmath{\x1 = 0.35; \y1 =-0.05; \z1=180; \w1=0.2; \xs=-8; \ys=0; \yss=-5;
\x2 = \x1 + 1; \y2 =\y1 +3; } 
\small
{\node [label={[label distance=\y1 cm]\z1: $00$},circle,fill=blue,draw=black,scale=\x1](B1) at (0+\xs,4+\ys) {};
\node [label={[label distance=\y1 cm]\z1: $10$},circle,fill=blue,draw=black,scale=\x1](B2) at (0+\xs,5+\ys) {};
\node [label={[label distance=\y1 cm]\z1: $20$},circle,fill=blue,draw=black,scale=\x1](B3) at (0+\xs,6+\ys) {};
\node [label={[label distance=\y1 cm]\z1: $01$},circle,fill=blue,draw=black,scale=\x1](B4) at (-1.5+\xs,2.5+\ys) {};
\node [label={[label distance=\y1 cm]\z1: $11$},circle,fill=blue,draw=black,scale=\x1](B5) at (-1.5+\xs,3.5+\ys) {};
\node [label={[label distance=\y1 cm]\z1: $21$},circle,fill=blue,draw=black,scale=\x1](B6) at (-1.5+\xs,4.5+\ys) {};
\node [label={[label distance=\y1 cm]\z1: $02$},circle,fill=blue,draw=black,scale=\x1](B7) at (-3+\xs,1+\ys) {};
\node [label={[label distance=\y1 cm]\z1: $12$},circle,fill=blue,draw=black,scale=\x1](B8) at (-3+\xs,2+\ys) {};
\node [label={[label distance=\y1 cm]\z1: $22$},circle,fill=blue,draw=black,scale=\x1](B9) at (-3+\xs,3+\ys) {};
\node [label={[label distance=\y1 cm]\z1: $03$},circle,fill=blue,draw=black,scale=\x1](B28) at (-4.5+\xs,-0.5+\ys) {};

\draw [line width=\w1 mm] (B1)--(B2)--(B3) (B4)--(B5)--(B6) (B7)--(B8)--(B9)  (B1)--(B4)--(B7)--(B28) (B2)--(B5)--(B8) (B3)--(B6)--(B9); 

\draw [line width={\w1+0.5 mm},red] (B28)--(B7)--(B4)--(B1)--(B2)--(B5)--(B8)--(B9)--(B6)--(B3); }

\end{tikzpicture}

\caption{Hamiltonian path of $\Pi_2^{2}$ and $\Pi_2^{3}$.} \label{fig:Hamiltonian path n=2}
\end{figure}
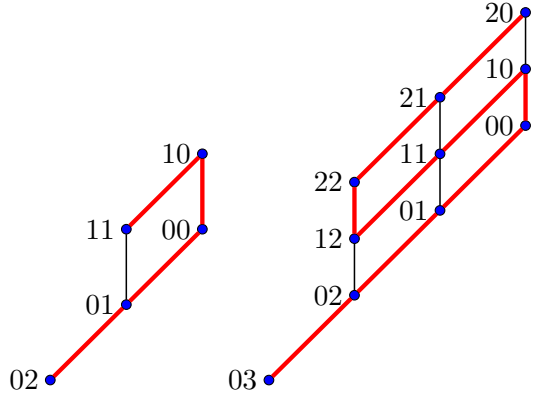

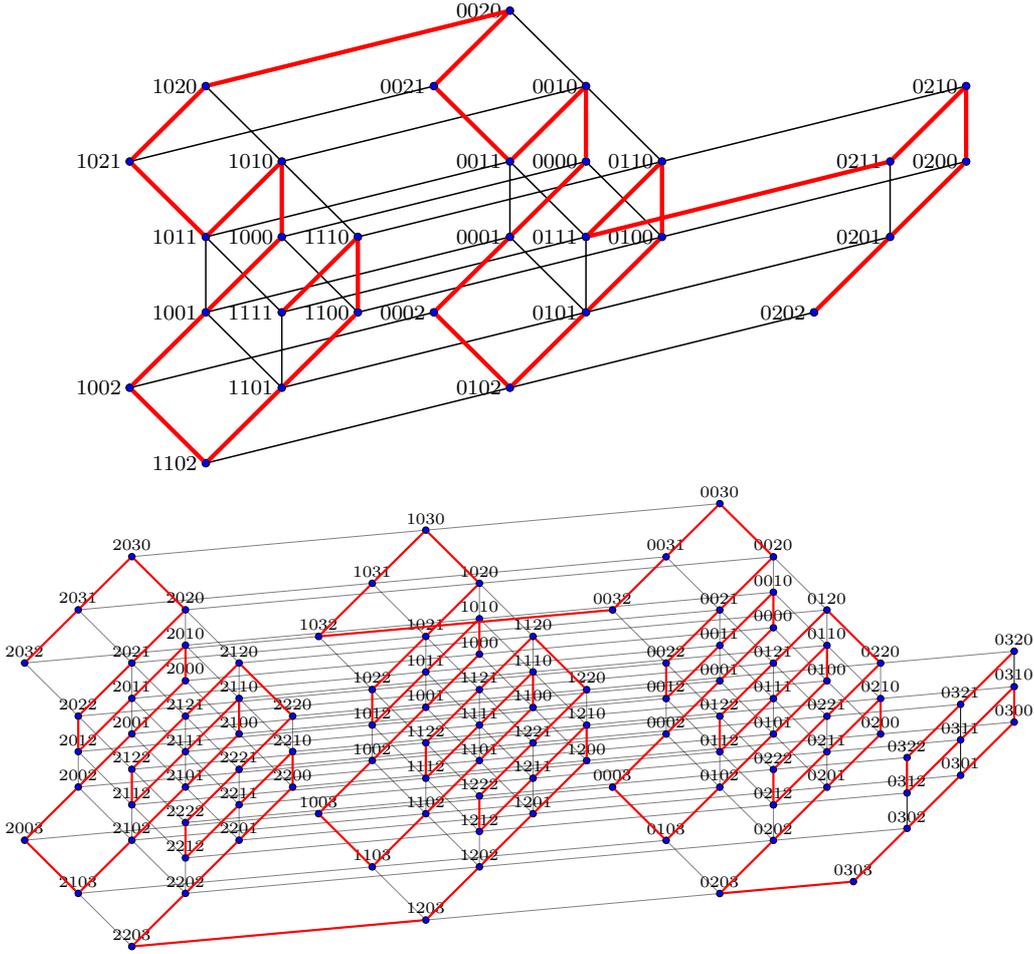
\begin{figure}[h!] \centering 
\begin{tikzpicture}[scale=1]
\tikzmath{\x1 = 0.35; \y1 =-0.05; \z1=180; \w1=0.2; \xs=-8; \ys=0; \yss=-5;
\x2 = \x1 + 1; \y2 =\y1 +3; } 
\scriptsize
\node [label={[label distance=\y1 cm]\z1: $0001$},circle,fill=blue,draw=black,scale=\x1](A1) at (0,1) {};
\node [label={[label distance=\y1 cm]\z1: $0101$},circle,fill=blue,draw=black,scale=\x1](A2) at (1,0) {};
\node [label={[label distance=\y1 cm]\z1: $0100$},circle,fill=blue,draw=black,scale=\x1](A3) at (2,1) {};
\node [label={[label distance=\y1 cm]\z1: $0000$},circle,fill=blue,draw=black,scale=\x1](A4) at (1,2) {};
\node [label={[label distance=\y1 cm]\z1: $0011$},circle,fill=blue,draw=black,scale=\x1](A5) at (0,2) {};
\node [label={[label distance=\y1 cm]\z1: $0111$},circle,fill=blue,draw=black,scale=\x1](A6) at (1,1) {};
\node [label={[label distance=\y1 cm]\z1: $0110$},circle,fill=blue,draw=black,scale=\x1](A7) at (2,2) {};
\node [label={[label distance=\y1 cm]\z1: $0010$},circle,fill=blue,draw=black,scale=\x1](A8) at (1,3) {};
\node [label={[label distance=\y1 cm]\z1: $0002$},circle,fill=blue,draw=black,scale=\x1](A9) at (-1,0) {};
\node [label={[label distance=\y1 cm]\z1: $0102$},circle,fill=blue,draw=black,scale=\x1](A10) at (0,-1) {};
\node [label={[label distance=\y1 cm]\z1: $0021$},circle,fill=blue,draw=black,scale=\x1](A11) at (-1,3) {};
\node [label={[label distance=\y1 cm]\z1: $0020$},circle,fill=blue,draw=black,scale=\x1](A12) at (0,4) {};
\node [label={[label distance=\y1 cm]\z1: $1001$},circle,fill=blue,draw=black,scale=\x1](A13) at (-4,0) {};
\node [label={[label distance=\y1 cm]\z1: $1101$},circle,fill=blue,draw=black,scale=\x1](A14) at (-3,-1) {};
\node [label={[label distance=\y1 cm]\z1: $1100$},circle,fill=blue,draw=black,scale=\x1](A15) at (-2,0) {};
\node [label={[label distance=\y1 cm]\z1: $1000$},circle,fill=blue,draw=black,scale=\x1](A16) at (-3,1) {};
\node [label={[label distance=\y1 cm]\z1: $1011$},circle,fill=blue,draw=black,scale=\x1](A17) at (-4,1) {};
\node [label={[label distance=\y1 cm]\z1: $1111$},circle,fill=blue,draw=black,scale=\x1](A18) at (-3,0) {};
\node [label={[label distance=\y1 cm]\z1: $1110$},circle,fill=blue,draw=black,scale=\x1](A19) at (-2,1) {};
\node [label={[label distance=\y1 cm]\z1: $1010$},circle,fill=blue,draw=black,scale=\x1](A20) at (-3,2) {};
\node [label={[label distance=\y1 cm]\z1: $1002$},circle,fill=blue,draw=black,scale=\x1](A21) at (-5,-1) {};
\node [label={[label distance=\y1 cm]\z1: $1102$},circle,fill=blue,draw=black,scale=\x1](A22) at (-4,-2) {};
\node [label={[label distance=\y1 cm]\z1: $1021$},circle,fill=blue,draw=black,scale=\x1](A23) at (-5,2) {};
\node [label={[label distance=\y1 cm]\z1: $1020$},circle,fill=blue,draw=black,scale=\x1](A24) at (-4,3) {};
\node [label={[label distance=\y1 cm]\z1: $0201$},circle,fill=blue,draw=black,scale=\x1](A25) at (5,1) {};
\node [label={[label distance=\y1 cm]\z1: $0200$},circle,fill=blue,draw=black,scale=\x1](A26) at (6,2) {};
\node [label={[label distance=\y1 cm]\z1: $0210$},circle,fill=blue,draw=black,scale=\x1](A27) at (6,3) {};
\node [label={[label distance=\y1 cm]\z1: $0211$},circle,fill=blue,draw=black,scale=\x1](A28) at (5,2) {};
\node [label={[label distance=\y1 cm]\z1: $0202$},circle,fill=blue,draw=black,scale=\x1](A29) at (4,0) {};

\draw [line width=\w1 mm] (A1)--(A2)--(A3)--(A4)--(A1)--(A5)--(A6)--(A7)--(A8)--(A5)--(A11)--(A12)--(A8) (A6)--(A2)--(A10)--(A9)--(A1) (A3)--(A7) (A4)--(A8) 
(A13)--(A14)--(A15)--(A16)--(A13)--(A17)--(A18)--(A19)--(A20)--(A17)--(A23)--(A24)--(A20) (A18)--(A14)--(A22)--(A21)--(A13) (A15)--(A19) (A16)--(A20)  (A25)--(A26)--(A27)--(A28)--(A25)--(A29) (A22)--(A10)--(A29) (A14)--(A2)--(A25) (A15)--(A3)--(A26) (A18)--(A6)--(A28) (A19)--(A7)--(A27)  (A21)--(A9) (A13)--(A1) (A16)--(A4) (A17)--(A5) (A20)--(A8) (A23)--(A11) (A24)--(A12); 

\draw [line width={\w1+0.5 mm} ,red] (A29)--(A25)--(A26)--(A27)--(A28)--(A6)--(A7)--(A3)--(A2)--(A10)--(A9)--(A1)--(A4)--(A8)--(A5)--(A11)--(A12)--(A24)--(A23)--(A17)--(A20)--(A16)--(A13)--(A21)--(A22)--(A14)--(A15)--(A19)--(A18)
; 

\end{tikzpicture} \begin{tikzpicture}[scale=0.47]
\tikzmath{\x1 = 0.35; \y1 =-0.05; \a1=110; \b1=10; \a2=220; \b2=20; \a3=270; \b3=24.5;   \z1=90; \w1=0.10; \xs=1; \ys=0; \yss=-5; \op=0.50;
\x2 = \x1 + 1; \y2 =\y1 +3; } 
\tiny
\node [label={[label distance=\y1 cm,]\z1: $2000$},circle,fill=blue,draw=black,scale=\x1](A1) at (0,4) {};
\node [label={[label distance=\y1 cm]\z1: $2010$},circle,fill=blue,draw=black,scale=\x1](A2) at (0,5) {};
\node [label={[label distance=\y1 cm]\z1: $2020$},circle,fill=blue,draw=black,scale=\x1](A3) at (0,6) {};
\node [label={[label distance=\y1 cm]\z1: $2001$},circle,fill=blue,draw=black,scale=\x1](A4) at (-1.5,2.5) {};
\node [label={[label distance=\y1 cm]\z1: $2011$},circle,fill=blue,draw=black,scale=\x1](A5) at (-1.5,3.5) {};
\node [label={[label distance=\y1 cm]\z1: $2021$},circle,fill=blue,draw=black,scale=\x1](A6) at (-1.5,4.5) {};
\node [label={[label distance=\y1 cm]\z1: $2002$},circle,fill=blue,draw=black,scale=\x1](A7) at (-3,1) {};
\node [label={[label distance=\y1 cm]\z1: $2012$},circle,fill=blue,draw=black,scale=\x1](A8) at (-3,2) {};
\node [label={[label distance=\y1 cm]\z1: $2022$},circle,fill=blue,draw=black,scale=\x1](A9) at (-3,3) {};
\node [label={[label distance=\y1 cm]\z1: $2100$},circle,fill=blue,draw=black,scale=\x1](A10) at (1.5,2.5) {};
\node [label={[label distance=\y1 cm]\z1: $2110$},circle,fill=blue,draw=black,scale=\x1](A11) at (1.5,3.5) {};
\node [label={[label distance=\y1 cm]\z1: $2120$},circle,fill=blue,draw=black,scale=\x1](A12) at (1.5,4.5) {};
\node [label={[label distance=\y1 cm]\z1: $2101$},circle,fill=blue,draw=black,scale=\x1](A13) at (0,1) {};
\node [label={[label distance=\y1 cm]\z1: $2111$},circle,fill=blue,draw=black,scale=\x1](A14) at (0,2) {};
\node [label={[label distance=\y1 cm]\z1: $2121$},circle,fill=blue,draw=black,scale=\x1](A15) at (0,3) {};
\node [label={[label distance=\y1 cm]\z1: $2102$},circle,fill=blue,draw=black,scale=\x1](A16) at (-1.5,-0.5) {};
\node [label={[label distance=\y1 cm]\z1: $2112$},circle,fill=blue,draw=black,scale=\x1](A17) at (-1.5,0.5) {};
\node [label={[label distance=\y1 cm]\z1: $2122$},circle,fill=blue,draw=black,scale=\x1](A18) at (-1.5,1.5) {};
\node [label={[label distance=\y1 cm]\z1: $2200$},circle,fill=blue,draw=black,scale=\x1](A19) at (3,1) {};
\node [label={[label distance=\y1 cm]\z1: $2210$},circle,fill=blue,draw=black,scale=\x1](A20) at (3,2) {};
\node [label={[label distance=\y1 cm]\z1: $2220$},circle,fill=blue,draw=black,scale=\x1](A21) at (3,3) {};
\node [label={[label distance=\y1 cm]\z1: $2201$},circle,fill=blue,draw=black,scale=\x1](A22) at (1.5,-0.5) {};
\node [label={[label distance=\y1 cm]\z1: $2211$},circle,fill=blue,draw=black,scale=\x1](A23) at (1.5,0.5) {};
\node [label={[label distance=\y1 cm]\z1: $2221$},circle,fill=blue,draw=black,scale=\x1](A24) at (1.5,1.5) {};
\node [label={[label distance=\y1 cm]\z1: $2202$},circle,fill=blue,draw=black,scale=\x1](A25) at (0,-2) {};
\node [label={[label distance=\y1 cm]\z1: $2212$},circle,fill=blue,draw=black,scale=\x1](A26) at (0,-1) {};
\node [label={[label distance=\y1 cm]\z1: $2222$},circle,fill=blue,draw=black,scale=\x1](A27) at (0,0) {};
\node [label={[label distance=\y1 cm]\z1: $2003$},circle,fill=blue,draw=black,scale=\x1](A28) at (-4.5,-0.5) {};
\node [label={[label distance=\y1 cm]\z1: $2103$},circle,fill=blue,draw=black,scale=\x1](A29) at (-3,-2) {};
\node [label={[label distance=\y1 cm]\z1: $2203$},circle,fill=blue,draw=black,scale=\x1](A30) at (-1.5,-3.5) {};
\node [label={[label distance=\y1 cm]\z1: $2030$},circle,fill=blue,draw=black,scale=\x1](A31) at (-1.5,7.5) {};
\node [label={[label distance=\y1 cm]\z1: $2031$},circle,fill=blue,draw=black,scale=\x1](A32) at (-3,6) {};
\node [label={[label distance=\y1 cm]\z1: $2032$},circle,fill=blue,draw=black,scale=\x1](A33) at (-4.5,4.5) {};

\node [label={[label distance=\y1 cm]\z1: $1000$},xshift=\a1,yshift=\b1,circle,fill=blue,draw=black,scale=\x1](A34) at (0,4) {};
\node [label={[label distance=\y1 cm]\z1: $1010$},xshift=\a1,yshift=\b1,circle,fill=blue,draw=black,scale=\x1](A35) at (0,5) {};
\node [label={[label distance=\y1 cm]\z1: $1020$},xshift=\a1,yshift=\b1,circle,fill=blue,draw=black,scale=\x1](A36) at (0,6) {};
\node [label={[label distance=\y1 cm]\z1: $1001$},xshift=\a1,yshift=\b1,circle,fill=blue,draw=black,scale=\x1](A37) at (-1.5,2.5) {};
\node [label={[label distance=\y1 cm]\z1: $1011$},xshift=\a1,yshift=\b1,circle,fill=blue,draw=black,scale=\x1](A38) at (-1.5,3.5) {};
\node [label={[label distance=\y1 cm]\z1: $1021$},xshift=\a1,yshift=\b1,circle,fill=blue,draw=black,scale=\x1](A39) at (-1.5,4.5) {};
\node [label={[label distance=\y1 cm]\z1: $1002$},xshift=\a1,yshift=\b1,circle,fill=blue,draw=black,scale=\x1](A40) at (-3,1) {};
\node [label={[label distance=\y1 cm]\z1: $1012$},xshift=\a1,yshift=\b1,circle,fill=blue,draw=black,scale=\x1](A41) at (-3,2) {};
\node [label={[label distance=\y1 cm]\z1: $1022$},xshift=\a1,yshift=\b1,circle,fill=blue,draw=black,scale=\x1](A42) at (-3,3) {};
\node [label={[label distance=\y1 cm]\z1: $1100$},xshift=\a1,yshift=\b1,circle,fill=blue,draw=black,scale=\x1](A43) at (1.5,2.5) {};
\node [label={[label distance=\y1 cm]\z1: $1110$},xshift=\a1,yshift=\b1,circle,fill=blue,draw=black,scale=\x1](A44) at (1.5,3.5) {};
\node [label={[label distance=\y1 cm]\z1: $1120$},xshift=\a1,yshift=\b1,circle,fill=blue,draw=black,scale=\x1](A45) at (1.5,4.5) {};
\node [label={[label distance=\y1 cm]\z1: $1101$},xshift=\a1,yshift=\b1,circle,fill=blue,draw=black,scale=\x1](A46) at (0,1) {};
\node [label={[label distance=\y1 cm]\z1: $1111$},xshift=\a1,yshift=\b1,circle,fill=blue,draw=black,scale=\x1](A47) at (0,2) {};
\node [label={[label distance=\y1 cm]\z1: $1121$},xshift=\a1,yshift=\b1,circle,fill=blue,draw=black,scale=\x1](A48) at (0,3) {};
\node [label={[label distance=\y1 cm]\z1: $1102$},xshift=\a1,yshift=\b1,circle,fill=blue,draw=black,scale=\x1](A49) at (-1.5,-0.5) {};
\node [label={[label distance=\y1 cm]\z1: $1112$},xshift=\a1,yshift=\b1,circle,fill=blue,draw=black,scale=\x1](A50) at (-1.5,0.5) {};
\node [label={[label distance=\y1 cm]\z1: $1122$},xshift=\a1,yshift=\b1,circle,fill=blue,draw=black,scale=\x1](A51) at (-1.5,1.5) {};
\node [label={[label distance=\y1 cm]\z1: $1200$},xshift=\a1,yshift=\b1,circle,fill=blue,draw=black,scale=\x1](A52) at (3,1) {};
\node [label={[label distance=\y1 cm]\z1: $1210$},xshift=\a1,yshift=\b1,circle,fill=blue,draw=black,scale=\x1](A53) at (3,2) {};
\node [label={[label distance=\y1 cm]\z1: $1220$},xshift=\a1,yshift=\b1,circle,fill=blue,draw=black,scale=\x1](A54) at (3,3) {};
\node [label={[label distance=\y1 cm]\z1: $1201$},xshift=\a1,yshift=\b1,circle,fill=blue,draw=black,scale=\x1](A55) at (1.5,-0.5) {};
\node [label={[label distance=\y1 cm]\z1: $1211$},xshift=\a1,yshift=\b1,circle,fill=blue,draw=black,scale=\x1](A56) at (1.5,0.5) {};
\node [label={[label distance=\y1 cm]\z1: $1221$},xshift=\a1,yshift=\b1,circle,fill=blue,draw=black,scale=\x1](A57) at (1.5,1.5) {};
\node [label={[label distance=\y1 cm]\z1: $1202$},xshift=\a1,yshift=\b1,circle,fill=blue,draw=black,scale=\x1](A58) at (0,-2) {};
\node [label={[label distance=\y1 cm]\z1: $1212$},xshift=\a1,yshift=\b1,circle,fill=blue,draw=black,scale=\x1](A59) at (0,-1) {};
\node [label={[label distance=\y1 cm]\z1: $1222$},xshift=\a1,yshift=\b1,circle,fill=blue,draw=black,scale=\x1](A60) at (0,0) {};
\node [label={[label distance=\y1 cm]\z1: $1003$},xshift=\a1,yshift=\b1,circle,fill=blue,draw=black,scale=\x1](A61) at (-4.5,-0.5) {};
\node [label={[label distance=\y1 cm]\z1: $1103$},xshift=\a1,yshift=\b1,circle,fill=blue,draw=black,scale=\x1](A62) at (-3,-2) {};
\node [label={[label distance=\y1 cm]\z1: $1203$},xshift=\a1,yshift=\b1,circle,fill=blue,draw=black,scale=\x1](A63) at (-1.5,-3.5) {};
\node [label={[label distance=\y1 cm]\z1: $1030$},xshift=\a1,yshift=\b1,circle,fill=blue,draw=black,scale=\x1](A64) at (-1.5,7.5) {};
\node [label={[label distance=\y1 cm]\z1: $1031$},xshift=\a1,yshift=\b1,circle,fill=blue,draw=black,scale=\x1](A65) at (-3,6) {};
\node [label={[label distance=\y1 cm]\z1: $1032$},xshift=\a1,yshift=\b1,circle,fill=blue,draw=black,scale=\x1](A66) at ({-4.5},4.5) {};

\node [label={[label distance=\y1 cm]\z1: $0000$},xshift=\a2,yshift=\b2,circle,fill=blue,draw=black,scale=\x1](A67) at (0,4) {};
\node [label={[label distance=\y1 cm]\z1: $0010$},xshift=\a2,yshift=\b2,circle,fill=blue,draw=black,scale=\x1](A68) at (0,5) {};
\node [label={[label distance=\y1 cm]\z1: $0020$},xshift=\a2,yshift=\b2,circle,fill=blue,draw=black,scale=\x1](A69) at (0,6) {};
\node [label={[label distance=\y1 cm]\z1: $0001$},xshift=\a2,yshift=\b2,circle,fill=blue,draw=black,scale=\x1](A70) at (-1.5,2.5) {};
\node [label={[label distance=\y1 cm]\z1: $0011$},xshift=\a2,yshift=\b2,circle,fill=blue,draw=black,scale=\x1](A71) at (-1.5,3.5) {};
\node [label={[label distance=\y1 cm]\z1: $0021$},xshift=\a2,yshift=\b2,circle,fill=blue,draw=black,scale=\x1](A72) at (-1.5,4.5) {};
\node [label={[label distance=\y1 cm]\z1: $0002$},xshift=\a2,yshift=\b2,circle,fill=blue,draw=black,scale=\x1](A73) at (-3,1) {};
\node [label={[label distance=\y1 cm]\z1: $0012$},xshift=\a2,yshift=\b2,circle,fill=blue,draw=black,scale=\x1](A74) at (-3,2) {};
\node [label={[label distance=\y1 cm]\z1: $0022$},xshift=\a2,yshift=\b2,circle,fill=blue,draw=black,scale=\x1](A75) at (-3,3) {};
\node [label={[label distance=\y1 cm]\z1: $0100$},xshift=\a2,yshift=\b2,circle,fill=blue,draw=black,scale=\x1](A76) at (1.5,2.5) {};
\node [label={[label distance=\y1 cm]\z1: $0110$},xshift=\a2,yshift=\b2,circle,fill=blue,draw=black,scale=\x1](A77) at (1.5,3.5) {};
\node [label={[label distance=\y1 cm]\z1: $0120$},xshift=\a2,yshift=\b2,circle,fill=blue,draw=black,scale=\x1](A78) at (1.5,4.5) {};
\node [label={[label distance=\y1 cm]\z1: $0101$},xshift=\a2,yshift=\b2,circle,fill=blue,draw=black,scale=\x1](A79) at (0,1) {};
\node [label={[label distance=\y1 cm]\z1: $0111$},xshift=\a2,yshift=\b2,circle,fill=blue,draw=black,scale=\x1](A80) at (0,2) {};
\node [label={[label distance=\y1 cm]\z1: $0121$},xshift=\a2,yshift=\b2,circle,fill=blue,draw=black,scale=\x1](A81) at (0,3) {};
\node [label={[label distance=\y1 cm]\z1: $0102$},xshift=\a2,yshift=\b2,circle,fill=blue,draw=black,scale=\x1](A82) at (-1.5,-0.5) {};
\node [label={[label distance=\y1 cm]\z1: $0112$},xshift=\a2,yshift=\b2,circle,fill=blue,draw=black,scale=\x1](A83) at (-1.5,0.5) {};
\node [label={[label distance=\y1 cm]\z1: $0122$},xshift=\a2,yshift=\b2,circle,fill=blue,draw=black,scale=\x1](A84) at (-1.5,1.5) {};
\node [label={[label distance=\y1 cm]\z1: $0200$},xshift=\a2,yshift=\b2,circle,fill=blue,draw=black,scale=\x1](A85) at (3,1) {};
\node [label={[label distance=\y1 cm]\z1: $0210$},xshift=\a2,yshift=\b2,circle,fill=blue,draw=black,scale=\x1](A86) at (3,2) {};
\node [label={[label distance=\y1 cm]\z1: $0220$},xshift=\a2,yshift=\b2,circle,fill=blue,draw=black,scale=\x1](A87) at (3,3) {};
\node [label={[label distance=\y1 cm]\z1: $0201$},xshift=\a2,yshift=\b2,circle,fill=blue,draw=black,scale=\x1](A88) at (1.5,-0.5) {};
\node [label={[label distance=\y1 cm]\z1: $0211$},xshift=\a2,yshift=\b2,circle,fill=blue,draw=black,scale=\x1](A89) at (1.5,0.5) {};
\node [label={[label distance=\y1 cm]\z1: $0221$},xshift=\a2,yshift=\b2,circle,fill=blue,draw=black,scale=\x1](A90) at (1.5,1.5) {};
\node [label={[label distance=\y1 cm]\z1: $0202$},xshift=\a2,yshift=\b2,circle,fill=blue,draw=black,scale=\x1](A91) at (0,-2) {};
\node [label={[label distance=\y1 cm]\z1: $0212$},xshift=\a2,yshift=\b2,circle,fill=blue,draw=black,scale=\x1](A92) at (0,-1) {};
\node [label={[label distance=\y1 cm]\z1: $0222$},xshift=\a2,yshift=\b2,circle,fill=blue,draw=black,scale=\x1](A93) at (0,0) {};
\node [label={[label distance=\y1 cm]\z1: $0003$},xshift=\a2,yshift=\b2,circle,fill=blue,draw=black,scale=\x1](A94) at (-4.5,-0.5) {};
\node [label={[label distance=\y1 cm]\z1: $0103$},xshift=\a2,yshift=\b2,circle,fill=blue,draw=black,scale=\x1](A95) at (-3,-2) {};
\node [label={[label distance=\y1 cm]\z1: $0203$},xshift=\a2,yshift=\b2,circle,fill=blue,draw=black,scale=\x1](A96) at (-1.5,-3.5) {};
\node [label={[label distance=\y1 cm]\z1: $0030$},xshift=\a2,yshift=\b2,circle,fill=blue,draw=black,scale=\x1](A97) at (-1.5,7.5) {};
\node [label={[label distance=\y1 cm]\z1: $0031$},xshift=\a2,yshift=\b2,circle,fill=blue,draw=black,scale=\x1](A98) at (-3,6) {};
\node [label={[label distance=\y1 cm]\z1: $0032$},xshift=\a2,yshift=\b2,circle,fill=blue,draw=black,scale=\x1](A99) at ({-4.5},4.5) {};

\node [label={[label distance=\y1 cm]\z1: $0300$},xshift=\a3,yshift=\b3,circle,fill=blue,draw=black,scale=\x1](A100) at (3,1) {};
\node [label={[label distance=\y1 cm]\z1: $0310$},xshift=\a3,yshift=\b3,circle,fill=blue,draw=black,scale=\x1](A101) at (3,2) {};
\node [label={[label distance=\y1 cm]\z1: $0320$},xshift=\a3,yshift=\b3,circle,fill=blue,draw=black,scale=\x1](A102) at (3,3) {};
\node [label={[label distance=\y1 cm]\z1: $0301$},xshift=\a3,yshift=\b3,circle,fill=blue,draw=black,scale=\x1](A103) at (1.5,-0.5) {};
\node [label={[label distance=\y1 cm]\z1: $0311$},xshift=\a3,yshift=\b3,circle,fill=blue,draw=black,scale=\x1](A104) at (1.5,0.5) {};
\node [label={[label distance=\y1 cm]\z1: $0321$},xshift=\a3,yshift=\b3,circle,fill=blue,draw=black,scale=\x1](A105) at (1.5,1.5) {};
\node [label={[label distance=\y1 cm]\z1: $0302$},xshift=\a3,yshift=\b3,circle,fill=blue,draw=black,scale=\x1](A106) at (0,-2) {};
\node [label={[label distance=\y1 cm]\z1: $0312$},xshift=\a3,yshift=\b3,circle,fill=blue,draw=black,scale=\x1](A107) at (0,-1) {};
\node [label={[label distance=\y1 cm]\z1: $0322$},xshift=\a3,yshift=\b3,circle,fill=blue,draw=black,scale=\x1](A108) at (0,0) {};
\node [label={[label distance=\y1 cm]\z1: $0303$},xshift=\a3,yshift=\b3,circle,fill=blue,draw=black,scale=\x1](A109) at (-1.5,-3.5) {};

\draw [line width=\w1 mm,opacity=\op] (A1)--(A34)--(A67) (A2)--(A35)--(A68) (A3)--(A36)--(A69) (A4)--(A37)--(A70) (A5)--(A38)--(A71) (A6)--(A39)--(A72) (A7)--(A40)--(A73) (A8)--(A41)--(A74) (A9)--(A42)--(A75) (A10)--(A43)--(A76) (A11)--(A44)--(A77) (A12)--(A45)--(A78) (A13)--(A46)--(A79) (A14)--(A47)--(A80) (A15)--(A48)--(A81) (A16)--(A49)--(A82) (A17)--(A50)--(A83) (A18)--(A51)--(A84) (A19)--(A52)--(A85)--(A100) (A20)--(A53)--(A86)--(A101) (A21)--(A54)--(A87)--(A102) (A22)--(A55)--(A88)--(A103) (A23)--(A56)--(A89)--(A104) (A24)--(A57)--(A90)--(A105) (A25)--(A58)--(A91)--(A106) (A26)--(A59)--(A92)--(A107) (A27)--(A60)--(A93)--(A108) (A28)--(A61)--(A94) (A29)--(A62)--(A95) (A30)--(A63)--(A96)--(A109) (A31)--(A64)--(A97) (A32)--(A65)--(A98) (A33)--(A66)--(A99);

\draw [line width=\w1 mm,opacity=\op] (A1)--(A2)--(A3) (A4)--(A5)--(A6) (A7)--(A8)--(A9) (A10)--(A11)--(A12) (A13)--(A14)--(A15) (A16)--(A17)--(A18) (A19)--(A20)--(A21) (A22)--(A23)--(A24) (A25)--(A26)--(A27) (A1)--(A4)--(A7)--(A28) (A2)--(A5)--(A8) (A3)--(A6)--(A9) (A10)--(A13)--(A16)--(A29) (A11)--(A14)--(A17) (A12)--(A15)--(A18) (A19)--(A22)--(A25)--(A30) (A20)--(A23)--(A26) (A21)--(A24)--(A27) (A1)--(A10)--(A19) (A2)--(A11)--(A20) (A31)--(A3)--(A12)--(A21) (A4)--(A13)--(A22) (A5)--(A14)--(A23) (A32)--(A6)--(A15)--(A24) (A7)--(A16)--(A25) (A8)--(A17)--(A26) (A33)--(A9)--(A18)--(A27)  (A28)--(A29)--(A30) (A31)--(A32)--(A33); 

\draw [line width=\w1 mm,opacity=\op] (A34)--(A35)--(A36) (A37)--(A38)--(A39) (A40)--(A41)--(A42) (A43)--(A44)--(A45) (A46)--(A47)--(A48) (A49)--(A50)--(A51) (A52)--(A53)--(A54) (A55)--(A56)--(A57) (A58)--(A59)--(A60) (A34)--(A37)--(A40)--(A61) (A35)--(A38)--(A41) (A36)--(A39)--(A42) (A43)--(A46)--(A49)--(A62) (A44)--(A47)--(A50) (A45)--(A48)--(A51) (A52)--(A55)--(A58)--(A63) (A53)--(A56)--(A59) (A54)--(A57)--(A60) (A34)--(A43)--(A52) (A35)--(A44)--(A53) (A64)--(A36)--(A45)--(A54) (A37)--(A46)--(A55) (A38)--(A47)--(A56) (A65)--(A39)--(A48)--(A57) (A40)--(A49)--(A58) (A41)--(A50)--(A59) (A66)--(A42)--(A51)--(A60)  (A61)--(A62)--(A63) (A64)--(A65)--(A66); 

\draw [line width=\w1 mm,opacity=\op] (A67)--(A68)--(A69) (A70)--(A71)--(A72) (A73)--(A74)--(A75) (A76)--(A77)--(A78) (A79)--(A80)--(A81) (A82)--(A83)--(A84) (A85)--(A86)--(A87) (A88)--(A89)--(A90) (A91)--(A92)--(A93) (A67)--(A70)--(A73)--(A94) (A68)--(A71)--(A74) (A69)--(A72)--(A75) (A76)--(A79)--(A82)--(A95) (A77)--(A80)--(A83) (A78)--(A81)--(A84) (A85)--(A88)--(A91)--(A96) (A86)--(A89)--(A92) (A87)--(A90)--(A93) (A67)--(A76)--(A85) (A68)--(A77)--(A86) (A97)--(A69)--(A78)--(A87) (A70)--(A79)--(A88) (A71)--(A80)--(A89) (A98)--(A72)--(A81)--(A90) (A73)--(A82)--(A91) (A74)--(A83)--(A92) (A99)--(A75)--(A84)--(A93)  (A94)--(A95)--(A96) (A97)--(A98)--(A99); 

\draw [line width=\w1 mm] (A109)--(A106)--(A103)--(A100)  (A107)--(A104)--(A101)  (A108)--(A105)--(A102) (A103)--(A104)--(A105) (A106)--(A107)--(A108) (A100)--(A101)--(A102);

\draw [line width={\w1+0.25 mm} ,red] (A30)--(A25)--(A22)--(A19)--(A20)--(A23)--(A26)--(A27)--(A24)--(A21)--(A12)--(A15)--(A18)--(A17)--(A14)--(A11)--(A10)--(A13)--(A16)--(A29)--(A28)--(A7)--(A4)--(A1)--(A2)--(A5)--(A8)--(A9)--(A6)--(A3)--(A31)--(A32)--(A33) 
(A63)--(A58)--(A55)--(A52)--(A53)--(A56)--(A59)--(A60)--(A57)--(A54)--(A45)--(A48)--(A51)--(A50)--(A47)--(A44)--(A43)--(A46)--(A49)--(A62)--(A61)--(A40)--(A37)--(A34)--(A35)--(A38)--(A41)--(A42)--(A39)--(A36)--(A64)--(A65)--(A66) 
(A96)--(A91)--(A88)--(A85)--(A86)--(A89)--(A92)--(A93)--(A90)--(A87)--(A78)--(A81)--(A84)--(A83)--(A80)--(A77)--(A76)--(A79)--(A82)--(A95)--(A94)--(A73)--(A70)--(A67)--(A68)--(A71)--(A74)--(A75)--(A72)--(A69)--(A97)--(A98)--(A99)
(A109)--(A106)--(A103)--(A100)--(A101)--(A104)--(A107)--(A108)--(A105)--(A102) (A109)--(A96)  (A99)--(A66) (A63)--(A30);

\end{tikzpicture}

\caption{Hamiltonian paths in $\Pi_4^{2}$ (above) and in $\Pi_4^{3}$ (below).} \label{fig:Hamiltonian path n=4}
\end{figure}

For $a$ even, the numbers $s^a_{2n+1}$ are even and $s^a_{2n}$ are odd. On
the other hand, for $a$ odd, the numbers $s^a_{3n+1}$ are even, but $s^a_{3n}$
and $s^a_{3n+2}$ are odd. A simple consequence of Theorem \ref{tm:Hamil}
is the following result on the existence of perfect and semi-perfect matchings
in metallic cubes.

\begin{corollary}
Graph $\Pi^a_n$ admits a perfect matching if the number of vertices is even,
otherwise it admits a semi-perfect matching.
\end{corollary} 

Now we examine the existence of Hamiltonian cycles in metallic cubes. 
\begin{theorem}
For even $a$ and $n$ odd, graphs $\Pi^a_n$ are Hamiltonian graphs.
\end{theorem}
\begin{proof}
Let $a$ be even. Then $\Pi^a_2$ is an $a\times a$ grid with additional vertex
$0a$. It is easy to see that grids $a\times a$ contain a Hamiltonian cycle,
which implies that graph $\Pi^a_2$ contains a cycle that visits all vertices 
but one, namely, the vertex $0a$. Now we wish to construct a Hamiltonian cycle
in graphs $\Pi^a_3$. Graph $\Pi^a_3$ can be decomposed into $a$ copies of
graph $\Pi^a_2$ and one copy of graph $\Pi^a_1$. By Theorem \ref{tm:Hamil},
there exists a Hamiltonian path in each subgraph $\Pi^a_2$ and $\Pi^a_{1}$.
Let $0H_{2},1H_2,\dots,(a-1)H_2$ and $0aH_1$ denote those paths. Connecting
starting points in $0H_2$ and $1H_2$, as well as the ending points in $0H_2$
and $1H_2$ yields a cycle in $0\Pi^a_2\oplus 1\Pi^a_2$. Since the number of
copies of graph $\Pi^a_2$ is even, we can pair up those copies to obtain
$\frac{a}{2}$ cycles in the same manner. Now choose vertices
$1v,1w\in 1\Pi^a_2$ so that the edge $(1v)(1w)$ is in Hamiltonian cycle in 
$1\Pi^a_2$, and the corresponding vertices $2v$ and $2w$ in $\Pi^a_2$. By
removing edges $(1v)(1w)$ and $(2v)(2w)$ from the cycles in $1\Pi^a_2$ and
$2\Pi^2$ and by adding edges $(1v)(2v)$ and $(1w)(2w)$, we obtain a Hamiltonian
cycle for subgraph $0\Pi^a_2\oplus 1\Pi^a_2\oplus 2\Pi^a_2\oplus 3\Pi^a_2$.
This method is applicable between every two $\frac{a}{2}$ cycles. Thus, we
obtained Hamiltonian cycle in $P_a\square \Pi^a_2$. Since $\Pi^a_1$ is path
graph with even number of vertices and, by Theorem \ref{tm:Hamil} and
construction so far, the path $0(a-1)H^a_1\subset 0H^a_2$ is a part of a
Hamiltonian cycle in $P_a\square \Pi^a_2$. Hence, we can extend a Hamiltonian
cycle in $P_a\square \Pi^a_2$ to $\Pi^a_3$ by adding a zig-zag shape path
to include all vertices of $\Pi^a_1$. Figure \ref{fig:Hamiltonian cycle n=3}
shows a Hamiltonian cycle for $a=2$ and $a=4$, but it is clear that
construction can be done whenever $a$ is even.  

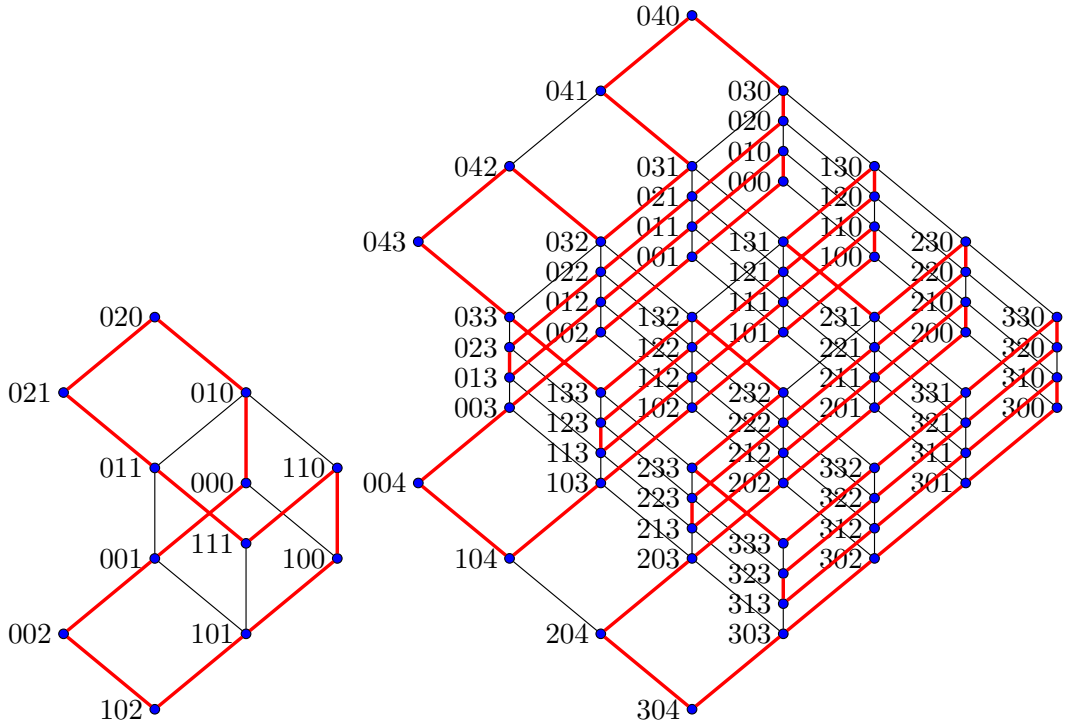
\begin{figure} \centering \begin{tikzpicture}[scale=0.20]
\tikzmath{\x1 = 0.35; \y1 =-0.05; \z1=180; \w1=0.15; \xs=6; \ys=-5; \yss=-5;
\x2 = \x1 + 1; \y2 =\y1 +3; } 
\small
\node [label={[label distance=\y1 cm]\z1: $000$},circle,fill=blue,draw=black,scale=\x1](1) at (0,30) {};
\node [label={[label distance=\y1 cm]\z1: $010$},circle,fill=blue,draw=black,scale=\x1](2) at (0,36) {};

\node [label={[label distance=\y1 cm]\z1: $001$},circle,fill=blue,draw=black,scale=\x1](3) at (-6,25) {};
\node [label={[label distance=\y1 cm]\z1: $011$},circle,fill=blue,draw=black,scale=\x1](4) at (-6,31) {};
\node [label={[label distance=\y1 cm]\z1: $002$},circle,fill=blue,draw=black,scale=\x1](5) at (-12,20) {};

\node [label={[label distance=\y1 cm]\z1: $100$},circle,fill=blue,draw=black,scale=\x1](6) at (0+\xs,30+\ys) {};
\node [label={[label distance=\y1 cm]\z1: $110$},circle,fill=blue,draw=black,scale=\x1](7) at (0+\xs,36+\ys) {};

\node [label={[label distance=\y1 cm]\z1: $101$},circle,fill=blue,draw=black,scale=\x1](8) at (-6+\xs,25+\ys) {};
\node [label={[label distance=\y1 cm]\z1: $111$},circle,fill=blue,draw=black,scale=\x1](9) at (-6+\xs,31+\ys) {};
\node [label={[label distance=\y1 cm]\z1: $102$},circle,fill=blue,draw=black,scale=\x1](10) at (-12+\xs,20+\ys) {};

\node [label={[label distance=\y1 cm]\z1: $020$},circle,fill=blue,draw=black,scale=\x1](11) at (0-\xs,36-\ys) {};
\node [label={[label distance=\y1 cm]\z1: $021$},circle,fill=blue,draw=black,scale=\x1](12) at (-6-\xs,31-\ys) {};

\draw [line width=\w1 mm] (1)--(2) (3)--(4) (7)--(6) (8)--(9)  (1)--(3)  (2)--(4) (6)--(8) (7)--(9) (1)--(6)  (2)--(7) (3)--(8)  (4)--(9)  (3)--(5)--(10)--(8) (2)--(11)--(12)--(4);

\draw [line width={\w1+0.4 mm},red] (9)--(7)--(6)--(8)--(10)--(5)--(3)--(1)--(2)--(11)--(12)--(4)--(9); 

\end{tikzpicture}  \begin{tikzpicture}[scale=0.20]
\tikzmath{\x1 = 0.35; \y1 =-0.05; \z1=180; \w1=0.15; \xs=6; \ys=-5; \yss=-5;
\x2 = \x1 + 1; \y2 =\y1 +3; } 
\small
\node [label={[label distance=\y1 cm]\z1: $000$},circle,fill=blue,draw=black,scale=\x1](1) at (0,30) {};
\node [label={[label distance=\y1 cm]\z1: $010$},circle,fill=blue,draw=black,scale=\x1](2) at (0,32) {};
\node [label={[label distance=\y1 cm]\z1: $020$},circle,fill=blue,draw=black,scale=\x1](3) at (0,34) {};
\node [label={[label distance=\y1 cm]\z1: $030$},circle,fill=blue,draw=black,scale=\x1](4) at (0,36) {};
\node [label={[label distance=\y1 cm]\z1: $001$},circle,fill=blue,draw=black,scale=\x1](5) at (-6,25) {};
\node [label={[label distance=\y1 cm]\z1: $011$},circle,fill=blue,draw=black,scale=\x1](6) at (-6,27) {};
\node [label={[label distance=\y1 cm]\z1: $021$},circle,fill=blue,draw=black,scale=\x1](7) at (-6,29) {};
\node [label={[label distance=\y1 cm]\z1: $031$},circle,fill=blue,draw=black,scale=\x1](8) at (-6,31) {};
\node [label={[label distance=\y1 cm]\z1: $002$},circle,fill=blue,draw=black,scale=\x1](9) at (-12,20) {};
\node [label={[label distance=\y1 cm]\z1: $012$},circle,fill=blue,draw=black,scale=\x1](10) at (-12,22) {};
\node [label={[label distance=\y1 cm]\z1: $022$},circle,fill=blue,draw=black,scale=\x1](11) at (-12,24) {};
\node [label={[label distance=\y1 cm]\z1: $032$},circle,fill=blue,draw=black,scale=\x1](12) at (-12,26) {};
\node [label={[label distance=\y1 cm]\z1: $003$},circle,fill=blue,draw=black,scale=\x1](13) at (-18,15) {};
\node [label={[label distance=\y1 cm]\z1: $013$},circle,fill=blue,draw=black,scale=\x1](14) at (-18,17) {};
\node [label={[label distance=\y1 cm]\z1: $023$},circle,fill=blue,draw=black,scale=\x1](15) at (-18,19) {};
\node [label={[label distance=\y1 cm]\z1: $033$},circle,fill=blue,draw=black,scale=\x1](16) at (-18,21) {};
\node [label={[label distance=\y1 cm]\z1: $004$},circle,fill=blue,draw=black,scale=\x1](17) at (-24,10) {};

\node [label={[label distance=\y1 cm]\z1: $100$},circle,fill=blue,draw=black,scale=\x1](18) at (0+\xs,30+\ys) {};
\node [label={[label distance=\y1 cm]\z1: $110$},circle,fill=blue,draw=black,scale=\x1](19) at (0+\xs,32+\ys) {};
\node [label={[label distance=\y1 cm]\z1: $120$},circle,fill=blue,draw=black,scale=\x1](20) at (0+\xs,34+\ys) {};
\node [label={[label distance=\y1 cm]\z1: $130$},circle,fill=blue,draw=black,scale=\x1](21) at (0+\xs,36+\ys) {};
\node [label={[label distance=\y1 cm]\z1: $101$},circle,fill=blue,draw=black,scale=\x1](22) at (-6+\xs,25+\ys) {};
\node [label={[label distance=\y1 cm]\z1: $111$},circle,fill=blue,draw=black,scale=\x1](23) at (-6+\xs,27+\ys) {};
\node [label={[label distance=\y1 cm]\z1: $121$},circle,fill=blue,draw=black,scale=\x1](24) at (-6+\xs,29+\ys) {};
\node [label={[label distance=\y1 cm]\z1: $131$},circle,fill=blue,draw=black,scale=\x1](25) at (-6+\xs,31+\ys) {};
\node [label={[label distance=\y1 cm]\z1: $102$},circle,fill=blue,draw=black,scale=\x1](26) at (-12+\xs,20+\ys) {};
\node [label={[label distance=\y1 cm]\z1: $112$},circle,fill=blue,draw=black,scale=\x1](27) at (-12+\xs,22+\ys) {};
\node [label={[label distance=\y1 cm]\z1: $122$},circle,fill=blue,draw=black,scale=\x1](28) at (-12+\xs,24+\ys) {};
\node [label={[label distance=\y1 cm]\z1: $132$},circle,fill=blue,draw=black,scale=\x1](29) at (-12+\xs,26+\ys) {};
\node [label={[label distance=\y1 cm]\z1: $103$},circle,fill=blue,draw=black,scale=\x1](30) at (-18+\xs,15+\ys) {};
\node [label={[label distance=\y1 cm]\z1: $113$},circle,fill=blue,draw=black,scale=\x1](31) at (-18+\xs,17+\ys) {};
\node [label={[label distance=\y1 cm]\z1: $123$},circle,fill=blue,draw=black,scale=\x1](32) at (-18+\xs,19+\ys) {};
\node [label={[label distance=\y1 cm]\z1: $133$},circle,fill=blue,draw=black,scale=\x1](33) at (-18+\xs,21+\ys) {};
\node [label={[label distance=\y1 cm]\z1: $104$},circle,fill=blue,draw=black,scale=\x1](34) at (-24+\xs,10+\ys) {};

\node [label={[label distance=\y1 cm]\z1: $200$},circle,fill=blue,draw=black,scale=\x1](35) at (0+2*\xs,30+2*\ys) {};
\node [label={[label distance=\y1 cm]\z1: $210$},circle,fill=blue,draw=black,scale=\x1](36) at (0+2*\xs,32+2*\ys) {};
\node [label={[label distance=\y1 cm]\z1: $220$},circle,fill=blue,draw=black,scale=\x1](37) at (0+2*\xs,34+2*\ys) {};
\node [label={[label distance=\y1 cm]\z1: $230$},circle,fill=blue,draw=black,scale=\x1](38) at (0+2*\xs,36+2*\ys) {};
\node [label={[label distance=\y1 cm]\z1: $201$},circle,fill=blue,draw=black,scale=\x1](39) at (-6+2*\xs,25+2*\ys) {};
\node [label={[label distance=\y1 cm]\z1: $211$},circle,fill=blue,draw=black,scale=\x1](40) at (-6+2*\xs,27+2*\ys) {};
\node [label={[label distance=\y1 cm]\z1: $221$},circle,fill=blue,draw=black,scale=\x1](41) at (-6+2*\xs,29+2*\ys) {};
\node [label={[label distance=\y1 cm]\z1: $231$},circle,fill=blue,draw=black,scale=\x1](42) at (-6+2*\xs,31+2*\ys) {};
\node [label={[label distance=\y1 cm]\z1: $202$},circle,fill=blue,draw=black,scale=\x1](43) at (-12+2*\xs,20+2*\ys) {};
\node [label={[label distance=\y1 cm]\z1: $212$},circle,fill=blue,draw=black,scale=\x1](44) at (-12+2*\xs,22+2*\ys) {};
\node [label={[label distance=\y1 cm]\z1: $222$},circle,fill=blue,draw=black,scale=\x1](45) at (-12+2*\xs,24+2*\ys) {};
\node [label={[label distance=\y1 cm]\z1: $232$},circle,fill=blue,draw=black,scale=\x1](46) at (-12+2*\xs,26+2*\ys) {};
\node [label={[label distance=\y1 cm]\z1: $203$},circle,fill=blue,draw=black,scale=\x1](47) at (-18+2*\xs,15+2*\ys) {};
\node [label={[label distance=\y1 cm]\z1: $213$},circle,fill=blue,draw=black,scale=\x1](48) at (-18+2*\xs,17+2*\ys) {};
\node [label={[label distance=\y1 cm]\z1: $223$},circle,fill=blue,draw=black,scale=\x1](49) at (-18+2*\xs,19+2*\ys) {};
\node [label={[label distance=\y1 cm]\z1: $233$},circle,fill=blue,draw=black,scale=\x1](50) at (-18+2*\xs,21+2*\ys) {};
\node [label={[label distance=\y1 cm]\z1: $204$},circle,fill=blue,draw=black,scale=\x1](51) at (-24+2*\xs,10+2*\ys) {};

\node [label={[label distance=\y1 cm]\z1: $300$},circle,fill=blue,draw=black,scale=\x1](52) at (0+3*\xs,30+3*\ys) {};
\node [label={[label distance=\y1 cm]\z1: $310$},circle,fill=blue,draw=black,scale=\x1](53) at (0+3*\xs,32+3*\ys) {};
\node [label={[label distance=\y1 cm]\z1: $320$},circle,fill=blue,draw=black,scale=\x1](54) at (0+3*\xs,34+3*\ys) {};
\node [label={[label distance=\y1 cm]\z1: $330$},circle,fill=blue,draw=black,scale=\x1](55) at (0+3*\xs,36+3*\ys) {};
\node [label={[label distance=\y1 cm]\z1: $301$},circle,fill=blue,draw=black,scale=\x1](56) at (-6+3*\xs,25+3*\ys) {};
\node [label={[label distance=\y1 cm]\z1: $311$},circle,fill=blue,draw=black,scale=\x1](57) at (-6+3*\xs,27+3*\ys) {};
\node [label={[label distance=\y1 cm]\z1: $321$},circle,fill=blue,draw=black,scale=\x1](58) at (-6+3*\xs,29+3*\ys) {};
\node [label={[label distance=\y1 cm]\z1: $331$},circle,fill=blue,draw=black,scale=\x1](59) at (-6+3*\xs,31+3*\ys) {};
\node [label={[label distance=\y1 cm]\z1: $302$},circle,fill=blue,draw=black,scale=\x1](60) at (-12+3*\xs,20+3*\ys) {};
\node [label={[label distance=\y1 cm]\z1: $312$},circle,fill=blue,draw=black,scale=\x1](61) at (-12+3*\xs,22+3*\ys) {};
\node [label={[label distance=\y1 cm]\z1: $322$},circle,fill=blue,draw=black,scale=\x1](62) at (-12+3*\xs,24+3*\ys) {};
\node [label={[label distance=\y1 cm]\z1: $332$},circle,fill=blue,draw=black,scale=\x1](63) at (-12+3*\xs,26+3*\ys) {};
\node [label={[label distance=\y1 cm]\z1: $303$},circle,fill=blue,draw=black,scale=\x1](64) at (-18+3*\xs,15+3*\ys) {};
\node [label={[label distance=\y1 cm]\z1: $313$},circle,fill=blue,draw=black,scale=\x1](65) at (-18+3*\xs,17+3*\ys) {};
\node [label={[label distance=\y1 cm]\z1: $323$},circle,fill=blue,draw=black,scale=\x1](66) at (-18+3*\xs,19+3*\ys) {};
\node [label={[label distance=\y1 cm]\z1: $333$},circle,fill=blue,draw=black,scale=\x1](67) at (-18+3*\xs,21+3*\ys) {};
\node [label={[label distance=\y1 cm]\z1: $304$},circle,fill=blue,draw=black,scale=\x1](68) at (-24+3*\xs,10+3*\ys) {};

\node [label={[label distance=\y1 cm]\z1: $040$},circle,fill=blue,draw=black,scale=\x1](69) at (0-\xs,36-\ys) {};
\node [label={[label distance=\y1 cm]\z1: $041$},circle,fill=blue,draw=black,scale=\x1](70) at (-6-\xs,31-\ys) {};
\node [label={[label distance=\y1 cm]\z1: $042$},circle,fill=blue,draw=black,scale=\x1](71) at (-12-\xs,26-\ys) {};
\node [label={[label distance=\y1 cm]\z1: $043$},circle,fill=blue,draw=black,scale=\x1](72) at (-18-\xs,21-\ys) {};

\draw [line width=\w1 mm] (1)--(2)--(3)--(4)--(69) (5)--(6)--(7)--(8)--(70) (9)--(10)--(11)--(12)--(71) (13)--(14)--(15)--(16)--(72)  (18)--(19)--(20)--(21) (22)--(23)--(24)--(25) (26)--(27)--(28)--(29) (30)--(31)--(32)--(33) (35)--(36)--(37)--(38) (39)--(40)--(41)--(42) (43)--(44)--(45)--(46) (47)--(48)--(49)--(50) (52)--(53)--(54)--(55) (56)--(57)--(58)--(59) (60)--(61)--(62)--(63)  (64)--(65)--(66)--(67)  (1)--(5)--(9)--(13)--(17)  (18)--(22)--(26)--(30)--(34) (35)--(39)--(43)--(47)--(51) (52)--(56)--(60)--(64)--(68)  (2)--(6)--(10)--(14)  (19)--(23)--(27)--(31) (36)--(40)--(44)--(48) (53)--(57)--(61)--(65) (3)--(7)--(11)--(15)  (20)--(24)--(28)--(32) (37)--(41)--(45)--(49) (54)--(58)--(62)--(66) (4)--(8)--(12)--(16)  (21)--(25)--(29)--(33) (38)--(42)--(46)--(50) (55)--(59)--(63)--(67) (1)--(18)--(35)--(52) (2)--(19)--(36)--(53) (3)--(20)--(37)--(54) (4)--(21)--(38)--(55) (5)--(22)--(39)--(56) (6)--(23)--(40)--(57) (7)--(24)--(41)--(58) (8)--(25)--(42)--(59) (9)--(26)--(43)--(60) (10)--(27)--(44)--(61) (11)--(28)--(45)--(62) (12)--(29)--(46)--(63) (13)--(30)--(47)--(64) (14)--(31)--(48)--(65) (15)--(32)--(49)--(66) (16)--(33)--(50)--(67)  (17)--(34)--(51)--(68) (69)--(70)--(71)--(72)   ;

\draw [line width={\w1+0.4 mm},red] (17)--(13)--(9)--(5)--(1)--(2)--(6)--(10)--(14)--(15)--(11)--(7)--(3)--(4) (34)--(30)--(26)--(22)--(18)--(19)--(23)--(27)--(31)--(32)--(28)--(24)--(20)--(21)--(25) (29)--(33) (51)--(47)--(43)--(39)--(35)--(36)--(40)--(44)--(48)--(49)--(45)--(41)--(37)--(38)--(42)  (46)--(50)  (68)--(64)--(60)--(56)--(52)--(53)--(57)--(61)--(65)--(66)--(62)--(58)--(54)--(55)--(59)--(63)--(67)  (17)--(34) (16)--(33) (51)--(68) (50)--(67)  (25)--(42) (29)--(46) (4)--(69)--(70)--(8)--(12)--(71)--(72)--(16); 

\end{tikzpicture}

\caption{Hamiltonian cycle in $\Pi_3^{2}$ and in $\Pi_3^{4}$.} \label{fig:Hamiltonian cycle n=3}
\end{figure}

Now we wish to prove that every metallic cube contains a Hamiltonian cycle
when $a$ is even and $n$ is odd, and contains a cycle that visits all
vertices but one when both $a$ and $n$ are even. The proofs for $n$ even
and odd are similar, so we present just a proof for odd $n$. Since
$\Pi^a_n=P_a\square\Pi^a_{n-1}\oplus\Pi^a_{n-2}$, by assumption of induction,
there is cycle in each copy $\Pi^a_{n-1}$ that visits all vertices but one.
Let $k\Pi^a_{n-1}$ denote the copy of graph $\Pi^a_{n-1}$ induced by the
vertices that start with $k$, where $0\leq k\leq a-1$.
Let $\alpha\in\mathcal{S}^a_{n-1}$ and let $k\alpha$ denote those vertices
omitted by Hamiltonian cycle in each subgraph $k\Pi^a_{n-1}$,
$0\leq k\leq a-1$. Furthermore, by assumption, there is Hamiltonian cycle
in $\Pi^a_{n-2}$. Now we choose $\beta,\gamma\in\mathcal{S}^a_{n-1}$ so
that vertices $k\alpha$  and $k\beta$ are adjacent in $k\Pi^a_{n-1}$ and
that the edge $(k\beta)(k\gamma)$ is part of Hamiltonian cycle in
$k\Pi^a_{n-1}$. For $k=0,2,4,\dots,a-2$, we remove the edge
$(k\beta)(k\gamma)$ and add edges $(k\beta)(k\alpha)$,
$(k\alpha)((k+1)\alpha)$, $((k+1)\alpha)(k+1)\beta$ and $(k\gamma)(k+1)\gamma$.
Thus we obtained Hamiltonian cycles in $k\Pi^a_{n-1}\oplus(k+1)\Pi^a_{n-1}$.
To merge those $\frac{a}{2}$ cycles and one cycle in $\Pi^a_{n-2}$ together,
we can choose any two corresponding edges that are part of Hamiltonian cycles
in the neighboring subgraphs and apply similar method to merge all cycles in
Hamiltonian cycle of graph $\Pi^a_{n}$. Choosing corresponding edges is
possible, for construction of Hamiltonian cycle in this manner is inductive,
hence, if some edge is part of Hamiltonian cycle in one copy of $\Pi^a_{n-1}$,
then it is part of Hamiltonian cycle in all copies of $\Pi^a_{n-1}$.
Similar observation holds for supgraphs $0\Pi^a_{n-1}$ and $0a\Pi^{n-2}$. 
Figure \ref{fig:Construction of Hamiltonian cycle} schematically shows the
described method of construction Hamiltonian cycles in $\Pi^a_n$.

\begin{figure} \centering \begin{tikzpicture}[scale=0.9]
\tikzmath{\a=1; \b=2; \x1 = 0.25; \y1 =-0.1; \z1=180; \w1=0.75; \xs=3; \ys=0; \yss=-1;
\x2 = \x1 + 1; \y2 =\y1 +3; } 
\small

 \draw[domain=6.25*180/20:42.5*180/20, smooth, variable=\x, red,line width={\w1}] plot ({\a*cos(\x)/2-\xs}, {2*\b*sin(\x)/3+1*\b/3});
  \draw[domain=2.5*180/20:6.25*180/20, dotted, smooth, variable=\x, red,line width={\w1}] plot ({\a*cos(\x)/2-\xs}, {2*\b*sin(\x)/3+1*\b/3});

\node [label={[label distance=\y1 cm]\z1: },circle,fill=blue,draw=black,scale=\x1](01) at ({\a*cos(2.5*180/20)/2-\xs},{\b*sin(16*180/20)}) {};
\node [label={[label distance=\y1 cm]\z1: },circle,fill=blue,draw=black,scale=\x1](02) at ({\a*cos(6.25*180/20)/2-\xs},{\b*sin(13*180/20)}) {};
\node [label={[label distance=\y1 cm]\z1: },circle,fill=blue,draw=black,scale=\x1](03) at ({\a*cos(10*180/20)/2-\xs},{\b*sin(10*180/20)*2/3+1*\b/3}) {};
\node [label={[label distance=\y1 cm]\z1: },circle,fill=blue,draw=black,scale=\x1](04) at ({\a*cos(13*180/20)/2-\xs},{\b*sin(13*180/20)*2/3+1*\b/3)}) {};
\node [label={[label distance=\y1 cm]\z1: },circle,fill=blue,draw=black,scale=\x1](05) at ({\a*cos(16*180/20)/2-\xs},{\b*sin(16*180/20)*2/3+1*\b/3}) {};
\node [label={[label distance=\y1 cm]\z1: },circle,fill=blue,draw=black,scale=\x1](06) at ({\a*cos(24*180/20)/2-\xs},{\b*sin(24*180/20)*2/3+1*\b/3}) {};
\node [label={[label distance=\y1 cm]\z1: },circle,fill=blue,draw=black,scale=\x1](07) at ({\a*cos(27*180/20)/2-\xs},{\b*sin(27*180/20)*2/3+1*\b/3}) {};
\node [label={[label distance=\y1 cm]\z1: },circle,fill=blue,draw=black,scale=\x1](08) at ({\a*cos(30*180/20)/2-\xs},{\b*sin(30*180/20)*2/3+1*\b/3}) {};
\node [label={[label distance=\y1 cm]\z1: },circle,fill=blue,draw=black,scale=\x1](09) at ({\a*cos(33*180/20)/2-\xs},{\b*sin(33*180/20)*2/3+1*\b/3}) {};
\node [label={[label distance=\y1 cm]\z1: },circle,fill=blue,draw=black,scale=\x1](10) at ({\a*cos(36*180/20)/2-\xs},{\b*sin(36*180/20)*2/3+1*\b/3}) {};

  \draw[domain=16*180/20:30*180/20, smooth, variable=\x, red,line width={\w1}] plot ({\a*cos(\x)}, {\b*sin(\x)});
    \draw[domain=33*180/20:53*180/20, smooth, variable=\x, red,line width={\w1}] plot ({\a*cos(\x)}, {\b*sin(\x)});
    
    \draw[domain=30*180/20:33*180/20,dotted, smooth, variable=\x, red,line width={\w1}] plot ({\a*cos(\x)}, {\b*sin(\x)});
    \draw[domain=13*180/20:16*180/20,dotted,smooth, variable=\x, red,line width={\w1}] plot ({\a*cos(\x)}, {\b*sin(\x)});

\node [label={[label distance=\y1 cm]\z1: },circle,fill=blue,draw=black,scale=\x1](11) at ({\a*cos(4*180/20)},{\b*sin(4*180/20)}) {};
\node [label={[label distance=\y1 cm]\z1: },circle,fill=blue,draw=black,scale=\x1](12) at ({\a*cos(7*180/20)},{\b*sin(7*180/20)}) {};
\node [label={[label distance=\y1 cm]\z1: },circle,fill=blue,draw=black,scale=\x1](13) at ({\a*cos(10*180/20)},{\b*sin(10*180/20)}) {};
\node [label={[label distance=\y1 cm]\z1: },circle,fill=blue,draw=black,scale=\x1](14) at ({\a*cos(13*180/20)},{\b*sin(13*180/20)}) {};
\node [label={[label distance=\y1 cm]\z1: },circle,fill=blue,draw=black,scale=\x1](15) at ({\a*cos(16*180/20)},{\b*sin(16*180/20)}) {};
\node [label={[label distance=\y1 cm]\z1: },circle,fill=blue,draw=black,scale=\x1](16) at ({\a*cos(24*180/20)},{\b*sin(24*180/20)}) {};
\node [label={[label distance=\y1 cm]\z1: },circle,fill=blue,draw=black,scale=\x1](17) at ({\a*cos(27*180/20)},{\b*sin(27*180/20)}) {};
\node [label={[label distance=\y1 cm]200: $0\beta$ },circle,fill=blue,draw=black,scale=\x1](18) at ({\a*cos(30*180/20)},{\b*sin(30*180/20)}) {};
\node [label={[label distance=\y1 cm]340: $0\gamma$ },circle,fill=blue,draw=black,scale=\x1](19) at ({\a*cos(33*180/20)},{\b*sin(33*180/20)}) {};
\node [label={[label distance=\y1 cm]\z1: },circle,fill=blue,draw=black,scale=\x1](20) at ({\a*cos(36*180/20)},{\b*sin(36*180/20)}) {};
\node [label={[label distance=\y1 cm]270: $0\alpha$ },circle,fill=blue,draw=black,scale=\x1](21) at ({\a*cos(30*180/20)},{\b*sin(30*180/20)+\yss}) {};

 \draw[domain=7*180/20:27*180/20, smooth, variable=\x, red,line width={\w1}] plot ({\a*cos(\x)+\xs}, {\b*sin(\x)});
    \draw[domain=30*180/20:44*180/20, smooth, variable=\x, red,line width={\w1}] plot ({\a*cos(\x)+\xs}, {\b*sin(\x)});

    \draw[domain=27*180/20:30*180/20,dotted, smooth, variable=\x, red,line width={\w1}] plot ({\a*cos(\x)+\xs}, {\b*sin(\x)});
    \draw[domain=4*180/20:7*180/20,dotted, smooth, variable=\x, red,line width={\w1}] plot ({\a*cos(\x)+\xs}, {\b*sin(\x)});
    
\node [label={[label distance=\y1 cm]\z1: },circle,fill=blue,draw=black,scale=\x1](22) at ({\a*cos(4*180/20)+\xs},{\b*sin(4*180/20)+\ys}) {};

\node [label={[label distance=\y1 cm]\z1: },circle,fill=blue,draw=black,scale=\x1](23) at ({\a*cos(7*180/20)+\xs},{\b*sin(7*180/20)+\ys}) {};
\node [label={[label distance=\y1 cm]\z1: },scale=\x1](22a) at ({\a*cos(4*180/20)+\xs+0.5},{\b*sin(4*180/20)+\ys}) {};

\node [label={[label distance=\y1 cm]\z1: },scale=\x1](23a) at ({\a*cos(4*180/20)+\xs+0.5},{\b*sin(7*180/20)+\ys}) {};

\node [label={[label distance=\y1 cm]\z1: },circle,fill=blue,draw=black,scale=\x1](24) at ({\a*cos(10*180/20)+\xs},{\b*sin(10*180/20)+\ys}) {};

\node [label={[label distance=\y1 cm]\z1: },circle,fill=blue,draw=black,scale=\x1](25) at ({\a*cos(13*180/20)+\xs},{\b*sin(13*180/20)+\ys}) {};
\node [label={[label distance=\y1 cm]\z1: },circle,fill=blue,draw=black,scale=\x1](26) at ({\a*cos(16*180/20)+\xs},{\b*sin(16*180/20)+\ys}) {};

\node [label={[label distance=\y1 cm]\z1: },circle,fill=blue,draw=black,scale=\x1](27) at ({\a*cos(24*180/20)+\xs},{\b*sin(24*180/20)+\ys}) {};

\node [label={[label distance=\y1 cm]200: $1\gamma$ },circle,fill=blue,draw=black,scale=\x1](28) at ({\a*cos(27*180/20)+\xs},{\b*sin(27*180/20)+\ys}) {};

\node [label={[label distance=\y1 cm]340: $1\beta$ },circle,fill=blue,draw=black,scale=\x1](29) at ({\a*cos(30*180/20)+\xs},{\b*sin(30*180/20)+\ys}) {};
\node [label={[label distance=\y1 cm]\z1: },circle,fill=blue,draw=black,scale=\x1](30) at ({\a*cos(33*180/20)+\xs},{\b*sin(33*180/20)+\ys}) {};
\node [label={[label distance=\y1 cm]\z1: },circle,fill=blue,draw=black,scale=\x1](31) at ({\a*cos(36*180/20)+\xs},{\b*sin(36*180/20)+\ys}) {};
\node [label={[label distance=\y1 cm]270: $1\alpha$ },circle,fill=blue,draw=black,scale=\x1](32) at ({\a*cos(30*180/20)+\xs},{\b*sin(30*180/20)+\yss}) {};

\draw[domain=16*180/20:30*180/20, smooth, variable=\x, red,line width={\w1}] plot ({\a*cos(\x)+2*\xs}, {\b*sin(\x)+\ys});
\draw[domain=33*180/20:53*180/20, smooth, variable=\x, red,line width={\w1}] plot ({\a*cos(\x)+2*\xs}, {\b*sin(\x)+\ys});

\draw[domain=30*180/20:33*180/20,dotted, smooth, variable=\x, red,line width={\w1}] plot ({\a*cos(\x)+2*\xs}, {\b*sin(\x)+\ys});
\draw[domain=13*180/20:16*180/20,dotted, smooth, variable=\x, red,line width={\w1}] plot ({\a*cos(\x)+2*\xs}, {\b*sin(\x)+\ys});

\node [label={[label distance=\y1 cm]\z1: },circle,fill=blue,draw=black,scale=\x1](44) at ({\a*cos(4*180/20)+2*\xs},{\b*sin(4*180/20)+\ys}) {};
\node [label={[label distance=\y1 cm]\z1: },circle,fill=blue,draw=black,scale=\x1](45) at ({\a*cos(7*180/20)+2*\xs},{\b*sin(7*180/20)+\ys}) {};
\node [label={[label distance=\y1 cm]\z1: },circle,fill=blue,draw=black,scale=\x1](46) at ({\a*cos(10*180/20)+2*\xs},{\b*sin(10*180/20)+\ys}) {};
\node [label={[label distance=\y1 cm]\z1: },circle,fill=blue,draw=black,scale=\x1](47) at ({\a*cos(13*180/20)+2*\xs},{\b*sin(13*180/20)+\ys}) {};
\node [label={[label distance=\y1 cm]\z1: },circle,fill=blue,draw=black,scale=\x1](48) at ({\a*cos(16*180/20)+2*\xs},{\b*sin(16*180/20)+\ys}) {};

\node [label={[label distance=\y1 cm]\z1: },scale=\x1](47a) at ({\a*cos(16*180/20)+3*\xs-0.5},{\b*sin(13*180/20)+\ys}) {};
\node [label={[label distance=\y1 cm]\z1: },scale=\x1](48a) at ({\a*cos(16*180/20)+3*\xs-0.5},{\b*sin(16*180/20)+\ys}) {};
\node [label={[label distance=\y1 cm]\z1: },circle,fill=blue,draw=black,scale=\x1](49) at ({\a*cos(24*180/20)+2*\xs},{\b*sin(24*180/20)+\ys}) {};
\node [label={[label distance=\y1 cm]\z1: },circle,fill=blue,draw=black,scale=\x1](50) at ({\a*cos(27*180/20)+2*\xs},{\b*sin(27*180/20)+\ys}) {};
\node [label={[label distance=\y1 cm]200: $2\beta$ },circle,fill=blue,draw=black,scale=\x1](51) at ({\a*cos(30*180/20)+2*\xs},{\b*sin(30*180/20)+\ys}) {};
\node [label={[label distance=\y1 cm]340: $2\gamma$ },circle,fill=blue,draw=black,scale=\x1](52) at ({\a*cos(33*180/20)+2*\xs},{\b*sin(33*180/20)+\ys}) {};
\node [label={[label distance=\y1 cm]\z1: },circle,fill=blue,draw=black,scale=\x1](53) at ({\a*cos(36*180/20)+2*\xs},{\b*sin(36*180/20)+\ys}) {};
\node [label={[label distance=\y1 cm]270: $2\alpha$ },circle,fill=blue,draw=black,scale=\x1](54) at ({\a*cos(30*180/20)+2*\xs},{\b*sin(30*180/20)+\yss}) {};

\draw[domain=30*180/20:44*180/20, smooth, variable=\x, red,line width={\w1}] plot ({\a*cos(\x)+3*\xs}, {\b*sin(\x)+\ys});
\draw[domain=7*180/20:27*180/20, smooth, variable=\x, red,line width={\w1}] plot ({\a*cos(\x)+3*\xs}, {\b*sin(\x)+\ys});

\draw[domain=27*180/20:30*180/20, dotted, smooth, variable=\x, red,line width={\w1}] plot ({\a*cos(\x)+3*\xs}, {\b*sin(\x)+\ys});
\draw[domain=4*180/20:7*180/20,dotted, smooth, variable=\x, red,line width={\w1}] plot ({\a*cos(\x)+3*\xs}, {\b*sin(\x)+\ys});

\node [label={[label distance=\y1 cm]\z1: },circle,fill=blue,draw=black,scale=\x1](55) at ({\a*cos(4*180/20)+3*\xs},{\b*sin(4*180/20)+\ys}) {};
\node [label={[label distance=\y1 cm]\z1: },scale=\x1](55a) at ({\a*cos(4*180/20)+3*\xs+0.5},{\b*sin(4*180/20)+\ys}) {};

\node [label={[label distance=\y1 cm]\z1: },circle,fill=blue,draw=black,scale=\x1](56) at ({\a*cos(7*180/20)+3*\xs},{\b*sin(7*180/20)+\ys}) {};
\node [label={[label distance=\y1 cm]\z1: },scale=\x1](56a) at ({\a*cos(4*180/20)+3*\xs+0.5},{\b*sin(7*180/20)+\ys}) {};

\node [label={[label distance=\y1 cm]\z1: },circle,fill=blue,draw=black,scale=\x1](57) at ({\a*cos(10*180/20)+3*\xs},{\b*sin(10*180/20)+\ys}) {};
\node [label={[label distance=\y1 cm]\z1: },circle,fill=blue,draw=black,scale=\x1](58) at ({\a*cos(13*180/20)+3*\xs},{\b*sin(13*180/20)+\ys}) {};
\node [label={[label distance=\y1 cm]\z1: },circle,fill=blue,draw=black,scale=\x1](59) at ({\a*cos(16*180/20)+3*\xs},{\b*sin(16*180/20)+\ys}) {};
\node [label={[label distance=\y1 cm]\z1: },circle,fill=blue,draw=black,scale=\x1](60) at ({\a*cos(24*180/20)+3*\xs},{\b*sin(24*180/20)+\ys}) {};
\node [label={[label distance=\y1 cm]200: $3\gamma$ },circle,fill=blue,draw=black,scale=\x1](61) at ({\a*cos(27*180/20)+3*\xs},{\b*sin(27*180/20)+\ys}) {};
\node [label={[label distance=\y1 cm]340: $3\beta$ },circle,fill=blue,draw=black,scale=\x1](62) at ({\a*cos(30*180/20)+3*\xs},{\b*sin(30*180/20)+\ys}) {};
\node [label={[label distance=\y1 cm]\z1: },circle,fill=blue,draw=black,scale=\x1](63) at ({\a*cos(33*180/20)+3*\xs},{\b*sin(33*180/20)+\ys}) {};
\node [label={[label distance=\y1 cm]\z1: },circle,fill=blue,draw=black,scale=\x1](64) at ({\a*cos(36*180/20)+3*\xs},{\b*sin(36*180/20)+\ys}) {};
\node [label={[label distance=\y1 cm]270: $3\alpha$ },circle,fill=blue,draw=black,scale=\x1](65) at ({\a*cos(30*180/20)+3*\xs},{\b*sin(30*180/20)+\yss}) {};

\node [label={[label distance=\y1 cm]90: $0a\Pi^a_{n-2}$},scale=\x1](a) at ({0-1*\xs},0+\b) {};
\node [label={[label distance=\y1 cm]90: $0\Pi^a_{n-1}$},scale=\x1](a) at (0,0+\b) {};
\node [label={[label distance=\y1 cm]90: $1\Pi^a_{n-1}$},scale=\x1](a) at ({0+1*\xs},0+\b) {};
\node [label={[label distance=\y1 cm]90: $2\Pi^a_{n-1}$},scale=\x1](a) at ({0+2*\xs},0+\b) {};
\node [label={[label distance=\y1 cm]90: $3\Pi^a_{n-1}$},scale=\x1](a) at ({0+3*\xs},0+\b) {};

\node [label={[label distance=\y1 cm]90: $\cdots$},scale=\x1](a) at ({(\a*cos(4*180/20)+\xs+0.5+\a*cos(16*180/20)+3*\xs-0.5)/2+5},{(\b*sin(4*180/20)+\ys+\b*sin(16*180/20)+\ys)/2}) {};


\draw [line width={\w1},red] (01)--(15) (02)--(14) (18)--(21)--(32)--(29) (19)--(28) (22)--(48) (23)--(47)  (51)--(54)--(65)--(62)  (61)--(52) (56)--(56a) (55)--(55a); 
\end{tikzpicture}  
\caption{Construction of a Hamiltonian cycle.} \label{fig:Construction of Hamiltonian cycle}
\end{figure}
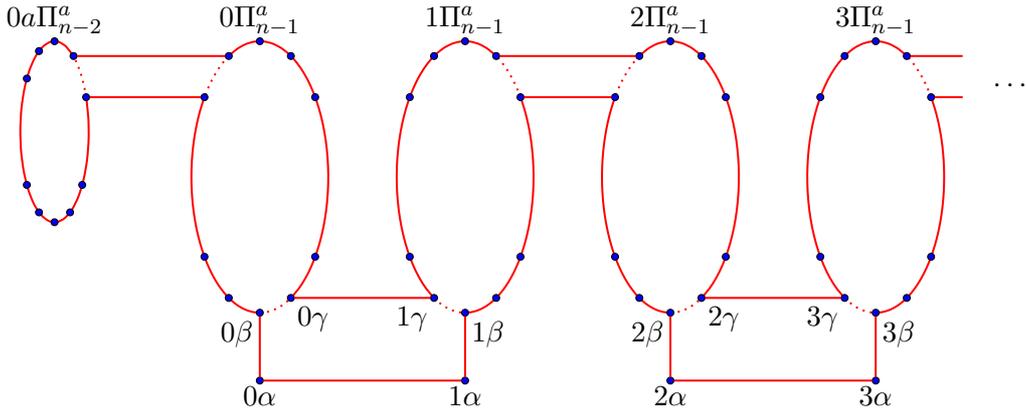
\end{proof}

\section{Concluding remarks}

In this paper we have introduced and studied a family of graphs generalizing,
on one hand, the Fibonacci and Lucas cubes, and, on the other hand, the Pell
graphs introduced recently by Munarini \cite{Munarini}. Their name, the
metallic cubes, reflects the fact that they are 
induced subgraphs of hypercubes and that their number of vertices $v_n$
satisfies two-term recurrences, $v_n = a v_{n-1} + v_{n-2}$, whose (larger)
characteristic roots are known as the metallic means. We have investigated
their basic structural, enumerative and metric properties and settled some
Hamiltonicity-related questions. Our results show that the new generalization
preserves many interesting properties of the Fibonacci, Lucas, and Pell cubes, 
indicating thusly their potential applicability in all related settings.
Our results, however, are far from presenting a comprehensive portrait of
the new family. In order to keep this contribution at a reasonable length,
we have opted for presenting just the basic results, without following all
directions they open. For example, our results on degree distribution open
the possibility to compute many degree-based topological invariants, while
our results on distances do the same for distance-based invariants such as,
e.g., the Wiener and the Mostar index \cite{klavmoll,klavmoll1,Doslic-JMC-19}.
Our results on degrees could be further developed by looking at irregularity,
in the way ref. \cite{taran} does for Pell graphs.
Similarly, apart from a short comment
about perfect and semi-perfect matchings, we have not exploited our
decomposition results to investigate any of matching-, independence- and 
dominance-related problems that have been studied for hypercubes and for
Fibonacci and Lucas cubes. Work on some of the mentioned problems is under
way, and we plan to report on it in subsequent papers.

\section*{Acknowledgments}
Partial support of Slovenian ARRS (Grant no. J1-3002) is gratefully
acknowledged by T. Do\v{s}li\'c.

\bibliographystyle{amsplain}
\bibliography{}

\end{document}